\documentclass[11pt]{article}
\usepackage[margin=1in]{geometry}
\usepackage{amsmath,amssymb,amsthm,mathtools}
\usepackage{enumitem}
\usepackage{float}
\usepackage{tikz}
\usetikzlibrary{decorations.pathreplacing}
\usepackage[colorlinks=true,linkcolor=blue,citecolor=blue,urlcolor=blue]{hyperref}
\newtheorem{theorem}{Theorem}[section]
\newtheorem{proposition}[theorem]{Proposition}
\newtheorem{lemma}[theorem]{Lemma}
\newtheorem{corollary}[theorem]{Corollary}
\newtheorem*{recalledmutation}{Corollary~\ref{cor:mutation1}}
\newtheorem*{recalledoccupation}{Corollary~\ref{cor:occupation1}}
\theoremstyle{remark}
\newtheorem*{remark}{Remark}
\newcommand{\E}{\mathbb E}
\newcommand{\Pp}{\mathbb P}
\newcommand{\one}{\mathbf 1}
\newcommand{\tcov}{\tau_{\mathrm{cov}}}
\newcommand{\dd}{\,\mathrm d}
\newcommand{\cF}{\mathcal F}
\newcommand{\cC}{\mathcal C}
\newcommand{\cU}{\mathcal U}
\newcommand{\z}{\mathbf{0}}
\title{Precise cover times for branching random walks on Hamming graphs:
logarithmic and iterated   logarithmic corrections}
\author{Zhenyuan Zhang\thanks{Laboratory for Information and Decision Systems
, Massachusetts Institute of Technology, Cambridge, MA, USA.
Email: \href{mailto:zzyzzy@mit.edu}{\texttt{zzyzzy@mit.edu}}.}}
\date{\today}
\hypersetup{
  pdftitle={Precise cover times for branching random walks on Hamming graphs:
logarithmic and iterated  logarithmic corrections},
  pdfauthor={Zhenyuan Zhang},
  pdfsubject={Precise cover times for branching random walks on Hamming graphs:
(iterated) logarithmic corrections},
  pdfkeywords={Yule process, first-passage asymptotics,
  hitting probabilities, genealogical dependence}
}

\begin{document}
\maketitle

\begin{abstract}
We prove tight asymptotics of the cover time $\tau_{\mathrm{cov}}(d)$ of a continuous-time branching random walk on
the Hamming graph $\{0,1,\dots,b-1\}^d$,  as $d\to\infty$. Particles move at rate one and branch at rate $\lambda\in(0,1]$. For $b>2$, we show that $\tau_{\mathrm{cov}}(d)=x_\star d+\lambda^{-1}\log d+O_{\mathbb P}(1)$. For $b=2$, we show that $\tau_{\mathrm{cov}}(d)=x_\star d+\chi^{-1}\log\log d+O_{\mathbb P}(1)$. Here, $x_\star$ and $\chi$ are explicit positive constants depending only on $b$ and $\lambda$. 
Our results sharpen previously known linear-order estimates that failed to identify $x_\star$ and the (iterated) logarithmic corrections. The dichotomy reflects the geometry of the last uncovered region: for $b>2$, there are exponentially many antipodes, whereas the binary hypercube has a unique antipode and its neighbors govern the final coverage. Our proofs combine classic spine change of measure techniques and many-to-few estimates with a multiscale decomposition of the genealogy and a weighted martingale analysis of the early population. 

\end{abstract}



\setcounter{tocdepth}{3}
\tableofcontents
\section{Introduction}\label{sec:introduction1}

\subsection{Model and main results}\label{sec:results1}

Fix an integer $b\ge2$. A continuous-time branching random walk on the Hamming graph $\{0,1,\ldots,b-1\}^d$ can be described as follows. Start from a single particle at the origin $\z=(0,\dots,0)$. Each alive particle independently performs continuous-time nearest-neighbor random walk with rate one. At the same time, it independently splits
into two particles at the same location with rate $\lambda$.  Write $V_t$ for the particles alive at time $t$,
$X_u(t)\in\{0,1,\ldots,b-1\}^d$ for the location of $u\in V_t$, and
$\Pp$ and $\E$ for probability and expectation when the process starts
from $\z$. 

In this paper, we focus on the regime $\lambda\in(0,1]$. For notational convenience, we let $\beta= b/(b-1)$. 
For $y\in\{0,1,\ldots,b-1\}^d$, define the \textit{first-passage time} to $y$ as
\[
 \tau_y=\inf\{t\ge0:\ X_u(t)=y\text{ for some }u\in V_t\}.
\]
The \textit{cover time} is defined as the maximum of first-passage times over all vertices:
\begin{equation}\label{eq:cover3}
 \tcov(d)=\max_{y\in\{0,1,\ldots,b-1\}^d}\tau_y .
\end{equation}

Our main goal is to study asymptotics of the cover times as $d\to\infty$. To state our main results, we need to define some preliminary constants. 
Put\footnote{This definition of $\Phi$ is motivated by the expected particle counts at an antipode vertex; see \eqref{eq:Phi} below.}
\begin{equation}\label{eq:rate1}
 \Phi(x)=\lambda x-\log b+\log(1-e^{-\beta x}).
\end{equation}
Since $\Phi'(x)>0$, $\Phi(x)\to-\infty$ as $x\downarrow0$, and
$\Phi(x)\to\infty$ as $x\to\infty$, there is a unique $x_\star>0$ with
$\Phi(x_\star)=0$.  Define
\begin{equation}\label{eq:constants1}
 \chi=\Phi'(x_\star)
 =\lambda+\frac{\beta e^{-\beta x_\star}}{1-e^{-\beta x_\star}}.
\end{equation}
We are now ready to present the main results of this paper.

\begin{theorem}\label{thm:cover1}
Let $b>2$ and $\lambda\in(0,1]$.  Then
\begin{equation}\label{eq:cover1}
 \tcov(d)=x_\star d+\lambda^{-1}\log d+O_{\Pp}(1).
\end{equation}
\end{theorem}

\begin{theorem}\label{thm:cover2}
Let $b=2$ and $\lambda\in(0,1]$.  Then
\begin{equation}\label{eq:cover2}
 \tcov(d)=x_\star d+\chi^{-1}\log\log d+O_{\Pp}(1).
\end{equation}
\end{theorem}

Our main results cover the regime $\lambda\in(0,1]$. The previous work \cite{BZ} identifies a phase transition for the first-passage times of branching random walks on the hypercube at $\lambda=1$, where for short distances (e.g., $d_{\mathrm{H}}(y,\z)=o(d)$), the first-passage time is of order $d$ if $\lambda<1$ but is of smaller order if $\lambda>1$. This corresponds to a localization--delocalization phase transition: if $\lambda<1$, mutation dominates and reaching a prescribed target requires macroscopic exploration; if $\lambda>1$, branching dominates and near targets are reached via atypically direct paths. Our theorems indicate that the critical phase $\lambda=1$ behaves in exactly the same way as $\lambda<1$ for cover times. We leave the analysis of cover times for the regime $\lambda>1$ for future work. 

We next discuss two applications of the main results. The first application concerns comparative statics: If the branching rate is fixed, how does increasing the mutation rate affect the cover time?\footnote{For all other results in this paper, the mutation rate is fixed equal to one.} Questions of this type are motivated by antiviral mutagenesis~\cite{GordonMolnupiravir,LobinskaMutagenesis}, as BRW on the Hamming cube (in random environments) resembles the evolution of viral quasispecies under mutation--selection; see \cite{BZQuenched}. It has been shown that for targets near the origin, increasing the mutation rate while fixing the branching rate may delay their first-passage events (see \cite[Corollary~3]{BZ} for the homogeneous model and \cite[Corollary~2.4]{BZQuenched} for the random-environment model). Our next result, on the contrary, shows that a larger mutation rate leads to a \textit{faster} coverage on the Hamming cube. 

\begin{corollary}\label{cor:mutation1}
Fix an integer $b\ge2$ and a branching rate $\lambda>0$. For each fixed total mutation rate $\mu\ge\lambda$, let $\tcov^{(\mu)}(d)$ denote the cover time. Then $\tcov^{(\mu)}(d)/d\xrightarrow{\Pp}a_\mu$ as $d\to\infty$, where $a_\mu>0$ is the unique solution of
\begin{equation}\label{eq:mutation1}
 \lambda a_\mu-\log b
 +\log\bigl(1-e^{-\beta\mu a_\mu}\bigr)=0,
 \qquad \beta=\frac{b}{b-1}.
\end{equation}
The function $\mu\mapsto a_\mu$ is strictly decreasing on $[\lambda,\infty)$, and $a_\mu\downarrow(\log b)/\lambda$ as $\mu\to\infty$.
\end{corollary}


Our second application concerns simultaneous occupation of the Hamming cube. By definition, the \emph{simultaneous occupation time} is the first time at which every vertex is occupied simultaneously:
\[
 \sigma_d
 =\inf\bigl\{t\ge0:
   \{X_u(t):u\in V_t\}=\{0,1,\ldots,b-1\}^d\bigr\}.
\]
Clearly, $\sigma_d\geq \tcov(d)$, while the other direction is unclear because in the delocalized regime $\lambda\in(0,1)$, the descendants of a particle just visiting a new site tend to drift away from that site.  
Our following result shows that somewhat surprisingly, $\sigma_d$ shares the same asymptotics as the cover time. Intuitively, this means that once the entire Hamming cube is visited, an extra time of $O_{\Pp}(1)$ suffices to witness the diversity of the population at a given snapshot.

\begin{corollary}\label{cor:occupation1}
Let $b\ge2$ and $\lambda\in(0,1]$. Then
\[
 \sigma_d=
 \begin{cases}
  x_\star d+\lambda^{-1}\log d+O_{\Pp}(1),&b>2,\\
  x_\star d+\chi^{-1}\log\log d+O_{\Pp}(1),&b=2.
 \end{cases}
\]
\end{corollary}

In the remainder of the introduction, we discuss motivations of studying asymptotics of cover times of BRW on the Hamming cube as well as many connections to the literature (Section \ref{sec:motivation1}). Furthermore, we provide  heuristics that lead to the main asymptotics (Section \ref{sec:targets1}) and a concise proof sketch where we also identify the technical difficulties (Section \ref{sec:proof1}).

\subsection{Motivation and related work}\label{sec:motivation1}

Our main results (Theorems~\ref{thm:cover1} and \ref{thm:cover2}) are closely connected to several streams of literature, including but not limited to, cover times of Markov chains, first-passage percolation, COBRA walks, and extrema of branching particle systems.  
In this section, we describe the motivation of our results and these connections from five aspects.

\paragraph{Cover times before mixing.}
The classic result of Aldous
\cite[Theorem~1.12; see also Section~7(A)]{AldousCube} shows that the cover time $\tau_{\rm cov}^{\rm RW}(d)$ of the simple random walk on the Hamming graph $\{0,1,\ldots,b-1\}^d$ satisfies
\[
 \frac{\tau_{\rm cov}^{\rm RW}(d)}{b^d\log(b^d)}
 \xrightarrow[d\to\infty]{L^1}1.
\]
Aldous' proof employs a coupon-collector argument after mixing. In our setting, the corresponding Hamming-graph walk has a
total-variation cutoff centered at
$\frac{b-1}{2b}\,d\log d$
with a window of order $d$ \cite[Theorem~5]{HoraHamming}.  When $b=2$,
this is the classic $(d\log d)/4+O(d)$ cutoff result of Diaconis, Graham, and
Morrison \cite{DiaconisGrahamMorrison}. 

Our results can thus be regarded as a variation of Aldous' theorem with branching.
Both works apply coupon-collecting arguments, but the techniques are different. The work \cite{AldousCube} relies on the near-independence of the walk far-enough in time by mixing, thus realizing it as an independent coupon-collecting problem. 
However, in our case, the branching population covers the hypercube in time of order $d$, well before the trajectory of a single particle has mixed. Instead, we create independence by bootstrapping the branching process, before identifying a coupon-collecting structure.

For $b=2$, Matthews proved a Gumbel limit theorem for
$\tau_{\rm cov}^{\rm RW}(d)$ \cite{MatthewsCube}. Precise cover-time
asymptotics have also been obtained for random subgraphs of the hypercube
above the connectivity threshold \cite{CooperFriezePegden}. Cover-time bounds
and speedup estimates are also known for independent parallel walks on the
hypercube and broader classes of graphs
\cite{ParallelWalks,ElsasserSauerwald,MultipleWalks}. 


\paragraph{First-passage percolation on the hypercube and the branching translation process.}
Our model is related to Durrett's branching translation process (BTP), described in Fill and Pemantle~\cite[Section~5]{FillPemantle} in the course of studying first-passage percolation (FPP) on the hypercube.\footnote{Assign iid $\mathrm{Exp}(1)$ passage times to the edges of the hypercube. FPP studies the passage time between two vertices, which by definition is the minimum sum of the total passage times over paths joining them. The set of vertices reached is also described by Richardson's growth model~\cite{Richardson}.} In BTP, we start from a single particle at $\z$, and each particle remains at its vertex and independently produces a child at each neighbouring vertex at rate one. It can thus be seen as a branching particle system whose branching and mutation events are coupled. Figure~\ref{fig:fpp_btp} compares FPP, BTP and our branching random walk.


\begin{figure}[htbp]
\centering
\begin{tikzpicture}[x=1cm,y=1cm,font=\footnotesize,line cap=round,line join=round]
  \colorlet{growthblue}{blue!55!black}
  \tikzset{edgelabel/.style={fill=white,inner sep=2pt},
    site/.style={circle,draw=black!65,fill=white,inner sep=0pt,minimum size=4.8mm},
    transition/.style={->,>=stealth,draw=growthblue,line width=1pt,shorten <=8pt,shorten >=8pt}}

  \node[font=\small\bfseries] at (1.1,3.15) {(a) FPP};
  \node at (1.1,2.75) {$\mathrm{Exp}(1)$ edge passage times};
  \draw[draw=black!45,line width=.8pt] (0,0) -- node[edgelabel,left=2pt] {$0.6$} (0,2.2)
    -- node[edgelabel,above=2pt] {$1.4$} (2.2,2.2);
  \draw[growthblue,line width=1.5pt] (0,0) -- node[edgelabel,below=2pt] {$0.4$} (2.2,0)
    -- node[edgelabel,right=2pt] {$0.7$} (2.2,2.2);
  \foreach \x/\y in {0/0,2.2/0,2.2/2.2}
    \fill[growthblue] (\x,\y) circle (2.3pt);
  \fill[black!65] (0,2.2) circle (2.3pt);
  \node[below left=3pt] at (0,0) {$\z$};
  \node[right=5pt] at (2.2,2.2) {target};
  \node at (1.1,-.90) {$\min\{0.4+0.7,\,0.6+1.4\}=1.1$};

  \begin{scope}[xshift=5.3cm]
    \node[font=\small\bfseries] at (1.1,3.15) {(b) BTP};
    \node at (1.1,2.75) {Births to neighbouring vertices};
    \draw[draw=black!30,line width=.8pt] (0,0) -- (2.2,0) -- (2.2,2.2) -- (0,2.2) -- cycle;
    \draw[transition] (0,0) -- node[edgelabel,above=2pt] {rate $2$} (2.2,0);
    \draw[transition] (0,0) -- node[edgelabel,left=2pt] {rate $2$} (0,2.2);
    \foreach \x/\y in {0/0,2.2/0,2.2/2.2,0/2.2}
      \node[site] at (\x,\y) {};
    \fill[growthblue] (-.095,0) circle (1.9pt);
    \fill[growthblue] (.095,0) circle (1.9pt);
    \fill[black!65] (2.2,0) circle (1.9pt);
    \node[below left=5pt] at (0,0) {$\z$};
    \node[align=center] at (1.1,-.95) {Parents stay; births at\\occupied vertices are allowed.};
  \end{scope}

  \begin{scope}[xshift=10.6cm]
    \node[font=\small\bfseries] at (1.1,3.15) {(c) Our BRW};
    \node at (1.1,2.75) {Jumps and local splitting};
    \draw[draw=black!30,line width=.8pt] (0,0) -- (2.2,0) -- (2.2,2.2) -- (0,2.2) -- cycle;
    \draw[transition] (0,0) -- node[edgelabel,above=2pt] {rate $1/2$} (2.2,0);
    \draw[transition] (0,0) -- node[edgelabel,left=2pt] {rate $1/2$} (0,2.2);
    \foreach \x/\y in {0/0,2.2/0,2.2/2.2,0/2.2}
      \node[site] at (\x,\y) {};
    \fill[growthblue] (0,0) circle (1.9pt);
    \node[below left=5pt] at (0,0) {$\z$};
    \draw[transition] (.15,-.90) -- node[edgelabel,above=2pt] {rate $\lambda$} (2.05,-.90);
    \node[site] at (.15,-.90) {};
    \node[site] at (2.05,-.90) {};
    \fill[growthblue] (.15,-.90) circle (1.9pt);
    \fill[growthblue] (1.955,-.90) circle (1.9pt);
    \fill[growthblue] (2.145,-.90) circle (1.9pt);
    \node at (1.1,-1.38) {same vertex};
  \end{scope}
\end{tikzpicture}
\caption{A comparison of FPP, BTP and our BRW on $\{0,1\}^2$.}
\label{fig:fpp_btp}
\end{figure}

Denote by $\mathbf{1}=(1,\dots,1)$ the antipode on the hypercube,  $T_d^{\mathrm{FPP}}$ the antipodal passage time in the FPP model, and $T_d^{\mathrm{BTP}}$ the antipodal passage time in the BTP model. By realizing FPP as a traversal process of BTP, a coupling argument yields $T_d^{\mathrm{FPP}}\geq T_d^{\mathrm{BTP}}$. Durrett employs this along with $T_d^{\mathrm{BTP}}\overset{\mathbb P}{\to}\theta:=\log(1+\sqrt{2})$ to obtain the lower bound $T_d^{\mathrm{FPP}}\ge\theta-o_{\Pp}(1)$; see \cite[Theorem~5.3]{FillPemantle}. By analyzing a suitable subprocess of BTP, Martinsson~\cite[Section~2 and Theorem~1.1]{MartinssonCube} improved this result to $T_d^{\mathrm{FPP}}=\theta+O_{\Pp}(d^{-1})$.

 For $b=2$ and $\lambda=1$, our definitions yield $\Phi(x)=\log\sinh x$, so $x_\star=\log(1+\sqrt2)=\theta$ and $\chi=\sqrt2$. Intuitively, the parameters coincide because our branching random walk at time $t$ and BTP at time $t/d$ have the same expected number of particles at each vertex.\footnote{This can be seen from the many-to-one identity~\eqref{eq:spine1} and the transition kernel~\eqref{eq:kernel1} below.}
As a consequence,\footnote{Lemmas~\ref{lem:endpoint1}(i) and~\ref{lem:cluster1}, together with the branching property and \eqref{eq:yule2}, give $\tau_{y_d}=x_\star d+O_{\Pp}(1)$ for the antipode $y_d$. The second asymptotic is Theorem~\ref{thm:cover2}.}
\begin{align}
 \frac{\tau_{\mathbf{1}}}d
 =\theta+O_{\Pp}(d^{-1})\qquad\text{and}
 \qquad
 \frac{\tcov(d)}d
 =\theta+\frac{\log\log d}{\sqrt2\,d}
       +O_{\Pp}(d^{-1}).\label{eq:derive}
       \end{align}
       We expect that our main techniques that derive \eqref{eq:derive} extend also to the BTP model and that one can show that it possesses the same asymptotic  \eqref{eq:derive} as our BRW model. 
On the other hand, there is an interesting distinction between passage and coverage in FPP: the time to reach every vertex from $\z$ is at least $\frac12\log(2+\sqrt5)+\log2-o_{\Pp}(1)$; see Fill and Pemantle~\cite[Theorem~6.4]{FillPemantle}. This lower bound is strictly larger than $\theta$, whereas the antipodal first-passage and cover times of our BRW model and BTP agree to the first order. In other words, models from branching particle systems predict successfully the FPP asymptotics on the hypercube, but not the covering time.

\paragraph{Branching particle systems on sequence spaces.}

In recent years, many studies have focused on branching particle systems on the hypercube or sequence spaces, thanks to connections to evolution processes in biology and ecology \cite{BMW,BZ,KoenigSurvey}. For instance, the branching random walk on the hypercube serves as a neutral mutation--reproduction process on
sequence space, without selection or a fitness landscape, and the cover time represents the first time all sequences are seen.

The closest result to ours is \cite[Corollary~4]{BZ}, where it is shown that $\tcov(d)\asymp d$ in probability. In other words, the cover time has linear order, but even its leading constant remains unknown \cite[Section~5.2(ii)]{BZ}.  In a similar vein, COBRA walks give another
model of branching exploration and information spread on graphs, where cover times measure the completion of dissemination of information. For COBRA walks, branching is
coupled with moving, and particles coalesce when they meet. This creates difficulty since the processes no longer have independent genealogical clusters. 

Balelli, Mili\v{s}i\'c, and Wainrib obtained bounds on partial cover times for
branching models on binary sequence spaces \cite{BMW}.\footnote{Here, the partial cover time refers to the time needed to visit a fixed positive fraction of the state space.}  For COBRA walks,
both partial and full cover times have been studied, with upper and lower bounds established 
\cite{CobraOriginal,CobraExpanders,CobraBounds}. However, we are not aware of results that provide sharp asymptotics of cover times for branching particle systems on sequence spaces.

\paragraph{First passage and extrema of branching random walks.}
Much of the classical theory of branching random walks studies the frontier on the real line, described by the largest or smallest particle in
a generation \cite{AddarioBerryReed,Aidekon}.  However, there is no ordering structure on general state spaces such as $\mathbb{R}^d$. On those spaces, first-passage times describe similar extremal behavior to one-dimensional maxima \cite{BCMZ,BZEuclideanTight}.
In this paper, we focus on the maximum of the first-passage times,
namely the cover time $\tcov(d)$ defined in \eqref{eq:cover3}. The
difficulty is thus to control the
first-passage times of all $b^d$ vertices simultaneously, leading to different techniques.

Cover times for other branching models are sparse in the literature and depend heavily on the underlying
space.  Benjamini and Kozma \cite{BenjaminiKozma} considered a branching random walk with $\Delta$ descendants at each step on the $\Delta$-regular tree, and proved that it covers
by time $n$ a complete subtree of height $n-O(\log n)$.  Roberts \cite{Roberts} later proved that the time to cover a ball of
radius $r$ on the regular tree is almost surely
$r+2\log\log r/\log(3/2)+o(\log\log r)$ (observe the iterated logarithmic correction term similar to that in Theorem \ref{thm:cover2}). For a
tree-indexed walk on a discrete torus of side length $N$ and fixed dimension
$D\ge5$,
Zhu proved concentration of the cover time at
$N^D\log(N^D)/\operatorname{BCap}(\{0\})$, where
$\operatorname{BCap}(\{0\})$ is the branching capacity of one vertex in that
model \cite{ZhuTori}.  Zhu's genealogy is a critical Galton--Watson tree
conditioned on its size, rather than the supercritical tree used
here.

\paragraph{Comparison with one-dimensional extrema.}
Sharp results for one-dimensional branching random walks show that 
the extremal particle is located at a linear distance in time with a logarithmic correction, and a na\"{i}ve first moment computation predicts incorrect asymptotics \cite{AddarioBerryReed,Zeitouni}. The logarithmic correction term is connected to modified second moment techniques and appears in related works on maximum modulus \cite{Mallein} and first-passage times \cite{BZEuclideanTight}.

On the other hand, our cover time analysis reveals a fundamentally different regime for the correction term, which will be explained in
Section~\ref{sec:targets1}. Our correction term features a dichotomy: logarithmic if $b\geq 3$ and iterated logarithmic if $b=2$. Roughly speaking, the radius of the branching random walk grows roughly linearly, and the leading linear term corresponds to the first time the antipode is hit, but this does not mean the entire hypercube has been covered. The correction term corresponds to the extra time it takes to reach \textit{all} vertices. 

\subsection{Heuristics of the asymptotics}
\label{sec:targets1}

In this section, we derive the asymptotics of \eqref{eq:cover1} and \eqref{eq:cover2} heuristically through first-moment asymptotics. We start by deriving the leading order $x_\star d$, followed by the (iterated) logarithmic correction terms.  We say two vertices on the Hamming graph are \textit{antipodal} if they differ in every coordinate, and a vertex is an \textit{antipode} if it is antipodal to the origin $\z$. The number of antipodes is exponential in $d$ if $b\geq 3$, and the antipode is unique if $b=2$.

The key idea is that once the particle system reaches a fixed antipode, it will soon cover the Hamming graph. The remaining time can be explained by vertices close to the antipodes.

\paragraph{The leading term $x_\star d$.}  Fix an antipode $y$.  Let
\[
 N_y(t)=\#\{u\in V_t:X_u(t)=y\}
\]
be the number of particles at $y$ at time $t$.  A standard first-moment computation shows that at time $xd$, its
expectation is exactly
\begin{align}
    \E[N_y(xd)]=\exp(d\Phi(x)).\label{eq:Phi}
\end{align}
Since $\Phi(x_\star)=0$, $x_\star d$ is the first-moment threshold at which the expected antipodal
occupancy is of order one.

\paragraph{The logarithmic correction term $\lambda^{-1}\log d$ for $b>2$.} 

For a fixed $a\in\mathbb R$, consider a bootstrapping procedure where we run the process until time 
$\lambda^{-1}\log d+a$, producing roughly 
$de^{\lambda a}$ many particles.  Their descendant clusters are independent, and
each has a uniformly positive probability of reaching a prescribed
target during the next $x_\star d$ units of time. Therefore, for some $c>0$, a fixed target is unhit at time $x_\star d+\lambda^{-1}\log d+a$ with probability around 
$\exp(-cde^{\lambda a})$.  Summing over the $b^d$ targets gives a non-coverage probability of around
\[
 \exp(d(\log b-ce^{\lambda a})).
\]
Therefore, changing the constant $a$ within a window of order $O(1)$ already changes the order of the heuristic coverage probability. This leads to \eqref{eq:cover1}.

\paragraph{The iterated logarithmic correction term $\chi^{-1}\log\log d$ for $b=2$.} 
Expanding $\Phi$ around $x_\star$ shows that, if $t=x_\star d+s$ and $s=o(\sqrt d)$, then
\[
 d\Phi\left(x_\star+\frac{s}{d}\right)
 =\chi s+O\left(\frac{s^2}{d}\right).
\]
Thus, for targets that are close to the antipode, the expected occupancy grows roughly like $e^{\chi s}$. In the binary hypercube, there are $d$ vertices adjacent to the antipode. At time
$x_\star d+s$, the number of arrivals at each of the $d$ vertices is approximately
Poisson with mean $ce^{\chi s}$, where $c>0$ is fixed. Therefore, the expected number of unvisited vertices in this
collection is approximately $d\exp(-ce^{\chi s}).$ Equating this with one yields
\[
 d\exp(-ce^{\chi s})\asymp 1\quad\Longleftrightarrow\quad s=\chi^{-1}\log\log d+O(1),
\]
giving \eqref{eq:cover2}.

\subsection{Proof sketch}\label{sec:proof1}

The general strategies for the upper and lower bounds are similar in the two cases $b>2$ and $b=2$. For the upper bound, we bootstrap the process until a certain time, possibly depending on the desired targets, after which we obtain independent clusters and hence independent hitting events. For the lower bound, we identify a set of potentially last-hit vertices, typically near the antipode, and give a lower bound on the probability that some of them remain unhit. 


\paragraph{The case $b>2$.}
For the upper bound, fix $C$ and condition on the process at time
$\lambda^{-1}\log d+C$. The population at this time is $\asymp d e^{\lambda C}$, whose descendants are independent.
Lemma~\ref{lem:cluster1} states that each cluster has a uniformly positive
chance of reaching any prescribed vertex during the next $x_\star d$ time. Hence, the probability that a fixed vertex is missed is at most $\exp(-ce^{\lambda C}d)$.  Taking $C$ large makes this
 $o(b^{-d})$, so the coverage probability tends to one.  The
key input of Lemma~\ref{lem:cluster1} is the second moment method, following second moment estimates in \cite{BZ}.

For the lower bound, we need to carefully construct a large collection of possible late targets. Intuitively, at time
$\lambda^{-1}\log d-B$, we examine the visited and unvisited symbols on each coordinate, and preselect a subset on the Hamming graph that contains targets far from the current range; validating this step is the key technical difficulty and requires a weighted martingale analysis (Proposition~\ref{prop:load1}). We then apply a packing argument on this set to further select an exponential subset $\mathcal C_d$ that is well-separated of order $d$ (Lemma~\ref{lem:packing1}). Lemma~\ref{lem:endpoint1}(i) then shows that
most of the targets in $\mathcal C_d$ remain unhit until the final interval of length
$\kappa\log d$, for some properly chosen $\kappa$:
\begin{itemize}
    \item $\kappa$ cannot be too small, because we want to avoid early hits;
    \item $\kappa$ cannot be too large, because we need to control the hitting probabilities in the final  $\kappa\log d$ time interval.
\end{itemize}
Next, using the well-separated property and the fact that particles rarely mutate order $d$ times in a $\kappa\log d$ time interval (Lemma~\ref{lem:truncation1}), each cluster in the final interval of length $\kappa\log d$ can only hit a single target in $\mathcal C_d$.  Elementary counting arguments (Lemma~\ref{lem:boxes1}) then yield a lower bound on the non-coverage probability.


\paragraph{The case $b=2$.}
The preceding techniques do not generalize to $b=2$, since the number of near-antipode vertices does not grow exponentially in $d$.

For the lower bound, we instead use the $d$
neighbors of the antipode. Lemma~\ref{lem:endpoint1}(i) shows
that almost all of them are still unvisited at the start of the final
$O(\log\log d)$ time interval. We then apply the second moment method to bound the number of unhit vertices in the final
$O(\log\log d)$ time interval from below. The difficulty in controlling the second moment is that these targets are only distance two
apart, so one descendant cluster can visit more than one of them (unlike the case $b>2$). 
\begin{itemize}
    \item Since $\log\log d$ is small, a crude bound suffices to control the contribution from a single cluster (Lemma~\ref{lem:cluster2}).
    \item For contributions from different clusters, we use independence and  borrow the
good-pair/bad-pair argument from Falgas-Ravry, Larsson, and Markstr\"om 
\cite[proof of Theorem~5(2)]{StructuredCoupons}.
\end{itemize}

For the upper bound, we split the targets according to the distance $\ell$ from the antipode (so of distance $d-\ell$ from the origin), and the bootstrapping procedure depends on $\ell$.
\begin{itemize}
    \item First, we consider $\ell\leq C\log\log d$ for some $C>0$. In this case, we condition on the process at time $\kappa\log\log d$ for some $\kappa>0$. Lemmas~\ref{lem:binary3} and~\ref{lem:binary4} show that the early mutations are mostly friendly (meaning that they bring particles closer to the targets), by controlling the adversarial mutations. Then, Proposition~\ref{prop:weights1} and
Lemma~\ref{lem:martingale2} quantify how each friendly move contributes to the hitting events, through Lemma~\ref{lem:binary2}(i) which provides a lower bound for hitting probabilities uniformly for near-antipodal targets. 
\item Second, for $\ell>C\log\log d$, we condition on the process at time
$\delta\ell$. By then, many
particles have made at most $\ell/2$ mutations, so are at most $d-\ell/2$ away from the target. In this case, Lemma~\ref{lem:binary2}(ii)
gives a uniformly positive lower bound on the probability of hitting
the target.
\end{itemize}
 Finally, summing over $\ell$ leads to a total probability of $o(1)$ for the non-coverage event.

\section{Notation and basic probabilistic identities}
\label{sec:identities1}\label{sec:notation1}

 We write
$A\ll B$, or $B\gg A$, when $A\le CB$, where $C$ is independent of $d$ but may depend on the
fixed parameters $(b,\lambda)$. If $A\ll B\ll A$ we write $A\asymp B$.   For
deterministic positive sequences,
$a_d=O(b_d)$ means $|a_d|\le Cb_d$ for all large $d$,
$a_d=o(b_d)$ means $a_d/b_d\to0$, and $a_d\sim b_d$ means
$a_d/b_d\to1$.  For random variables, $X_d=O_{\Pp}(b_d)$ means that
$(X_d/b_d)_d$ is tight, while $X_d=o_{\Pp}(b_d)$ means that
$X_d/b_d\to0$ in probability.  Constants denoted by $c$ and $C$ may change
from line to line.   The notation $\operatorname{Pois}(\mu)$ denotes a Poisson
random variable with mean $\mu$. 
For an event $A$, write $\one_A$ for its indicator. In the rest of this section, we collect a few fundamental facts that will be frequently used throughout this paper.

The population of the branching random walk forms a Yule process, and specifically a continuous-time
\mbox{Galton--Watson} process in which every particle splits into two at rate
$\lambda$.  Thus, if $Z_t=|V_t|$, then from state $n$ it jumps to $n+1$ at
rate $\lambda n$.  In particular,
\begin{equation}\label{eq:yule1}
 \E[Z_t]=e^{\lambda t},
 \qquad e^{-\lambda t}Z_t\longrightarrow W>0
 \quad\text{almost surely}.
\end{equation}
These standard results on Yule processes may be found in
\cite{AthreyaNey}.  In particular, the population has the geometric law
\begin{equation}\label{eq:yule2}
 \Pp(Z_t=k)=e^{-\lambda t}(1-e^{-\lambda t})^{k-1},\qquad k\ge1.
\end{equation}
Together with nonnegative-martingale convergence, this also shows that
$W$ has the mean-one exponential law and is positive almost surely.

Let $(X(t))$ denote the underlying random walk.  For
$x\in\{0,1,\ldots,b-1\}^d$, let $\Pp_x$ and $\E_x$ denote the probability and
expectation operator when the initial particle is at $x$.  For
$u\in V_t$, let $X_u^{\mathrm{anc}}(s)$ be the location at time $s$ of the
ancestor of $u$.  The \emph{many-to-one identity} is
\begin{equation}\label{eq:spine1}
 \E_x\!\left[\sum_{u\in V_t}
 F((X_u^{\mathrm{anc}}(s))_{0\le s\le t})\right]
 =e^{\lambda t}\E_x[F((X(s))_{0\le s\le t})]
\end{equation}
for every nonnegative path functional $F$.  It follows by conditioning on
the first branching time, or equivalently by checking that both sides solve
the same backward equation.  We will also use the Poisson--Chernoff bound
\cite[Theorem~5.4]{MitzenmacherUpfal}
\begin{equation}\label{eq:chernoff1}
 \Pp\!\left(\operatorname{Pois}(\mu)\ge x\right)
 \le \exp\!\left(-x\log\frac{x}{\mu}+x-\mu\right)
 \le\left(\frac{e\mu}{x}\right)^x,\qquad x\ge\mu>0.
\end{equation}
We also use the following identity.
If an event occurs on each live particle at predictable rate $r_u(s)$, and $H_u(s)$ is a nonnegative predictable mark,
then 
\begin{equation}\label{eq:event1}
 \E\!\left[\sum_{\substack{\text{events }(u,s)\\s\le t}}H_u(s)\right]
 =
 \E\!\left[\int_0^t\sum_{u\in V_s}r_u(s)H_u(s)\dd s\right].
\end{equation}


Next, we recall the spine change of measure for the Yule tree; see, for
example, \cite{LyonsPemantlePeres,HR}.  If $\mathcal T_t$ is the Yule tree, the measure
$\mathbb Q_t$ on pairs $(\mathcal T_t,U)$ is defined by
\begin{equation}\label{eq:palm1}
 \E_{\mathbb Q_t}[F(\mathcal T_t,U)]
 =e^{-\lambda t}\E\!\left[\sum_{u\in V_t}F(\mathcal T_t,u)\right].
\end{equation}
Under $\mathbb Q_t$, the tree is biased by its number $Z_t$ of leaves at time
$t$.  Moreover, conditional on the tree, $U$ is selected uniformly among the
leaves at time $t$.  We call $U$ the spine particle at time $t$ and its
ancestral line the spine.

For vertices $x,y\in\{0,1,\ldots,b-1\}^d$, write $d_{\rm H}(x,y)$ for the number of coordinates
in which they differ.  For the underlying walk, write
$P_t(x,y)=\Pp_x(X(t)=y)$.  Then by \cite[Lemma~6]{BZ},
\begin{equation}\label{eq:kernel1}
 P_t(x,y)=b^{-d}A_t^{d-d_{\rm H}(x,y)}B_t^{d_{\rm H}(x,y)},
 \qquad A_t=1+(b-1)e^{-\beta t/d},\quad B_t=1-e^{-\beta t/d}.
\end{equation}

We also use the Mecke equation for a Poisson point process.  If $\eta$ has
intensity measure $\mu$ and $F$ is nonnegative and measurable, then
\begin{equation}\label{eq:mecke1}
 \E\!\left[\sum_{p\in\eta}F(p,\eta)\right]
 =\int \E[F(p,\eta+\delta_p)]\,\mu(\mathrm d p),
\end{equation}
where $\delta_p$ denotes the unit point mass at $p$.

\section{Proof of Theorem~\ref{thm:cover1}}\label{sec:nonbinary1}

\subsection{Proof of the upper bound}\label{sec:kernel1}


For $y\in\{0,1,\ldots,b-1\}^d$, define
\[
 J_y(t)=\#\{\text{mutation jumps into }y\text{ during }(0,t]\}.
\]
This quantity records each hit on $y$ separately,
including repeated hits along one ancestral line; a branching birth
at $y$ is not counted.

\begin{lemma}\label{lem:arrival1}
Let $x,y\in\{0,1,\ldots,b-1\}^d$ and
$\ell=d-d_{\rm H}(x,y)$.  Then
\begin{equation}\label{eq:arrival1}
 \E_x[J_y(t)]=\int_0^t e^{\lambda s}b^{-d}
 A_s^\ell B_s^{d-\ell}H_{d,\ell}(s)\dd s,
\end{equation}
where
\begin{equation}\label{eq:predecessor1}
 H_{d,\ell}(s)=\frac\ell d\frac{B_s}{A_s}
 +\frac{d-\ell}{d(b-1)}\left(\frac{A_s}{B_s}+b-2\right).
\end{equation}
In particular, if $x$ and $y$ are antipodal (meaning $\ell=0$), then
\begin{equation}\label{eq:arrival2}
 \E_x[J_y(t)]=\int_0^t e^{\lambda s}b^{-d}
 \left(\frac{A_sB_s^{d-1}}{b-1}
       +\frac{b-2}{b-1}B_s^d\right)\dd s.
\end{equation}
\end{lemma}

\begin{proof}
By \eqref{eq:spine1}, at time $s$, the expected number of particles at a vertex $w$ is
$e^{\lambda s}P_s(x,w)$.  A jump into $y$ must be made from a neighbor
$w$ with $d_{\rm H}(w,y)=1$, and the mutation $w\to y$ has rate
$1/(d(b-1))$.  Therefore, applying \eqref{eq:event1} from
Section~\ref{sec:identities1} gives
\[
 \E_x[J_y(t)]
 =\int_0^t\frac{e^{\lambda s}}{d(b-1)}
 \sum_{w:\,d_{\rm H}(w,y)=1}P_s(x,w)\dd s.
\]
We evaluate the sum over the neighbors of $y$. There are two cases.
\begin{itemize}
    \item If the different coordinate of $w,y$
is one of the $\ell=d-d_{\rm H}(x,y)$ matching coordinates, there are $b-1$
choices for $w$, each with transition weight
$b^{-d}A_s^{\ell-1}B_s^{d-\ell+1}$. 
\item If it is one of the $d-\ell$
coordinates where they differ, one choice for $w$ uses the symbol of $x$
and has weight $b^{-d}A_s^{\ell+1}B_s^{d-\ell-1}$; the other $b-2$
choices have weight $b^{-d}A_s^\ell B_s^{d-\ell}$.
\end{itemize}   Using
$P_s(x,y)=b^{-d}A_s^\ell B_s^{d-\ell}$ then proves
\eqref{eq:arrival1} and \eqref{eq:predecessor1}.
Setting $\ell=0$ gives \eqref{eq:arrival2}.
\end{proof}


For the next lemma, set
\begin{align}
    q_\star=e^{-\beta x_\star},\qquad
 R=\frac{1+(b-1)q_\star}{1-q_\star},\label{eq:R}
\end{align}
so that $R>1$.

\begin{lemma}\label{lem:endpoint1}
\begin{enumerate}[label=\textup{(\roman*)}]
\item For every $C_{\rm ov},C_D<\infty$ and $\vartheta>R$, there is a
constant $C<\infty$ such that, for all large $d$, all
$|D|\le C_D\log d$, and all vertices $x,y$ satisfying
$d-d_{\rm H}(x,y)=\ell\le C_{\rm ov}\log d$,
\begin{equation*}
 \E_x[J_y(x_\star d-D)]
 \le C\vartheta^\ell e^{-\chi D}.
\end{equation*}

\item For each fixed integer $\ell\ge0$ and every $S_d=o(d^{1/2})$,
uniformly over vertices $x,y$ satisfying $d-d_{\rm H}(x,y)=\ell$,
\begin{equation}\label{eq:endpoint2}
 \E_x[J_y(x_\star d+S_d)]
 \sim
 \frac{1+(b-1)q_\star+(b-2)(1-q_\star)}
 {(b-1)(1-q_\star)\chi}e^{\chi S_d}R^\ell.
\end{equation}
\end{enumerate}
\end{lemma}

The proof is deferred to Appendix~\ref{app:endpoint1}.  It applies the many-to-one identity \eqref{eq:spine1} and adapts an integral decomposition estimate in
\cite[Lemma~22]{BZ}.

We also need the following uniform lower bound on first-passage times with arbitrary targets.

\begin{lemma}\label{lem:cluster1}
There is a constant $p_0>0$, depending only on $(b,\lambda)$, such that, for
all large $d$ and all $x,y\in\{0,1,\ldots,b-1\}^d$,
\begin{equation}\label{eq:cluster1}
 \Pp_x\!\left(\tau_y\le x_\star d\right)\ge p_0.
\end{equation}
\end{lemma}

The proof uses the exact many-to-two identity for the particle count at the
target \cite{HR}.  After dividing the pair formula by the square of the first moment,
the resulting ratio factors coordinate by coordinate.  A uniform convexity
bound then controls the second moment at every Hamming distance, and
Paley--Zygmund gives \eqref{eq:cluster1}.  The details are elaborated in
Appendix~\ref{app:cluster1}.

\begin{remark}
Lemma~\ref{lem:cluster1} can be compared to estimates of first-passage times in Blanchet and Zhang \cite[Theorem~2]{BZ}.  Their result provides tight asymptotics of the first-passage times for $d_{\rm H}(x,y)\le d/L_1$, where $L_1$
is a sufficiently large fixed constant. Lemma \ref{lem:cluster1} does not assume this condition.
\end{remark}

\begin{proof}[Proof of the upper bound in Theorem~\ref{thm:cover1}]
Let $(\cF_t)_{t\ge0}$ be the natural filtration.
We apply the independent-cluster argument of
\cite[Lemma~25 and the proof of Corollary~4]{BZ}.  
Fix $C\in\mathbb R$.  By
\eqref{eq:yule1},
$e^{-\lambda s}Z_s\to W>0$ almost surely.  Fix $w>0$.  On $\{W\ge w\}$,
almost-sure convergence gives
$Z_{\lambda^{-1}\log d+C}\ge \frac12wd\,e^{\lambda C}$ for all sufficiently
large $d$.  Dominated convergence therefore gives
\[
 \Pp\!\left(
 \left\{Z_{\lambda^{-1}\log d+C}<\frac12wd\,e^{\lambda C}\right\}
 \cap\{W\ge w\}\right)\longrightarrow0.
\]

Conditional on $\cF_{\lambda^{-1}\log d+C}$, the descendant clusters are
independent, so Lemma~\ref{lem:cluster1} gives, for every target $y$,
\[
 \Pp\!\left(\tau_y>x_\star d+\lambda^{-1}\log d+C
 \mid\cF_{\lambda^{-1}\log d+C}\right)
 \le(1-p_0)^{Z_{\lambda^{-1}\log d+C}}
 \le\exp\!\left(-p_0Z_{\lambda^{-1}\log d+C}\right).
\]
Consequently, on
$\{Z_{\lambda^{-1}\log d+C}\ge\frac12wd\,e^{\lambda C}\}$,
\[
 \Pp\!\left(\tcov(d)>x_\star d+\lambda^{-1}\log d+C
 \mid\cF_{\lambda^{-1}\log d+C}\right)
 \le\exp\!\left(d\log b-\frac12p_0wd\,e^{\lambda C}\right).
\]
Choose $C_0=C_0(w)$ so that
$p_0we^{\lambda C_0}/2>\log b$.  For every $C\ge C_0$,
\[
\begin{split}
 &\Pp\!\left(\tcov(d)>x_\star d+\lambda^{-1}\log d+C\right)\\
 &\quad\le \Pp(W<w)
 +\Pp\!\left(
 \left\{Z_{\lambda^{-1}\log d+C}<\frac12wd\,e^{\lambda C}\right\}
 \cap\{W\ge w\}\right)\\
 &\qquad\quad
 +\exp\!\left(d\log b-\frac12p_0wd\,e^{\lambda C}\right).
\end{split}
\]
The second and third terms tend to zero as $d\to\infty$.  Hence,
\[
 \limsup_{d\to\infty}
 \Pp\!\left(\tcov(d)>x_\star d+\lambda^{-1}\log d+C\right)
 \le \Pp(W<w).
\]
Finally let $w\downarrow0$; since $W>0$ almost surely,
$\Pp(W<w)\to0$. This completes the proof.
\end{proof}
\subsection{Proof of the lower bound}

\subsubsection{The lower-bound argument}\label{sec:nonbinary2}
\suppressfloats[t]

Throughout this section, we define $s_{B,d}=\lambda^{-1}\log d-B$.  We say a coordinate is \emph{saturated} if every nonzero symbol in $\{0,\dots,b-1\}$ has
appeared on that coordinate by time $s_{B,d}$.  For each coordinate $i$, we define the \textit{least-saturated
direction} $c_i$ as follows. 
\begin{itemize}
    \item If $i$ is saturated, let $c_i$ be the smallest
nonzero symbol among those carried by the fewest particles at time
$s_{B,d}$. 
\item  Otherwise, let $c_i$ be the smallest unseen
nonzero symbol.  In particular, $c_i=1$ on an untouched coordinate. 
\end{itemize}
Figure~\ref{fig:symbols1} illustrates these possible
types of coordinates and the rule for choosing the least-saturated
direction.

\begin{figure}[t]
\centering
\begin{tikzpicture}[x=1cm,y=1cm,
  every node/.style={font=\footnotesize},
  leaf/.style={circle,draw,minimum size=4.5mm,inner sep=0pt},
  chosen/.style={circle,draw=orange!80!black,fill=orange!18,
    line width=.8pt,minimum size=4.5mm,inner sep=0pt}]
  \draw[rounded corners=2pt,fill=gray!4] (0,1.15) rectangle (4.8,3.65);
  \node[font=\small\bfseries] at (2.4,3.33) {untouched};
  \node[anchor=west] at (.22,2.86) {nonzero symbols seen: $\emptyset$};
  \node[anchor=west] at (.22,2.38) {symbols at the leaves:};
  \foreach \x in {.72,1.48,2.24,3.00,3.76}
    \node[leaf] at (\x,1.91) {$0$};
  \node[anchor=north west,align=left,font=\scriptsize] at (.22,1.58)
    {\textcolor{blue!65!black}{least-saturated direction: $c_i=1$}};

  \draw[rounded corners=2pt,fill=blue!3] (5.05,1.15) rectangle (9.85,3.65);
  \node[font=\small\bfseries] at (7.45,3.33) {touched, unsaturated};
  \node[anchor=west] at (5.27,2.86) {nonzero symbols seen: $\{1,2\}$};
  \node[anchor=west] at (5.27,2.38) {symbols at the leaves:};
  \node[leaf] at (5.77,1.91) {$0$};
  \node[leaf] at (6.53,1.91) {$1$};
  \node[leaf] at (7.29,1.91) {$2$};
  \node[leaf] at (8.05,1.91) {$1$};
  \node[leaf] at (8.81,1.91) {$0$};
  \node[anchor=north west,align=left,font=\scriptsize] at (5.27,1.58)
    {\textcolor{blue!65!black}{least-saturated direction: $c_i=3$}};

  \draw[rounded corners=2pt,fill=orange!4] (10.10,1.15) rectangle (15.25,3.65);
  \node[font=\small\bfseries] at (12.675,3.33) {saturated};
  \node[anchor=west] at (10.32,2.86) {nonzero symbols seen: $\{1,2,3\}$};
  \node[anchor=west] at (10.32,2.38) {symbols at the leaves:};
  \node[leaf] at (10.82,1.91) {$1$};
  \node[leaf] at (11.58,1.91) {$1$};
  \node[chosen] at (12.34,1.91) {$2$};
  \node[leaf] at (13.10,1.91) {$3$};
  \node[leaf] at (13.86,1.91) {$3$};
  \node[anchor=north west,align=left,font=\scriptsize] at (10.32,1.58)
    {\textcolor{blue!65!black}{least-saturated direction: $c_i=2$}};
\end{tikzpicture}
\caption{The coordinate classification, illustrated for $b=4$. }
\label{fig:symbols1}
\end{figure}
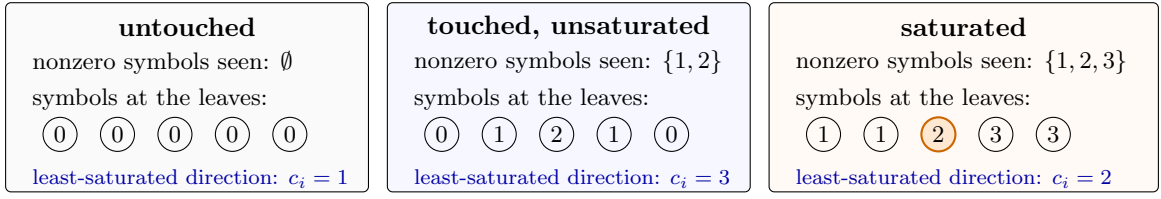

 For $u\in V_{s_{B,d}}$, let $L_u$ be the number of saturated coordinates
on which $u$ carries the least-saturated direction $c_i$, i.e.,
$X_u(s_{B,d})(i)=c_i$.
 The next proposition bounds the sum of the exponential weights
$\vartheta^{L_u}$ over the particles alive at time $s_{B,d}$.

\begin{proposition}\label{prop:load1}
Assume $b>2$.  For every fixed $B\ge0$ and every fixed
$1<\vartheta<1+\lambda$, the following estimate holds for all sufficiently
large $d$:
\begin{equation*}
 \E\!\left[\sum_{u\in V_{s_{B,d}}}\vartheta^{L_u}\right]
 \ll_{B,\vartheta} d.
\end{equation*}
\end{proposition}


\begin{lemma}\label{lem:ratio1}
If $b>2$, then $R-1<\lambda$.
\end{lemma}

The proof of Lemma~\ref{lem:ratio1} is deferred to
Appendix~\ref{app:ratio1}, and the proof of
Proposition~\ref{prop:load1} is postponed to
Section~\ref{sec:weights1}.  By Lemma~\ref{lem:ratio1}, Proposition~\ref{prop:load1}  applies to every fixed $\vartheta>R$ sufficiently close to $R$.
Before applying it to the lower bound in Theorem~\ref{thm:cover1}, we
extract well-separated
targets that the early population has
not visited.  A set
$\mathcal C$ in the Hamming graph is called \textit{$r$-separated} if
$d_{\rm H}(y,z)\ge r$ for every distinct $y,z\in\mathcal C$.

\begin{lemma}\label{lem:packing1}
For every $\epsilon>0$, there is $B_0<\infty$ such that, for every fixed
$B\ge B_0$, there are constants $h,\delta,C_{\rm path}>0$ and an
$\cF_{s_{B,d}}$-measurable set $\cC_d$ such that, with probability at least
$1-\epsilon$ for all sufficiently large $d$:
\begin{enumerate}[label=(\roman*)]
\item $|\cC_d|\ge e^{hd}$ and $\cC_d$ is $\delta d$-separated;
\item no point of $\cC_d$ was hit before $s_{B,d}$;
\item for every
$u\in V_{s_{B,d}}$ and $y\in\cC_d$,
\[
 d-d_{\rm H}(X_u(s_{B,d}),y)=L_u\le C_{\rm path}\log d.
\]
\end{enumerate}
\end{lemma}

\begin{proof}
It suffices to consider $0<\epsilon<1/2$; for larger $\epsilon$, apply the
result with $\epsilon=1/4$. We split the proof into two steps.
\par\smallskip\noindent\emph{Step 1: mutation counts.}
The expected total number of mutations before $s_{B,d}$ is
$\int_0^{s_{B,d}} e^{\lambda r}\dd r\ll d\,e^{-\lambda B}$.
If more than $\epsilon d$ coordinates were touched, there would be more
than $\epsilon d$ mutations.  Hence,
\[
 \Pp\!\left(\#\{\text{touched coordinates}\}>\epsilon d\right)
 \ll \frac{Ce^{-\lambda B}}{\epsilon}.
\]
Choose $B_0$ so that this is at most $\epsilon/2$ for $B\geq B_0$.  By the many-to-one formula \eqref{eq:spine1},
\begin{align*}
 \E\!\left[\#\{u\in V_{s_{B,d}}:\text{the line to }u\text{ has }
 \ge C_{\rm path}\log d\text{ mutations}\}\right]
 &=e^{\lambda s_{B,d}}\Pp\!\left(\operatorname{Pois}(s_{B,d})
 \ge C_{\rm path}\log d\right)\\
 &\le d\left(\frac{e}{\lambda C_{\rm path}}\right)^{C_{\rm path}\log d}
 =o(1)
\end{align*}
for a sufficiently large fixed $C_{\rm path}$, where we used
$s_{B,d}\le\lambda^{-1}\log d$ and the bound \eqref{eq:chernoff1}.  Thus, outside an
event of probability $o(1)$, every ancestral line ending at time $s_{B,d}$ has
at most $C_{\rm path}\log d$ mutations.

\par\smallskip\noindent\emph{Step 2: a packing procedure.}
  On the event from
Step~1, there are at least $(1-\epsilon)d$ untouched coordinates, each
contributing $b-1$ independent choices.  Hence, we have an untouched set with at least
$(b-1)^{(1-\epsilon)d}$ elements, by choosing all directions in the untouched coordinates and the least-saturated directions in the touched ones. 
Lemma~\ref{lem:packing2} below then yields constants $h,\delta>0$
and a $\delta d$-separated subset $\cC_d$ with at least $e^{hd}$ points.
 Every matching coordinate between $X_u(s_{B,d}),y$ requires a
mutation on the corresponding ancestral line, so Step~1 gives
$L_u\le C_{\rm path}\log d$ simultaneously for all $u$.
\end{proof}


\begin{lemma}\label{lem:truncation1}
Fix $\delta,\kappa,C_0>0$ and let $0\leq r_d \le C_0d-\kappa\log d$.  From every particle alive at $r_d$, follow
its descendants for time $\kappa\log d$ and kill each  lineage upon its
$\lfloor\delta d/3\rfloor$th mutation.  For every nonempty, possibly
$\cF_{r_d}$-measurable, $\delta d$-separated set $\cC_d$, the following statements hold:
\begin{enumerate}[label=\textup{(\roman*)}]
\item each truncated descendant cluster hits at most one point of $\cC_d$;
\item $\Pp(\text{truncation deletes a visit to }\cC_d)=o(1)$.
\end{enumerate}
\end{lemma}

\begin{proof}
(i) This follows immediately from the triangle inequality.

(ii) We bound the deletion probability by the probability that any lineage reaches the threshold. By the many-to-one formula \eqref{eq:spine1}, the latter probability is at most 
\begin{equation}\label{eq:truncation1}
 e^{\lambda(r_d+\kappa\log d)}
 \Pp\!\left(\operatorname{Pois}(\kappa\log d)
 \ge\lfloor\delta d/3\rfloor\right)
 \le \exp(Cd-cd\log d)=o(1).
\end{equation}
Indeed, for all large $d$, $\lfloor\delta d/3\rfloor>\kappa\log d$, and
\eqref{eq:chernoff1} gives the exponent in
\eqref{eq:truncation1}.
\end{proof}

Finally, we need the following elementary occupancy bound.
It is a finite, non-identically distributed version of the second-moment
coupon calculation in \cite[Lemma~1(2)]{StructuredCoupons}: the condition
that each coupon contains at most one target makes the indicators of two
empty boxes negatively correlated.  The proof is deferred to Appendix~\ref{app:boxes1}.

\begin{lemma}\label{lem:boxes1}
Let $\mathcal U$ be a nonempty finite set and let
$(S_v)_v$ be a finite collection of independent random subsets
of $\mathcal U$, each containing at most one point.  Suppose that
$\Pp(y\in S_v)\le p$ for every $v,y$, where $0\le p<1$.  Then, for every
$y\in\mathcal U$,
\begin{equation}\label{eq:boxes1}
 \Pp\!\left(y\notin\bigcup_vS_v\right)
 \ge \exp\!\left(-\frac{\sum_v\Pp(y\in S_v)}{1-p}\right),
\end{equation}
and
\begin{equation}\label{eq:boxes2}
 \Pp\!\left(\bigcup_vS_v=\mathcal U\right)
 \le\left(
 \sum_{y\in\mathcal U}
 \exp\!\left(-\frac{\sum_v\Pp(y\in S_v)}{1-p}\right)
 \right)^{-1}.
\end{equation}
\end{lemma}

\begin{proof}[Proof of the lower bound in Theorem~\ref{thm:cover1}] We split the proof into four steps. 

\par\smallskip\noindent\emph{Step 1: setting up stages.}
Fix $0<\eta<1/4$.  Choose $B$ large enough that the event in
Lemma~\ref{lem:packing1} has probability at least $1-\eta$.  By
Proposition~\ref{prop:load1}, Lemma~\ref{lem:ratio1}, and
Markov's inequality, fix $\vartheta\in(R,1+\lambda)$ and choose
$M<\infty$ so that, outside another event of probability at most $\eta$,
\begin{equation}\label{eq:load2}
 \sum_{u\in V_{s_{B,d}}}\vartheta^{L_u}\le Md.
\end{equation}

For $A>B$, put
\[
 T_{A,d}=x_\star d+\lambda^{-1}\log d-A.
\]
Our goal is to prove
\begin{equation*}
 \lim_{A\to\infty}\limsup_{d\to\infty}
 \Pp\!\left(\tcov(d)\le T_{A,d}\right)=0.
\end{equation*}
Also pick
\begin{equation}\label{eq:window1}
 \frac1\chi<\kappa<\frac1\lambda.
\end{equation}
Let
\begin{equation*}
 r_{A,d}=T_{A,d}-\kappa\log d.
\end{equation*}
We split the post-$s_{B,d}$ evolution of the process at $r_{A,d}$, leaving a terminal
interval of length $\kappa\log d$.

\begin{figure}[ht]
\centering
\begin{tikzpicture}[x=.96\linewidth,y=1cm,>=stealth,
  every node/.style={font=\footnotesize,inner sep=1pt}]
  \draw[->] (.02,0)--(.99,0);
  \foreach \x in {.02,.16,.61,.79,.96}
    \draw (\x,-.07)--(\x,.07);
  \node[below=2pt] at (.02,0) {$0$};
  \node[above=2pt] at (.16,0) {$s_{B,d}$};
  \node[above=2pt,text=gray] at (.61,0) {$x_\star d$};
  \node[above=2pt] at (.79,0) {$r_{A,d}$};
  \node[above=2pt] at (.96,0) {$T_{A,d}$};
  \draw[decorate,decoration={brace,mirror,amplitude=3pt}]
    (.16,-.30)--(.79,-.30)
    node[midway,below=4pt]
    {$r_{A,d}-s_{B,d}=x_\star d-(A-B)-\kappa\log d$};
  \draw[decorate,decoration={brace,mirror,amplitude=3pt}]
    (.79,-.30)--(.96,-.30)
    node[midway,below=4pt]
    {$T_{A,d}-r_{A,d}=\kappa\log d$};
\end{tikzpicture}
\caption{The key time points considered here.}
\label{fig:nonbinary1}
\end{figure}

\par\smallskip\noindent\emph{Step 2: the interval $[s_{B,d}, r_{A,d}]$.}
We condition on $\cF_{s_{B,d}}$ and restrict to the event on which
\eqref{eq:load2} and the conclusions of
Lemma~\ref{lem:packing1} hold.   For $y\in\cC_d$, the branching property,
Lemma~\ref{lem:packing1}(iii), Lemma~\ref{lem:endpoint1}(i), and
\eqref{eq:load2} give
\begin{equation}\label{eq:arrivals1}
\begin{split}
 &\E\!\left[\#\{\text{jumps into }y\text{ during }
 (s_{B,d},r_{A,d}]\}\,\middle|\,\cF_{s_{B,d}}\right]\\
 &\quad=\sum_{u\in V_{s_{B,d}}}
 \E_{X_u(s_{B,d})}\!\left[
 J_y\!\left(x_\star d-(A-B)-\kappa\log d\right)\right]\\
 &\quad\ll e^{-\chi(A-B+\kappa\log d)}
 \sum_{u\in V_{s_{B,d}}}\vartheta^{L_u}\\
 &\quad\le Md\,e^{-\chi(A-B)}d^{-\chi\kappa}
 =o(1).
\end{split}
\end{equation}
Since every target in $\cC_d$
was unvisited at time $s_{B,d}$, define
\begin{equation*}
 \cU_d=\{y\in\cC_d:\tau_y>r_{A,d}\}.
\end{equation*}
Summing \eqref{eq:arrivals1} over $y\in\cC_d$ and applying
conditional Markov's inequality shows that
\begin{equation*}
 \Pp\!\left(
 |\cC_d\setminus\cU_d|>\frac{|\cC_d|}{4}
 \,\middle|\,\cF_{s_{B,d}}\right)=o(1).
\end{equation*}
Thus $|\cU_d|\ge3|\cC_d|/4$ with conditional probability $1-o(1)$.

\par\smallskip\noindent\emph{Step 3: the interval $[r_{A,d},T_{A,d}]$.}
Condition now on $\cF_{r_{A,d}}$.  For each fixed $A$, all
large $d$ satisfy $r_{A,d}\ge0$ and
$r_{A,d}+\kappa\log d=T_{A,d}\le(x_\star+1)d$. Moreover, on the event from Step~2,
$\cU_d$ is nonempty.  Thus, Lemma~\ref{lem:truncation1} applies with
$C_0=x_\star+1$ and the separation constant $\delta$ given by 
Lemma~\ref{lem:packing1}, and we may without loss of generality kill each lineage upon its
$\lfloor\delta d/3\rfloor$th mutation.  In addition, under this truncation, each cluster
hits at most one point of $\cU_d$ by Lemma~\ref{lem:truncation1}.  For each $v\in V_{r_{A,d}}$, let
\[
 S_v=\{y\in\cU_d:\text{the truncated descendants of $v$ hit $y$
 during }[r_{A,d},T_{A,d}]\}.
\]
Conditional on $\cF_{r_{A,d}}$, these sets are independent and $|S_v|\le1$.
Since $X_v(r_{A,d})\ne y$ for every $v\in V_{r_{A,d}}$ and $y\in\cU_d$,
the incoming jump rate to a fixed vertex and the
many-to-one identity \eqref{eq:spine1} give
\begin{equation}\label{eq:probability1}
 \max_{\substack{v\in V_{r_{A,d}}\\y\in\cU_d}}
 \Pp(y\in S_v\mid\cF_{r_{A,d}})
 \le \frac1{d(b-1)}\int_0^{\kappa\log d}e^{\lambda u}\dd u
 \ll d^{\lambda\kappa-1}=o(1),
\end{equation}
where we have used \eqref{eq:window1}.

Note that every hit of $y\in\cU_d$ contains an incoming jump into $y$ by
the cluster.  Since $\cC_d$ is
$\cF_{s_{B,d}}$-measurable, the tower property, Lemma~\ref{lem:packing1}(iii), Lemma~\ref{lem:endpoint1}(i), and \eqref{eq:load2}  give
\begin{equation*}
\begin{split}
 &\E\left[\sum_{y\in\cU_d}\sum_{v\in V_{r_{A,d}}}
 \Pp(y\in S_v\mid\cF_{r_{A,d}})
       \,\middle|\,\cF_{s_{B,d}}\right]
 \\
 &\quad\le\sum_{y\in\cC_d}\sum_{u\in V_{s_{B,d}}}
 \E_{X_u(s_{B,d})}\!\left[J_y\!\left(x_\star d-(A-B)\right)\right]
 \\
 &\quad\ll e^{-\chi(A-B)}
 \sum_{y\in\cC_d}\sum_{u\in V_{s_{B,d}}}\vartheta^{L_u}
 \\
 &\quad\le M|\cC_d|d\,e^{-\chi(A-B)}.
\end{split}
\end{equation*}
Consequently, except on an event of conditional probability at most
$Me^{-\chi(A-B)/2}$,
\begin{equation}\label{eq:terminal2}
 \sum_{y\in\cU_d}\sum_{v\in V_{r_{A,d}}}
 \Pp(y\in S_v\mid\cF_{r_{A,d}})
 \le |\cC_d|d e^{-\chi(A-B)/2}.
\end{equation}
On the event $|\cU_d|\ge3|\cC_d|/4$ obtained in Step~2, define the
$\cF_{r_{A,d}}$-measurable set
$\mathcal L_d:=\{y\in\cU_d:
\sum_{v\in V_{r_{A,d}}}
\Pp(y\in S_v\mid\cF_{r_{A,d}})
\le4d\,e^{-\chi(A-B)/2}\}$.
Equation~\eqref{eq:terminal2} shows that at most
$|\cC_d|/4$ points of $\cU_d$ lie outside $\mathcal L_d$, and hence
$|\mathcal L_d|\ge|\cC_d|/2$.

\par\smallskip\noindent\emph{Step 4: targets with small total hitting probability.}
Conditional on $\cF_{r_{A,d}}$, the random sets $S_v\cap\mathcal L_d$ are
independent and each contains  at most one point.  By 
Lemma~\ref{lem:boxes1} and  \eqref{eq:probability1}, for $d$ large enough, the
conditional probability that every target in $\mathcal L_d$ is hit is at
most
\begin{equation}\label{eq:coverage1}
\begin{split}
 &\left(
 \sum_{y\in\mathcal L_d}
 \exp\!\left(
 -\frac{\sum_v\Pp(y\in S_v\mid\cF_{r_{A,d}})}
 {1-\max_{\substack{v\in V_{r_{A,d}}\\z\in\cU_d}}
 \Pp(z\in S_v\mid\cF_{r_{A,d}})}
 \right)\right)^{-1}\\
 &\quad\le\frac{2}{|\cC_d|}
 \exp\!\left(5d\,e^{-\chi(A-B)/2}\right).
\end{split}
\end{equation}
Choose $A-B$ large enough that $5e^{-\chi(A-B)/2}<h/2$.  By Lemma~\ref{lem:packing1}(i), 
$|\cC_d|\ge e^{hd}$, so 
the right-hand side of \eqref{eq:coverage1} is at most
$e^{-hd/3}$ for all large $d$.

Combining the preceding estimates, we get
\[
 \Pp\!\left(\tcov(d)\le T_{A,d}\right)
 \le2\eta+Me^{-\chi(A-B)/2}+o(1).
\]
Therefore,
\[
 \limsup_{d\to\infty}\Pp\!\left(\tcov(d)\le T_{A,d}\right)
 \le 2\eta+Me^{-\chi(A-B)/2}.
\]
Finally, let $A\to\infty$ with $B,\eta$ fixed, and then let $\eta\downarrow0$.
This proves the lower-tail assertion in Theorem~\ref{thm:cover1}.
\end{proof}
\subsubsection{Proof of Proposition~\ref{prop:load1}}\label{sec:weights1}

Fix $B\ge0$ and take $d$ large enough that $s:=s_{B,d}\ge0$.  We first reduce
the sum over all particles to the spine particle using the change of measure
in \eqref{eq:palm1}. Recall that  for $u\in V_{s_{B,d}}$, $L_u$ is the number of saturated
coordinates $i$ for which $X_u(s_{B,d})(i)=c_i$, where $c_i$ is the least-saturated direction.

\begin{lemma}\label{lem:load1}
Under $\mathbb Q_s$, let $U$ be the spine particle at time $s$, and let
$q_U(\mathcal T_s)$ be the conditional probability that one coordinate is
saturated by time $s$ and $U$ carries its least-saturated direction.  Then, for
every $\vartheta>1$,
\begin{equation}\label{eq:load3}
 \E\!\left[\sum_{u\in V_s}\vartheta^{L_u}\right]
 \le d\,e^{-\lambda B}
 \E_{\mathbb Q_s}\!\left[
 e^{(\vartheta-1)dq_U(\mathcal T_s)}\right].
\end{equation}
\end{lemma}

\begin{proof}
Condition on the Yule tree $\mathcal T_s$.  Applying the
spine change of measure \eqref{eq:palm1} and independence across coordinates gives
\begin{equation*}
 \E\!\left[\sum_u\vartheta^{L_u}\right]
 =e^{\lambda s}\E_{\mathbb Q_s}\!\left[
 (1+(\vartheta-1)q_U(\mathcal T_s))^d\right]\le d\,e^{-\lambda B}\E_{\mathbb Q_s}\!\left[
 e^{(\vartheta-1)dq_U(\mathcal T_s)}\right].
\end{equation*}
This proves \eqref{eq:load3}.
\end{proof}


\paragraph{Graphical construction.}
For a Yule tree $\mathcal T_s$, let us identify each genealogical edge with the time interval
between its endpoints, and let $\ell_{\mathcal T_s}$ be the resulting length
measure.
For each $j\in\{0,1,\ldots,b-1\}$, independently place a Poisson point
process $\mathcal P_j$ on the tree with intensity
$\ell_{\mathcal T_s}/(d(b-1))$, corresponding to the mutation events.
Call each point of $\mathcal P_j$ a $j$-mark.   For a point $p$ of the tree, let $n_s(p)$ be the number of leaves at time $s$ that  descend from $p$.  

Imagine propagating the symbol $0$ from the root
down the tree.  At a $j$-mark, replace the current symbol by $j$ if it differs from $j$ (we call it a \textit{retained mark});
otherwise ignore the point (we call it an \textit{ignored mark}).  Along each edge, the marks are
applied in chronological order, and at a split both children inherit the
symbol immediately before the split. 
See Figure~\ref{fig:tree1} for an illustration.

\begin{figure}[ht]
\centering
\begin{tikzpicture}[x=1cm,y=1cm,>=stealth,
  every node/.style={font=\footnotesize},
  symzero/.style={draw=gray!75!black,line width=1.15pt},
  symone/.style={draw=blue!70!black,line width=1.15pt},
  symtwo/.style={draw=orange!85!black,line width=1.15pt}]
  \draw[->,gray!75] (.45,2.55)--(10.55,2.55)
    node[above left,text=black] {time};
  \node[anchor=north] at (.45,2.55) {$0$};
  \node[anchor=north] at (10.25,2.55) {$s$};

  \coordinate (root) at (.45,0);
  \coordinate (a) at (1.55,0);
  \coordinate (splitone) at (3.00,0);
  \coordinate (ignored) at (4.05,.43);
  \coordinate (p) at (5.20,.91);
  \coordinate (splittwo) at (6.35,1.38);
  \coordinate (upperone) at (10.25,1.92);
  \coordinate (uppertwo) at (10.25,.72);
  \coordinate (b) at (4.35,-.62);
  \coordinate (splitthree) at (6.35,-1.27);
  \coordinate (lowerone) at (10.25,-.72);
  \coordinate (c) at (8.25,-1.61);
  \coordinate (lowertwo) at (10.25,-1.92);

  \draw[symzero] (root)--(a);
  \draw[symone] (a)--(splitone)--(ignored)--(p);
  \draw[symtwo] (p)--(splittwo)--(upperone);
  \draw[symtwo] (splittwo)--(uppertwo);
  \draw[symone] (splitone)--(b);
  \draw[symzero] (b)--(splitthree)--(lowerone);
  \draw[symzero] (splitthree)--(c);
  \draw[symtwo] (c)--(lowertwo);

  \fill (root) circle (1.7pt);
  \fill (splitone) circle (1.7pt);
  \fill (splittwo) circle (1.7pt);
  \fill (splitthree) circle (1.7pt);
  \node[anchor=east,align=right] at (.33,0) {$\emptyset$\\symbol $0$};
  \node[circle,draw=orange!85!black,fill=white,minimum size=5mm,inner sep=0pt]
    at (upperone) {$2$};
  \node[circle,draw=orange!85!black,fill=white,minimum size=5mm,inner sep=0pt]
    at (uppertwo) {$2$};
  \node[circle,draw=gray!75!black,fill=white,minimum size=5mm,inner sep=0pt]
    at (lowerone) {$0$};
  \node[circle,draw=orange!85!black,fill=white,minimum size=5mm,inner sep=0pt]
    at (lowertwo) {$2$};

  \fill[blue!70!black] (a) circle (2.4pt);
  \node[above=3pt,text=blue!70!black] at (a)
    {$\mathcal P_1:0\to1$};

  \draw[blue!70!black,fill=white,line width=.8pt] (ignored) circle (2.8pt);
  \draw[blue!70!black,line width=.7pt]
    ([xshift=-2pt,yshift=-2pt]ignored)--([xshift=2pt,yshift=2pt]ignored);
  \draw[blue!70!black,line width=.7pt]
    ([xshift=-2pt,yshift=2pt]ignored)--([xshift=2pt,yshift=-2pt]ignored);
  \node[anchor=east,align=right,text=blue!70!black] at (3.88,.92)
    {$\mathcal P_1$\\ignored};

  \fill[orange!85!black] (p) circle (2.4pt);
  \node[anchor=south,text=orange!85!black] at (5.15,1.48)
    {$p\in\mathcal P_2:1\to2$};

  \fill[gray!75!black] (b) circle (2.4pt);
  \node[below=3pt,text=gray!75!black] at (b)
    {$\mathcal P_0:1\to0$};

  \fill[orange!85!black] (c) circle (2.4pt);
  \node[below=3pt,text=orange!85!black] at (c)
    {$\mathcal P_2:0\to2$};

  \draw[decorate,decoration={brace,mirror,amplitude=4pt},orange!85!black]
    (10.62,.72)--(10.62,1.92)
    node[midway,right=6pt,align=left,text=black]
    {$n_s(p)=2$\\descendant leaves};

  \node[anchor=west,font=\scriptsize] at (.45,-2.88) {edge symbols:};
  \draw[symzero] (2.65,-2.88)--(3.10,-2.88);
  \node[anchor=west] at (3.20,-2.88) {$0$};
  \draw[symone] (3.75,-2.88)--(4.20,-2.88);
  \node[anchor=west] at (4.30,-2.88) {$1$};
  \draw[symtwo] (4.85,-2.88)--(5.30,-2.88);
  \node[anchor=west] at (5.40,-2.88) {$2$};
  \fill[black] (6.65,-2.88) circle (2.4pt);
  \node[anchor=west] at (6.80,-2.88) {retained mark};
  \draw[black,fill=white,line width=.7pt] (9.25,-2.88) circle (2.8pt);
  \draw[line width=.65pt]
    (9.18,-2.95)--(9.32,-2.81) (9.18,-2.81)--(9.32,-2.95);
  \node[anchor=west] at (9.42,-2.88) {ignored mark};
\end{tikzpicture}
\caption{The graphical construction for a single coordinate. The second $\mathcal P_1$ is ignored because the current symbol is already
$1$ before that event. }
\label{fig:tree1}
\end{figure}
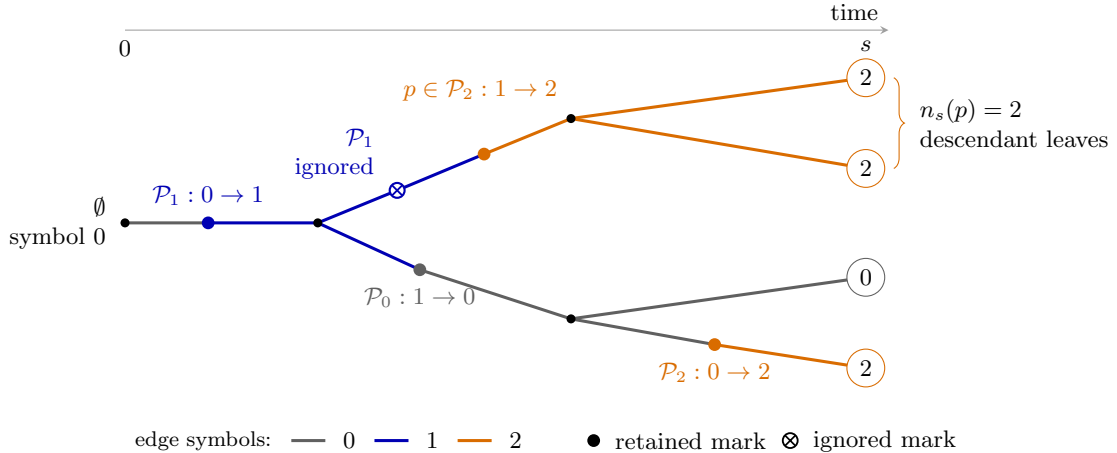
Next, we collect a few simple facts regarding the count $n_s(\cdot)$. 
For each $j$, the sum $\sum_{p\in\mathcal P_j}n_s(p)$ counts every descendant below every $j$-mark.  Conditional on
$\mathcal T_s$, the sums are independent and identically distributed for different $j$.
Write $M^{\rm dom}$ for a random variable with their common conditional law,
and put
\begin{equation}\label{eq:dom}
 \Pi_{\mathcal T_s}(k)
 :=\Pp(M^{\rm dom}\ge k\mid\mathcal T_s),\qquad k\ge1.
\end{equation}
For every $j\ne0$, the number of  leaves at time $s$ carrying symbol $j$ therefore satisfies
\begin{equation}\label{eq:terminal3}
 \#\{w\in V_s:X_w(s)=j\}
 \le\sum_{p\in\mathcal P_j}n_s(p).
\end{equation}

For $0\le u\le s$, let $p_u$ be the point on the ancestral line of $U$ at
time $s-u$.  Then $n_s(p_u)$ is nondecreasing in $u$.  Under $\mathbb Q_s$,
its distribution is
\begin{equation}\label{eq:tagged1}
 \mathbb Q_s\bigl(n_s(p_u)=n\bigr)
 =n e^{-2\lambda u}(1-e^{-\lambda u})^{n-1},
 \qquad n\ge1.
\end{equation}
To see this, condition at time $s-u$.  Each particle there roots an
independent Yule cluster of duration $u$.   By
\eqref{eq:yule2}, an ordinary cluster has probability
$e^{-\lambda u}(1-e^{-\lambda u})^{n-1}$ of size $n$ and mean
$e^{\lambda u}$; size-biasing gives \eqref{eq:tagged1}.

For $a\in\{1,\ldots,b-1\}$, let
$\Pp^{\mathcal T_s,U;a,u}$ denote the law of the point processes just
constructed after one additional point has been inserted in $\mathcal P_a$
at $p_u$,
and let $\E^{\mathcal T_s,U;a,u}$ denote the corresponding expectation.
Under this law, let $E_{a,u}$ be the event that
\begin{enumerate}[label=(\roman*),leftmargin=*]
\item the inserted point is retained;
\item it is the last actual transition into $a$ on the ancestral line of
$U$; and
\item the coordinate is saturated by time $s$, $a$ is its least-saturated
direction, and $U$ carries $a$ at time $s$.
\end{enumerate}
The Mecke equation \eqref{eq:mecke1} then relates $q_U(\mathcal T_s)$ to the probability of
$E_{a,u}$ as follows.

\begin{lemma}\label{lem:load2}
In the above setting,
\begin{equation}\label{eq:entry1}
 q_U(\mathcal T_s)
 =\frac1{d(b-1)}\sum_{a=1}^{b-1}\int_0^s
   \Pp^{\mathcal T_s,U;a,u}(E_{a,u})\dd u.
\end{equation}
\end{lemma}

\begin{proof}
Let $\varnothing$ denote the root and $[\varnothing,w]$ the unique
 path from $\varnothing$ to $w$.  For $a\in\{1,\ldots,b-1\}$, let
\[
 \mathcal A_a
 :=\{\text{the coordinate is saturated by time $s$, $a$ is its
 least-saturated direction, and $X_U(s)=a$}\}.
\]
Pathwise,
\[
 \one_{\bigcup_{a=1}^{b-1}\mathcal A_a}
 =
 \sum_{a=1}^{b-1}
 \sum_{p\in\mathcal P_a\cap[\varnothing,U]}
 \one_{\mathcal A_a}
 \one_{\{p\text{ is retained}\}}
 \one_{\{p\text{ is the last actual transition into }a\}}.
\]
Indeed, on $\mathcal A_a$ there is a unique last actual transition into
$a$ on $[\varnothing,U]$, while all terms vanish outside
$\bigcup_a\mathcal A_a$.  Therefore,
\[
\begin{aligned}
 q_U(\mathcal T_s)
 &=
 \sum_{a=1}^{b-1}
 \E\!\left[
 \sum_{p\in\mathcal P_a\cap[\varnothing,U]}
 \one_{\mathcal A_a}
 \one_{\{p\text{ is retained}\}}
 \one_{\{p\text{ is the last actual transition into }a\}}
 \,\middle|\,\mathcal T_s,U\right]  \\
 &=
 \frac1{d(b-1)}
 \sum_{a=1}^{b-1}\int_0^s
 \E^{\mathcal T_s,U;a,u}\!\left[\one_{E_{a,u}}\right]\dd u =
 \frac1{d(b-1)}
 \sum_{a=1}^{b-1}\int_0^s
 \Pp^{\mathcal T_s,U;a,u}(E_{a,u})\dd u,
\end{aligned}
\]
where the second equality follows from the Mecke equation
\eqref{eq:mecke1}, since each process $\mathcal P_a$ has intensity
$\ell_{\mathcal T_s}/(d(b-1))$ by construction.  This proves \eqref{eq:entry1}.
\end{proof}

To further bound $\Pp^{\mathcal T_s,U;a,u}(E_{a,u})$, observe that if the added point
accounts for many terminal leaves carrying $a$, then every other nonzero
symbol must occur at least as often because $a$ is least saturated.  The
pathwise bound \eqref{eq:terminal3} then turns this
condition into independent tail events.

Define
\begin{equation}\label{eq:tagged2}
 \mathsf T_k:=\int_0^s
 \one_{\{n_s(p_u)\le2k\}}\dd u,\quad k\geq 1;
 \qquad \mathsf T_0:=0
\end{equation}
and 
\begin{equation}\label{eq:auxiliary1}
 I_s:=\sum_{k\ge1}\Pi_{\mathcal T_s}(k)^{b-2}
 (\mathsf T_k-\mathsf T_{k-1}),
\end{equation}

\begin{lemma}\label{lem:load3}
Fix $a\in\{1,\ldots,b-1\}$ and $0\le u\le s$, and let $k\in\mathbb{N}$ be such that $n_s(p_u)\in\{2k-1,2k\}$.  Then
\begin{equation}\label{eq:inserted1}
 \Pp^{\mathcal T_s,U;a,u}(E_{a,u})
 \le\frac{2u}{d}+\Pi_{\mathcal T_s}(k)^{b-2}.
\end{equation}
 In particular, 
\begin{equation}\label{eq:clean1}
 d\,q_U(\mathcal T_s)\le\frac{s^2}{d}+I_s.
\end{equation}
\end{lemma}

\begin{proof}
In the configuration with the additional point, write $C_u$ for the number
of descendants at time $s$ for which the inserted point is retained and the
path from time $s-u$ to the descendant has no further mutation of this coordinate.
By independence of the exponential clocks before and after the inserted point,
\begin{equation}\label{eq:unclean1}
 \E^{\mathcal T_s,U;a,u}\!\left[
 n_s(p_u)-C_u\,\middle|\,
 \text{the inserted point is retained}\right]
 =n_s(p_u)(1-e^{-u/d})\le\frac ud n_s(p_u).
\end{equation}
Note that the event $C_u<n_s(p_u)/2$ implies
$n_s(p_u)-C_u>n_s(p_u)/2$.  Conditional Markov's
inequality and \eqref{eq:unclean1} therefore give
\begin{equation}\label{eq:unclean2}
 \Pp^{\mathcal T_s,U;a,u}\!\left(
 C_u<\frac{n_s(p_u)}2
 \,\middle|\,\text{the inserted point is retained}\right)
 \le \frac{u n_s(p_u)/d}{n_s(p_u)/2}
 =\frac{2u}{d}.
\end{equation}
Since $E_{a,u}$ includes retention of the inserted point,
\eqref{eq:unclean2} yields
\begin{equation}\label{eq:unclean3}
 \Pp^{\mathcal T_s,U;a,u}
 \left(E_{a,u}\cap
 \left\{C_u<\frac{n_s(p_u)}2\right\}\right)
 \le\frac{2u}{d}.
\end{equation}
On the event $E_{a,u}\cap\{C_u\ge n_s(p_u)/2\}$, the terminal count of 
$a$ is at least $k=\lceil n_s(p_u)/2\rceil$.  Since $a$ is the least-saturated direction,
every $j\in\{1,\ldots,b-1\}\setminus\{a\}$ has terminal count of at least
$k$.  The pathwise inequality
\eqref{eq:terminal3} therefore gives
\begin{equation}\label{eq:clean2}
 E_{a,u}\cap\left\{C_u\ge\frac{n_s(p_u)}2\right\}
 \subseteq
 \bigcap_{\substack{1\le j\le b-1\\j\ne a}}
 \left\{\sum_{p\in\mathcal P_j}n_s(p)\ge k\right\}.
\end{equation}
By independence, \eqref{eq:unclean3} and \eqref{eq:clean2} together give
\eqref{eq:inserted1}.

By definition, for each $k$, the set of times for which
$n_s(p_u)\in\{2k-1,2k\}$ has length
$\mathsf T_k-\mathsf T_{k-1}$.   Inserting
\eqref{eq:inserted1} into Lemma~\ref{lem:load2} and summing over $a$ now gives
\[
 d\,q_U(\mathcal T_s)
 \le\frac1{b-1}\sum_{a=1}^{b-1}
 \left(\int_0^s\frac{2u}{d}\dd u+I_s\right)
 =\frac{s^2}{d}+I_s,
\]
which proves \eqref{eq:clean1}.
\end{proof}

In view of Lemmas \ref{lem:load1} and \ref{lem:load3}, it remains to control exponential moments of $I_s$. This will be the goal of the next few lemmas.

\begin{lemma}\label{lem:auxiliary1}
For every
$\alpha>0$,
\begin{equation}\label{eq:auxiliary2}
 \E_{\mathbb Q_s}[e^{\alpha I_s}]
 \le1+\sum_{k\ge1}
 \E_{\mathbb Q_s}\!\left[
 (e^{\alpha\mathsf T_k}-e^{\alpha\mathsf T_{k-1}})
 \Pi_{\mathcal T_s}(k)^{b-2}\right].
\end{equation}
\end{lemma}

\begin{proof}
Let $N_{\min}$ be the minimum of $b-2$ independent variables with
the conditional law of $M^{\rm dom}$ and independent of anything else. Recalling \eqref{eq:dom}, for every integer $k\ge1$,
\begin{equation*}
 \mathbb Q_s(N_{\min}\ge k\mid\mathcal T_s,U)
 =\Pi_{\mathcal T_s}(k)^{b-2}.
\end{equation*}
A telescoping sum argument then yields
\begin{equation*}
\begin{split}
 \E_{\mathbb Q_s}[\mathsf T_{N_{\min}}\mid\mathcal T_s,U]
 &=\sum_{k\ge1}(\mathsf T_k-\mathsf T_{k-1})
   \mathbb Q_s(N_{\min}\ge k\mid\mathcal T_s,U)\\
 &=\sum_{k\ge1}(\mathsf T_k-\mathsf T_{k-1})
   \Pi_{\mathcal T_s}(k)^{b-2}=I_s,
\end{split}
\end{equation*}
where the last equality follows from the definition \eqref{eq:auxiliary1}.  By conditional Jensen's
inequality and another telescoping sum argument, we have
\[
\begin{split}
 e^{\alpha I_s}
 &=\exp\!\left(\alpha
   \E_{\mathbb Q_s}[\mathsf T_{N_{\min}}\mid\mathcal T_s,U]\right)\\
 &\le \E_{\mathbb Q_s}
   [e^{\alpha\mathsf T_{N_{\min}}}\mid\mathcal T_s,U]\\
 &=1+\sum_{k\ge1}
   (e^{\alpha\mathsf T_k}-e^{\alpha\mathsf T_{k-1}})
   \mathbb Q_s(N_{\min}\ge k\mid\mathcal T_s,U)\\
 &=1+\sum_{k\ge1}
   (e^{\alpha\mathsf T_k}-e^{\alpha\mathsf T_{k-1}})
   \Pi_{\mathcal T_s}(k)^{b-2}.
\end{split}
\]
Taking $\mathbb Q_s$-expectations yields
\eqref{eq:auxiliary2}.
\end{proof}


\begin{lemma}\label{lem:tagged1}
Suppose that $\alpha>0$ and $1\le p<\infty$
satisfy $p\alpha<2\lambda$. Then uniformly for $k\ge1$,
\begin{equation}\label{eq:holding1}
 \|e^{\alpha\mathsf T_k}-e^{\alpha\mathsf T_{k-1}}\|_{L^p(\mathbb Q_s)}
 \ll_{\alpha,p}(1+k)^{\alpha/\lambda-1/p}.
\end{equation}
\end{lemma}

\begin{proof}
By the definition \eqref{eq:tagged2} and the fundamental theorem of calculus,
\begin{equation*}
 e^{\alpha\mathsf T_k}-e^{\alpha\mathsf T_{k-1}}
 =\alpha\int_0^s e^{\alpha u}
 \one_{\{n_s(p_u)\in\{2k-1,2k\}\}}\dd u.
\end{equation*}
By \eqref{eq:tagged1} and $1-z\le e^{-z}$,
\[
\begin{aligned}
 \mathbb Q_s\!\left(n_s(p_u)\in\{2k-1,2k\}\right)
 &=\sum_{n\in\{2k-1,2k\}}
 n e^{-2\lambda u}(1-e^{-\lambda u})^{n-1}\\
 &\ll k e^{-2\lambda u}
 (1-e^{-\lambda u})^{2k-2}\ll k e^{-2\lambda u}
 \exp\!\left(-c k e^{-\lambda u}\right).
\end{aligned}
\]
Applying Minkowski's integral inequality and the change of
variable $z=ke^{-\lambda u}$ gives
\begin{equation}\label{eq:holding3}
\begin{split}
 \|e^{\alpha\mathsf T_k}-e^{\alpha\mathsf T_{k-1}}\|_{L^p(\mathbb Q_s)}
 &\ll k^{1/p}\int_0^\infty
 e^{(\alpha-2\lambda/p)u}
 \exp\!\left(-\frac{ck e^{-\lambda u}}{p}\right)\dd u\\
 &={\lambda}^{-1}k^{\alpha/\lambda-1/p}
 \int_0^k z^{2/p-\alpha/\lambda-1}e^{-cz/p}\dd z.
\end{split}
\end{equation}
Since $p\alpha<2\lambda$, 
$$\int_0^k z^{2/p-\alpha/\lambda-1}e^{-cz/p}\dd z
 \le \int_0^1 z^{2/p-\alpha/\lambda-1}\dd z
 +\int_1^\infty z^{2/p-\alpha/\lambda-1}e^{-cz/p}\dd z
 \ll_{\alpha,p}1.$$
Inserting into \eqref{eq:holding3} then proves \eqref{eq:holding1}.
\end{proof}


\begin{lemma}\label{lem:tails1}
For every integer $P\ge2$ and every $0\le\gamma<P-1$, for $d$ large enough,
\begin{equation*}
 \sum_{k\ge1}(1+k)^\gamma
 \E_{\mathbb Q_s}\!\left[\Pi_{\mathcal T_s}(k)^P\right]
 \ll_{P,\gamma,B}1.
\end{equation*}
\end{lemma}

\begin{proof}
Define a finite measure on $\mathbb N$ by
\begin{equation}\label{eq:measure1}
 \nu_{\mathcal T_s}(A)
 :=\frac1{d(b-1)}\int_0^s
 \#\!\left\{v\in V_{s-u}:
 \#\{w\in V_s:w\text{ descends from }v\}\in A\right\}\dd u,
 \qquad A\subseteq\mathbb N.
\end{equation}
It follows from the graphical construction that conditional on $\mathcal T_s$ and for any
fixed $j$, the number $N_k$ of marks in $\mathcal P_j$ with exactly $k$
terminal descendants satisfies
\begin{equation}\label{eq:compound3}
 (N_k)_{k\ge1}\text{ are independent},\qquad
 N_k\sim\operatorname{Pois}(\nu_{\mathcal T_s}(\{k\}))\quad(k\ge1),\qquad
 M^{\rm dom}=\sum_{k\ge1}kN_k.
\end{equation}
Thus, $M^{\rm dom}$ is a compound-Poisson sum with intensity
$\nu_{\mathcal T_s}$.  Moreover, the total measure 
$\nu_{\mathcal T_s}(\mathbb N)=\sum_{k\ge1}\nu_{\mathcal T_s}(\{k\})$ is the mean
number of marks.

Define 
$N_u(k)=\#\{v\in V_{s-u}:v\text{ has at least $k$ descendants in }V_s\}$.
Conditionally on $|V_{s-u}|=n$, the variable
$N_u(k)\sim\mathrm{Bin}(n,(1-e^{-\lambda u})^{k-1})$, thanks to 
\eqref{eq:yule2}. 
Therefore, for every integer $L\ge2$,
\begin{equation}\label{eq:moment1}
 \E[N_u(k)^L]\ll_L\sum_{j=1}^L
 (1-e^{-\lambda u})^{j(k-1)}\E[|V_{s-u}|^j]
 \ll_L\sum_{j=1}^L
 e^{j\lambda(s-u)}(1-e^{-\lambda u})^{j(k-1)} .
\end{equation}
By \eqref{eq:measure1}, Minkowski's integral inequality and
\eqref{eq:moment1} give, with $\|\cdot\|_p$ denoting the $L^p$ norm under the Yule law,
\begin{align*}
 \|\nu_{\mathcal T_s}([k,\infty))\|_L
 &\le\frac1{d(b-1)}\int_0^s\|N_u(k)\|_L\dd u\\
 &\ll_L\frac1d\int_0^s
 \left(\sum_{j=1}^L e^{j\lambda(s-u)}
 (1-e^{-\lambda u})^{j(k-1)}\right)^{1/L}\dd u\\
 &\ll_L\sum_{j=1}^L
 \frac{e^{(j/L)\lambda s}}d
 \int_0^s e^{-(j/L)\lambda u}
 (1-e^{-\lambda u})^{(j/L)(k-1)}\dd u,
\end{align*}
where in the last line we used
$(\sum_j a_j)^{1/L}\le\sum_j a_j^{1/L}$ for nonnegative $a_j$ and $L\geq 2$.

Lemma~\ref{lem:beta1} below 
applied with exponent $j/L$ gives some $c_{j/L}>0$ such that
\begin{align}
 \int_0^\infty e^{-(j/L)\lambda u}
 (1-e^{-\lambda u})^{(j/L)(k-1)}\dd u
 &\ll_{j/L}(1+k)^{-j/L}\label{eq:beta1}
\end{align}
and
\begin{align}
 \int_0^s e^{-(j/L)\lambda u}
 (1-e^{-\lambda u})^{(j/L)(k-1)}\dd u
 &\ll_{j/L}(1+k)^{-j/L}e^{-c_{j/L} k e^{-\lambda s}},
 \qquad k>e^{\lambda s}.\label{eq:beta2}
\end{align}
When $k\le e^{\lambda s}$, \eqref{eq:beta1} and
$e^{\lambda s}=d\,e^{-\lambda B}$ give
\[
 \frac{e^{(j/L)\lambda s}}d(1+k)^{-j/L}
 =\frac{e^{\lambda s}}{d}(1+k)^{-1}
 \left(\frac{1+k}{e^{\lambda s}}\right)^{1-j/L}
 \ll_B\frac1{1+k}.
\]
Here, we used that $(1+k)/e^{\lambda s}\le1+e^{-\lambda s}\le2$ for all $d$ large enough.  When $k>e^{\lambda s}$, \eqref{eq:beta2} gives
\[
 \frac{e^{(j/L)\lambda s}}d(1+k)^{-j/L}
 e^{-c_{j/L}k e^{-\lambda s}}
 =\frac{e^{\lambda s}}{d(1+k)}
 \bigl((1+k)e^{-\lambda s}\bigr)^{1-j/L}
 e^{-c_{j/L}k e^{-\lambda s}}
 \ll_B\frac1{1+k}.
\]
Indeed, $k>e^{\lambda s}$ implies
$(1+k)e^{-\lambda s}\le2k e^{-\lambda s}$ for all large $d$, and
$(k e^{-\lambda s})^{1-j/L}e^{-c_{j/L}k e^{-\lambda s}}$ is uniformly
bounded.  Hence, for every $k\ge1$,
\begin{equation}\label{eq:measure2}
 \|\nu_{\mathcal T_s}([k,\infty))\|_L
 \ll_{L,B}\frac1{1+k}.
\end{equation}
The size-bias identity \eqref{eq:palm1}, \eqref{eq:measure2}, and Cauchy--Schwarz together imply that
\begin{equation}\label{eq:measure3}
 \E_{\mathbb Q_s}\!\left[\nu_{\mathcal T_s}([k,\infty))^P\right]
 =e^{-\lambda s}\E\!\left[|V_s|\nu_{\mathcal T_s}([k,\infty))^P\right]
 \le e^{-\lambda s}\|V_s\|_2
 \|\nu_{\mathcal T_s}([k,\infty))\|_{2P}^P
 \ll_{P,B}\frac1{(1+k)^P}.
\end{equation}
 Similarly, for $r_0>0$, applying \eqref{eq:measure2} with $k=1$ gives
\begin{align}
    \begin{split}
        \label{eq:moments1}
 \E_{\mathbb Q_s}[\nu_{\mathcal T_s}(\mathbb N)^{r_0}]
 =e^{-\lambda s}\E[|V_s|\nu_{\mathcal T_s}(\mathbb N)^{r_0}]
 &\le e^{-\lambda s}\|V_s\|_2
 \|\nu_{\mathcal T_s}(\mathbb N)\|_{2r_0}^{r_0}\\
 &\le e^{-\lambda s}\|V_s\|_2
 \|\nu_{\mathcal T_s}(\mathbb N)\|_{\lceil \max\{2,2r_0\}\rceil}^{r_0}
 \ll_{r_0,B}1.
    \end{split}
\end{align}

Applying Lemma~\ref{lem:compound1} conditionally on
$\mathcal T_s$ to \eqref{eq:compound3} and then using \eqref{eq:dom} lead to
\[
\begin{split}
 &\sum_{k\ge1}(1+k)^\gamma\Pi_{\mathcal T_s}(k)^P
 \ll_\gamma(1+\nu_{\mathcal T_s}(\mathbb N))^{\gamma+1}
 \sum_{k\ge1}(1+k)^\gamma
 \nu_{\mathcal T_s}([k,\infty))^P,
\end{split}
\]
and hence 
\begin{align}
    \sum_{k\ge1}(1+k)^\gamma
 \E_{\mathbb Q_s}[\Pi_{\mathcal T_s}(k)^P]
 \ll_\gamma\sum_{k\ge1}(1+k)^\gamma
 \E_{\mathbb Q_s}\!\left[
 (1+\nu_{\mathcal T_s}(\mathbb N))^{\gamma+1}
 \nu_{\mathcal T_s}([k,\infty))^P\right].\label{eq:moments2}
\end{align}
For each $k$, H\"older's inequality, 
\eqref{eq:measure3}, and \eqref{eq:moments1} give
\begin{align*}
 &\E_{\mathbb Q_s}\!\left[(1+\nu_{\mathcal T_s}(\mathbb N))^{\gamma+1}
 \nu_{\mathcal T_s}([k,\infty))^P\right]\\
 &\quad\le
 \left(\E_{\mathbb Q_s}[
 (1+\nu_{\mathcal T_s}(\mathbb N))^{2\gamma+2}]
 \right)^{1/2}
 \left(\E_{\mathbb Q_s}[
 \nu_{\mathcal T_s}([k,\infty))^{2P}]\right)^{1/2}
 \ll_{P,\gamma,B}\frac1{(1+k)^P}.
\end{align*}
Inserting into \eqref{eq:moments2} and using 
$\gamma<P-1$ completes the proof.
\end{proof}

\begin{proof}[Proof of Proposition~\ref{prop:load1}]
Set $\alpha:=\vartheta-1$.  Since
$1<\vartheta<1+\lambda$ and $b>2$,
\begin{equation*}
 0<\alpha<\lambda\min\{2,b-2\}.
\end{equation*}
We may then choose an integer $P>b-2$ large enough such that
\begin{equation}\label{eq:holder1}
 \frac{P\alpha}{P-(b-2)}<2\lambda,
 \qquad
 \frac{\alpha}{\lambda}<\frac{(b-2)(P-2)}P.
\end{equation}
 The first inequality in \eqref{eq:holder1} will be required to apply
Lemma~\ref{lem:tagged1} with the H\"older exponent
$P/(P-(b-2))$; the second inequality in \eqref{eq:holder1} allows us to choose 
$\gamma\in(1+\alpha P/((b-2)\lambda),P-1)$.
We next apply H\"older's inequality with conjugate exponents
$P/(P-(b-2))$ and $P/(b-2)$, along with
Lemmas~\ref{lem:tagged1} and~\ref{lem:tails1}, to obtain
\begin{align*}
 \E_{\mathbb Q_s}\!\left[
 (e^{\alpha\mathsf T_k}-e^{\alpha\mathsf T_{k-1}})
 \Pi_{\mathcal T_s}(k)^{b-2}\right]
 &\le
 \|e^{\alpha\mathsf T_k}-e^{\alpha\mathsf T_{k-1}}\|
 _{L^{P/(P-(b-2))}(\mathbb Q_s)}
 \E_{\mathbb Q_s}[\Pi_{\mathcal T_s}(k)^P]^{(b-2)/P}\\
 &\ll
 (1+k)^{\alpha/\lambda-(P-(b-2))/P}
 (1+k)^{-\gamma(b-2)/P}\\
 &= (1+k)^{-1+\alpha/\lambda-(b-2)(\gamma-1)/P}.
\end{align*}
The exponent is strictly less than $-1$ by the choice of $\gamma$, so the right-hand side is summable.
Inserting into Lemmas~\ref{lem:load3} and \ref{lem:auxiliary1} and using $s^2/d=O((\log d)^2/d)=o(1)$ then lead to
\[
\limsup_{d\to\infty}\E_{\mathbb Q_s}\!\left[
 e^{\alpha d q_U(\mathcal T_s)}\right]\ll \limsup_{d\to\infty}\E_{\mathbb Q_s}\!\left[
 e^{\alpha (s^2/d+I_s)}\right]\leq  \limsup_{d\to\infty}\E_{\mathbb Q_s}[e^{\alpha I_s}]<\infty.
\]
Lemma~\ref{lem:load1} then completes the proof since $\alpha=\vartheta-1$.
\end{proof}

\section{Proof of Theorem~\ref{thm:cover2}}\label{sec:binary1}
In this section, we assume throughout $b=2$. 
We first introduce a short-hand notation of a shell. Recall that the underlying random walk is the nearest-neighbor random walk on the hypercube. Starting at a location $x\in\{0,1\}^d$, we may project the random walk using the distance function from $x$. In this case, the projected random walk is a continuous-time Ehrenfest chain starting at $0$: from state
$k$, it jumps to $k+1$ at rate $(d-k)/d$ and to $k-1$ at rate $k/d$. In particular, the antipode corresponds to the single target at state $d$. When $x$ is clear from context, we call the targets at state $\ell$ the \textit{shell} $\ell$.

\subsection{Proof of the lower bound}\label{sec:binary2}

The lower bound will use the $d$ targets in shell $d-1$.  Because these
targets are only of Hamming distance two apart, one descendant cluster
may visit several of them.  The following factorial-moment estimate controls
that dependence.

\begin{lemma}\label{lem:cluster2}
For a set $A\subseteq\{0,1\}^d$, let $J_A(K)$ be the
number of mutation jumps landing in $A$.  There exists $C<\infty$,
depending only on the branching and total mutation rates, such that uniformly over $A$ and the starting state $x$,
\begin{equation}\label{eq:cluster2}
 \E_x[J_A(K)(J_A(K)-1)]
 \le e^{C(1+K)}\E_x[J_A(K)].
\end{equation}
\end{lemma}

The proof is an application of the many-to-two formula and is deferred to Appendix~\ref{app:cluster2}.

\begin{proof}[Proof of the lower bound in Theorem~\ref{thm:cover2}]
For $i\in\{1,\ldots,d\}$, let $y_i$ be the vertex with a zero in coordinate
$i$ and ones elsewhere.  These are precisely the $d$ vertices in shell
$d-1$. We split the proof into three steps. 
\par\smallskip\noindent\emph{Step 1: setting up early targets.}
Fix $A>0$ and let  $T_{A,d}:=x_\star d+\chi^{-1}\log\log d-A$.  Our goal is to prove
\begin{equation*}
 \lim_{A\to\infty}\limsup_{d\to\infty}\Pp\!\left(\tcov(d)\le T_{A,d}\right)=0.
\end{equation*}
Choose a fixed $\kappa>\chi^{-1}$ and let
$r_d=T_{A,d}-\kappa\log\log d$.

\begin{figure}[ht]
\centering
\begin{tikzpicture}[x=.96\linewidth,y=1cm,>=stealth,
  every node/.style={font=\footnotesize,inner sep=1pt}]
  \draw[->] (.02,0)--(.99,0);
  \foreach \x in {.02,.68,.82,.96}
    \draw (\x,-.07)--(\x,.07);
  \node[below=2pt] at (.02,0) {$0$};
  \node[above=2pt] at (.68,0) {$r_d$};
  \node[above=2pt,text=gray] at (.82,0) {$x_\star d$};
  \node[above=2pt] at (.96,0) {$T_{A,d}$};
  \draw[decorate,decoration={brace,mirror,amplitude=3pt}]
    (.68,-.30)--(.96,-.30)
    node[midway,below=4pt]
    {$T_{A,d}-r_d=\kappa\log\log d$};
\end{tikzpicture}
\caption{The key time points considered here.}
\label{fig:binary1}
\end{figure}
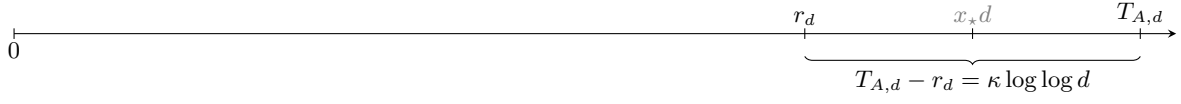

Lemma~\ref{lem:endpoint1}(i) applied with $C_D=C_{\mathrm{ov}}=\ell=1$ gives
\[
 \sup_i\E[J_{y_i}(r_d)]
 \le C e^{-\chi A}(\log d)^{1-\chi\kappa}=o(1).
\]
Let $\cU_d:=\{i:\tau_{y_i}>r_d\}$ be the unreached index set in shell $d-1$ at time $r_d$.  Since a
visit to $y_i$ requires an incoming jump,
$\E[d-|\cU_d|]\le\sum_{i=1}^d\E[J_{y_i}(r_d)]=o(d)$.  Markov's inequality
therefore gives $|\cU_d|=d-o_{\Pp}(d)$.

Condition on $\cF_{r_d}$.  The sets $V_{r_d}$ and $\cU_d$ are then fixed.  The descendant clusters of particles at time $r_d$ over the next
$\kappa\log\log d$ units of time are
conditionally independent.  For $v\in V_{r_d}$, define
\[
 S_v=\{i\in\cU_d:\text{ a descendant of $v$ visits $y_i$
 during $(r_d,T_{A,d}]$}\}.
\]
Figure~\ref{fig:binary2} illustrates this definition.\footnote{Here, each $S_v$ need not be a singleton, in contrast to the case $b\geq 3$.}

\begin{figure}[ht]
\centering
\begin{tikzpicture}[x=1cm,y=1cm,>=stealth,
  every node/.style={font=\footnotesize},
  cluster/.style={draw,rounded corners=2pt,fill=blue!4,
    minimum width=3.25cm,minimum height=.72cm,align=center},
  target/.style={circle,draw,fill=gray!4,minimum size=7mm,inner sep=0pt},
  samecluster/.style={->,draw=orange!85!black,line width=1.15pt},
  othercluster/.style={->,draw=blue!65!black,line width=.9pt}]
  \node[font=\small\bfseries] at (2.35,3.45)
    {independent descendant clusters};
  \node[font=\small\bfseries] at (11.20,3.45)
    {remaining targets};

  \node[cluster,draw=orange!85!black,fill=orange!4] (c1) at (2.35,2.55)
    {cluster rooted at $v_1$\\$S_{v_1}=\{1,2\}$};
  \node[cluster] (c2) at (2.35,1.45) {cluster rooted at $v_2$};
  \node[cluster] (c3) at (2.35,.35) {cluster rooted at $v_3$};

  \node[target] (y1) at (11.20,2.85) {$y_1$};
  \node[target] (y2) at (11.20,1.95) {$y_2$};
  \node[target] (y3) at (11.20,1.05) {$y_3$};
  \node[target] (y4) at (11.20,.15) {$y_4$};

  \draw[samecluster] (c1.east) to[bend left=7] (y1.west);
  \draw[samecluster] (c1.east) to[bend right=4] (y2.west);
  \draw[othercluster] (c2.east) to[bend left=3] (y3.west);
  \draw[othercluster] (c3.east) to[bend right=8] (y4.west);

  \node[text=orange!85!black] at (6.75,2.58)
    {one cluster, two targets};
\end{tikzpicture}
\caption{Illustrating the definition of $S_v$ through cluster hits.  One cluster may hit several nearby
targets, so each $S_v$ need not have size at most one.}
\label{fig:binary2}
\end{figure}
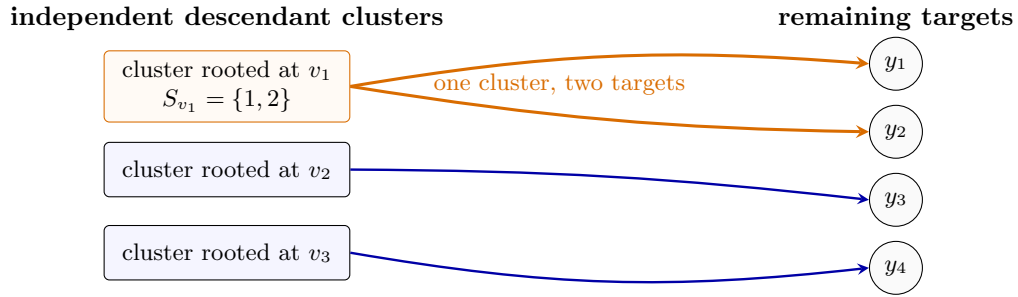

Conditional on $\cF_{r_d}$, the sets $S_v$ are independent across distinct
$v\in V_{r_d}$, but a single $S_v$ may contain several indices.  For
$D\subseteq\{0,1\}^d$, write $J_D^{(v)}(t)$ for the number of mutation
jumps into $D$ by the cluster $v$ by time $t$.  Because
$X_v(r_d)\ne y_i$ for $i\in\cU_d$, an upper bound similar to  \eqref{eq:probability1} gives
\begin{equation}\label{eq:binary2}
 \Pp(i\in S_v\mid\cF_{r_d})
 \le \E[J_{\{y_i\}}^{(v)}(\kappa\log\log d)\mid\cF_{r_d}]
 \le\frac1d\int_0^{\kappa\log\log d}e^{\lambda s}\dd s
 \ll\frac{e^{\lambda\kappa\log\log d}}d
 =o(1),
\end{equation}
uniformly for $v\in V_{r_d}$ and $i\in\cU_d$.

In addition, the tower property, followed by the fact that every such visit requires an
incoming jump, gives
\begin{equation*}
\begin{split}
 &\E\!\left[\sum_{i\in\cU_d}\sum_{v\in V_{r_d}}
   \Pp(i\in S_v\mid\cF_{r_d})\right]\\
 &\quad=\sum_{i=1}^d\E\!\left[
   \one_{\{\tau_{y_i}>r_d\}}\sum_{v\in V_{r_d}}
   \one_{\{i\in S_v\}}\right]
 \le\sum_{i=1}^d\E[J_{y_i}(T_{A,d})]
 \le C e^{-\chi A}d\log d.
\end{split}
\end{equation*}
Here, the final inequality uses Lemma~\ref{lem:endpoint1}(i) with $\ell=1$, since $x_\star d-T_{A,d}=A-\chi^{-1}\log\log d$ and hence
$e^{-\chi(x_\star d-T_{A,d})}=e^{-\chi A}\log d$.
Therefore, by Markov's inequality, except on an event of probability at most $Ce^{-\chi A/2}$,
\begin{equation}\label{eq:binary4}
 \sum_{i\in\cU_d}\sum_{v\in V_{r_d}}
 \Pp(i\in S_v\mid\cF_{r_d})
 \le e^{-\chi A/2}d\log d.
\end{equation}
On this event, all but at most $d/4$ indices of $\cU_d$ satisfy
$\sum_{v\in V_{r_d}}\Pp(i\in S_v\mid\cF_{r_d})
\le4e^{-\chi A/2}\log d$.
Since $|\cU_d|=d-o_{\Pp}(d)$, outside a further event of probability
$o(1)$ we have $|\cU_d|\ge3d/4$.  Define the
$\cF_{r_d}$-measurable set
\begin{equation}\label{eq:targets1}
\mathcal L_d:=\left\{i\in\cU_d:
\sum_{v\in V_{r_d}}\Pp(i\in S_v\mid\cF_{r_d})
\le4e^{-\chi A/2}\log d\right\}.
\end{equation}
On the intersection of these events above, $|\mathcal L_d|\ge d/2$.

\par\smallskip\noindent\emph{Step 2: pairs visited by the same descendant cluster.}
Conditional on
$\cF_{r_d}$, the fact that
$X_v(r_d)\notin\{y_i:i\in\cU_d\}$ gives
\[
 |S_v|(|S_v|-1)
 \le J_{\{y_i:i\in\cU_d\}}^{(v)}(\kappa\log\log d)
       \bigl(J_{\{y_i:i\in\cU_d\}}^{(v)}(\kappa\log\log d)-1\bigr).
\]
Applying Lemma~\ref{lem:cluster2} conditionally to each cluster then leads to
\[
\begin{split}
 &\sum_{\substack{i,j\in\cU_d\\i\ne j}}
   \sum_{v\in V_{r_d}}\Pp(\{i,j\}\subseteq S_v\mid\cF_{r_d})\\
 &\quad=\sum_{v\in V_{r_d}}
   \E[|S_v|(|S_v|-1)\mid\cF_{r_d}]\\
 &\quad\le e^{C(1+\kappa\log\log d)}
   \sum_{v\in V_{r_d}}\E[
   J_{\{y_i:i\in\cU_d\}}^{(v)}(\kappa\log\log d)
   \mid\cF_{r_d}].
\end{split}
\]
After taking expectation, the first-moment sum on the right-hand side satisfies
\begin{equation*}
 \E\!\left[\sum_{v\in V_{r_d}}
 \E[J_{\{y_i:i\in\cU_d\}}^{(v)}(\kappa\log\log d)
 \mid\cF_{r_d}]\right]
 \le\sum_{i=1}^d\E[J_{y_i}(T_{A,d})]
 \ll e^{-\chi A}d\log d\ll d\log d,
\end{equation*}
again by Lemma~\ref{lem:endpoint1}(i).  By choosing $M>C\kappa+3$, we conclude that
\[
 \E\!\left[\sum_{v\in V_{r_d}}
 \E[|S_v|(|S_v|-1)\mid\cF_{r_d}]\right]
 \ll d(\log d)^{M-2}.
\]
Markov's inequality gives
\begin{equation}\label{eq:binary6}
 \Pp\!\left(\sum_{v\in V_{r_d}}
 \E[|S_v|(|S_v|-1)\mid\cF_{r_d}]>d(\log d)^M\right)
 \ll(\log d)^{-2}=o(1).
\end{equation}

Define the following $\cF_{r_d}$-measurable event:
\begin{equation}\label{eq:binary7}
\begin{split}
 \mathcal G_{A,d}:={}&
 \left\{\sum_{i\in\cU_d}\sum_{v\in V_{r_d}}
 \Pp(i\in S_v\mid\cF_{r_d})
 \le e^{-\chi A/2}d\log d\right\}
 \cap\left\{|\cU_d|\ge\frac{3d}{4}\right\}\\
 &\cap\left\{\sum_{v\in V_{r_d}}
 \E[|S_v|(|S_v|-1)\mid\cF_{r_d}]
 \le d(\log d)^M\right\}.
\end{split}
\end{equation}
In particular, from step 1, on $\mathcal G_{A,d}$ we have  $|\mathcal L_d|\ge d/2$.
The probability estimates following
\eqref{eq:binary4}, together with \eqref{eq:binary6}, give
\begin{equation}\label{eq:binary8}
 \Pp(\mathcal G_{A,d}^{\mathrm c})
 \le Ce^{-\chi A/2}+o_A(1).
\end{equation}
\par\smallskip\noindent\emph{Step 3: second moment for unvisited targets.}
We condition on 
$\cF_{r_d}$ and restrict to the event $\mathcal G_{A,d}$.  In particular, $|\mathcal L_d|\ge d/2$.  For $i\in\mathcal L_d$, define
\[
 I_i=\one_{\{y_i\text{ is unhit during }(r_d,T_{A,d}]\}},
 \qquad Z=\sum_{i\in\mathcal L_d}I_i.
\]

For a fixed $i$, independence gives
\[
 \E[I_i\mid\cF_{r_d}]
 =\prod_{v\in V_{r_d}}
   \left(1-\Pp(i\in S_v\mid\cF_{r_d})\right).
\]
By~\eqref{eq:binary2}, every probability in this product is at most
$1/5$ for $d$ large enough.  Moreover,
\[
 -\log(1-u)=\int_0^u\frac{\mathrm d s}{1-s}\le\frac54u,
 \qquad 0\le u\le\frac15.
\]
It follows from \eqref{eq:targets1} that
\[
\begin{split}
 \E[I_i\mid\cF_{r_d}]
 &\ge\exp\!\left(-\frac54\sum_{v\in V_{r_d}}
       \Pp(i\in S_v\mid\cF_{r_d})\right)\\
 &\ge\exp\!\left(-5e^{-\chi A/2}\log d\right)
 =d^{-5e^{-\chi A/2}}.
\end{split}
\]
Since $\mathcal L_d$ is $\cF_{r_d}$-measurable and $|\mathcal L_d|\ge d/2$, we have
\begin{equation}\label{eq:binary9}
 \E[Z\mid\cF_{r_d}]=\sum_{i\in\mathcal L_d}\E[I_i\mid\cF_{r_d}]
 \ge\frac12d^{1-5e^{-\chi A/2}}.
\end{equation}
From now on, take $A$ large enough that $5e^{-\chi A/2}<1/8$.

To compare the joint and marginal probabilities of missing two targets in the shell $d-1$, we
follow the coupon-collecting approach in \cite[Section~3]{StructuredCoupons}. 
Fix distinct $i,j\in\mathcal L_d$ and one descendant cluster $v$. By 
\eqref{eq:binary2} and the inclusion--exclusion principle,
\begin{equation*}
\begin{split}
 &\frac{\Pp(\{i,j\}\cap S_v=\emptyset\mid\cF_{r_d})}
 {\Pp(i\notin S_v\mid\cF_{r_d})
  \Pp(j\notin S_v\mid\cF_{r_d})}\\
 &=1+\frac{
 \Pp(\{i,j\}\subseteq S_v\mid\cF_{r_d})
 -\Pp(i\in S_v\mid\cF_{r_d})
  \Pp(j\in S_v\mid\cF_{r_d})}
 {\Pp(i\notin S_v\mid\cF_{r_d})
  \Pp(j\notin S_v\mid\cF_{r_d})}\\
 &\le1+\frac{25}{16}\Pp(\{i,j\}\subseteq S_v\mid\cF_{r_d})
 \le\exp\!\left(2\Pp(\{i,j\}\subseteq S_v\mid\cF_{r_d})\right).
\end{split}
\end{equation*}
By independence, we arrive at
\begin{equation}\label{eq:binary11}
\begin{split}
 \E[I_iI_j\mid\cF_{r_d}]
 &=\prod_{v\in V_{r_d}}
   \Pp(\{i,j\}\cap S_v=\emptyset\mid\cF_{r_d})\\
 &\le\left(\prod_{v\in V_{r_d}}
       \Pp(i\notin S_v\mid\cF_{r_d})\right)
       \left(\prod_{v\in V_{r_d}}
       \Pp(j\notin S_v\mid\cF_{r_d})\right)\\
 &\qquad\times\exp\!\left(2\sum_{v\in V_{r_d}}
       \Pp(\{i,j\}\subseteq S_v\mid\cF_{r_d})\right)\\
 &=\E[I_i\mid\cF_{r_d}]\E[I_j\mid\cF_{r_d}]
   \exp\!\left(2\sum_{v\in V_{r_d}}
       \Pp(\{i,j\}\subseteq S_v\mid\cF_{r_d})\right).
\end{split}
\end{equation}

Next, following the good-pair/bad-pair decomposition of
\cite[proof of Theorem~5(2)]{StructuredCoupons}, we say
$(i,j)\in\mathcal L_d^2$, $i\ne j$ is bad if
\[
 \sum_{v\in V_{r_d}}
 \Pp(\{i,j\}\subseteq S_v\mid\cF_{r_d})>d^{-1/4},
\]
and it is good otherwise.  Because $|S_v|(|S_v|-1)$ counts
pairs of distinct indices in $S_v$, the definition \eqref{eq:binary7} of 
$\mathcal G_{A,d}$ gives
\[
\begin{split}
 &d^{-1/4}\#\{\text{bad ordered pairs in }\mathcal L_d^2\}\\
 &\quad\le\sum_{\substack{i,j\in\mathcal L_d\\i\ne j}}
       \sum_{v\in V_{r_d}}
       \Pp(\{i,j\}\subseteq S_v\mid\cF_{r_d})\\
 &\quad\le\sum_{\substack{i,j\in\cU_d\\i\ne j}}
       \sum_{v\in V_{r_d}}
       \Pp(\{i,j\}\subseteq S_v\mid\cF_{r_d})\\
 &\quad=\sum_{v\in V_{r_d}}
       \E[|S_v|(|S_v|-1)\mid\cF_{r_d}]
 \le d(\log d)^M.
\end{split}
\]
Thus, there are at most $d^{5/4}(\log d)^M$ bad pairs.

On the other hand, for a good pair, \eqref{eq:binary11} implies
\[
\begin{split}
 \operatorname{Cov}(I_i,I_j\mid\cF_{r_d})
 &\le\E[I_i\mid\cF_{r_d}]\E[I_j\mid\cF_{r_d}]\\
 &\quad\times\left[
   \exp\!\left(2\sum_{v\in V_{r_d}}
       \Pp(\{i,j\}\subseteq S_v\mid\cF_{r_d})\right)-1\right]\\
 &\ll d^{-1/4}\E[I_i\mid\cF_{r_d}]
       \E[I_j\mid\cF_{r_d}].
\end{split}
\]
The last inequality follows from $e^{2x}-1\le2e^2x$ for $0\le x\le1$ and the definition of a good pair.
For a bad pair, we have trivially
$\operatorname{Cov}(I_i,I_j\mid\cF_{r_d})
\le\E[I_iI_j\mid\cF_{r_d}]\le1$.
Since $I_i$ is an indicator,
$\operatorname{Var}(I_i\mid\cF_{r_d})\le\E[I_i\mid\cF_{r_d}]$.
Expanding the conditional variance and separating the good and bad pairs then gives
\[
\begin{split}
 \operatorname{Var}(Z\mid\cF_{r_d})
 &=\sum_{i\in\mathcal L_d}
       \operatorname{Var}(I_i\mid\cF_{r_d})
   +\sum_{\substack{i,j\in\mathcal L_d\\i\ne j}}
       \operatorname{Cov}(I_i,I_j\mid\cF_{r_d})\\
 &\ll\sum_{i\in\mathcal L_d}\E[I_i\mid\cF_{r_d}]
   +d^{-1/4}
       \sum_{\substack{i,j\in\mathcal L_d\\i\ne j,\ (i,j)\text{ good}}}
       \E[I_i\mid\cF_{r_d}]\E[I_j\mid\cF_{r_d}]
       +d^{5/4}(\log d)^M\\
 &\ll d+d^{-1/4}
       \left(\sum_{i\in\mathcal L_d}\E[I_i\mid\cF_{r_d}]\right)^2
       +d^{5/4}(\log d)^M\\
 &\ll d^{-1/4}\E[Z\mid\cF_{r_d}]^2
       +d^{5/4}(\log d)^M.
\end{split}
\]
On $\mathcal G_{A,d}$, \eqref{eq:binary9} gives
\[
\begin{split}
 \frac{\operatorname{Var}(Z\mid\cF_{r_d})}
      {\E[Z\mid\cF_{r_d}]^2}
 &\ll d^{-1/4}+4d^{-3/4+10e^{-\chi A/2}}(\log d)^M=o_A(1),
\end{split}
\]
where we used $10e^{-\chi A/2}<1/4$.  On $\mathcal G_{A,d}$,
conditional Chebyshev's inequality gives
\begin{equation}\label{eq:binary12}
\begin{split}
 \Pp(Z=0\mid\cF_{r_d})
 &\le\Pp\!\left(
   |Z-\E[Z\mid\cF_{r_d}]|\ge\E[Z\mid\cF_{r_d}]
   \,\middle|\,\cF_{r_d}\right)\\
 &\le\frac{\operatorname{Var}(Z\mid\cF_{r_d})}
          {\E[Z\mid\cF_{r_d}]^2}=o_A(1).
\end{split}
\end{equation}

Because $\mathcal L_d\subseteq\cU_d$, none of the targets indexed by
$\mathcal L_d$ was visited by time $r_d$.  Hence, on
$\mathcal G_{A,d}$, the event $\{\tcov(d)\le T_{A,d}\}$ implies that all
these targets are visited during $(r_d,T_{A,d}]$, which is the event
$\{Z=0\}$.  Since $\mathcal G_{A,d}$ is $\cF_{r_d}$-measurable,
\eqref{eq:binary8} and \eqref{eq:binary12} together yield
\[
\begin{split}
 \Pp\!\left(\tcov(d)\le T_{A,d}\right)
 &\le \Pp(\mathcal G_{A,d}^{\mathrm c})
 +\E\!\left[\one_{\mathcal G_{A,d}}
 \Pp(Z=0\mid\cF_{r_d})\right]\\
 &\le Ce^{-\chi A/2}+o_A(1).
\end{split}
\]
Thus,
$\limsup_{d\to\infty}\Pp(\tcov(d)\le T_{A,d})
\le Ce^{-\chi A/2}$.
Letting $A\to\infty$ proves the lower bound of 
Theorem~\ref{thm:cover2}.
\end{proof}

\subsection{Proof of the upper bound}

From now on $b=2$, so $\beta=2$.  We recall the identities
\begin{equation}\label{eq:binary13}
 1-q_\star=2q_\star^{\lambda/2},\qquad
 R-1=\frac{2q_\star}{1-q_\star},\qquad
 \chi=\lambda+R-1,
\end{equation}
which will be used frequently in this section.

\subsubsection{Hitting estimates for descendants of one particle}\label{sec:binary3}

First, we record the following parameter inequalities. The proof is deferred to Appendix~\ref{app:binary2}.

\begin{lemma}\label{lem:binary1}
Assume $b=2$.  Then
$R-1<\lambda$ and hence $(R-1)^2<\lambda$.  Moreover, there exists $a>1$
such that $a\log a-a+1>\lambda$ and $a\log R<\chi$.
\end{lemma}

We also need the following uniform upper bounds for first-passage times.

\begin{lemma}\label{lem:binary2}
\begin{enumerate}[label=(\roman*),leftmargin=*]
\item\label{part:binary1}
Fix $C<\infty$.  For all $d$ large enough, uniformly for
$0\le D\le C\log\log d$ and
$r\in\{0,\ldots,d\}$ with $r\le C\log\log d$, if
$e^{-\chi D}R^r\le1$, then, for any start $x$ and target $y$ satisfying
$d_{\rm H}(x,y)=d-r$,
\begin{equation}\label{eq:rare1}
 \Pp_x\!\left(\tau_y\le x_\star d-D\right)
 \gg_C e^{-\chi D}R^r.
\end{equation}

\item\label{part:binary2}
There is a constant $\delta>0$ such that, for every
$0\le\ell\le d$ and every $x,y$ satisfying
$d_{\rm H}(x,y)\le d-\ell/2$,
\begin{equation*}
 \Pp_x\!\left(\tau_y\le x_\star d-2\delta\ell\right)\gg1
\end{equation*}
for all $d$ large enough.
\end{enumerate}
\end{lemma}

Lemma \ref{lem:binary2} may be compared to Lemma~\ref{lem:cluster1}, which gives a uniformly positive hitting probability but does not retain the small factor
$e^{-\chi D}R^r$ at an earlier time.  The proof of Lemma \ref{lem:binary2} follows \cite[Proposition~24]{BZ}, employing the second moment method.  The details are given in
Appendix~\ref{app:binary1}.

\subsubsection{The weighted early population}\label{sec:binary4}

The lower bound \eqref{eq:rare1} of Lemma \ref{lem:binary2}(i) shows that if the target is almost antipodal to the origin, every move from the origin to the target scales the first hitting probability up by around a factor of $R>1$. The idea is then to apply this observation to design a weighted second moment method that counts particle hits,\footnote{The weighted approach is inspired by  \cite[Appendix~D.2]{BZ}, where a particle is weighted by its probability of hitting the target in a future time interval.} conditional on an early time that is of order $\asymp\log\log d$. In this section, we introduce the weights. Let $D_v(s)$ be the number of mutations on the ancestral line of particle $v$ until time $s$. For $s$ small, it is very likely that the mutations are all in a favorable direction toward the target. In view of \eqref{eq:rare1}, this motivates the definition

\begin{equation}\label{eq:weights1}
 Z_R(s)=\sum_{v\in V_s}R^{D_v(s)},\qquad
 M_R(s)=e^{-\chi s}Z_R(s).
\end{equation}
The sum $Z_R(s)$ counts the total weight of the early population.  A branching event doubles the weight, while a mutation multiplies the weight
by $R$.  It then follows from \eqref{eq:binary13} that the drift of $Z_R(s)$ is $\chi Z_R(s)$ and that $(M_R(s))$ is a martingale.  


For a mutation event $\xi$ occurring by time $s$, define 
\[
 Z_\xi(s)=
 \sum_{\substack{v\in V_s:\,v\text{ descends from the particle}\\
                   \text{immediately after the mutation }\xi}}
 R^{D_v(s)},\qquad
 H_s^2=\sum_{\xi:\,\xi\text{ occurs by time }s}
 \left(\frac{Z_\xi(s)}{Z_R(s)}\right)^2.
\]
Figure~\ref{fig:weighted1} illustrates these definitions. Our next two results control $H_s$ from above.

\begin{figure}[ht]
\centering
\begin{tikzpicture}[x=1cm,y=1cm,>=stealth,
  every node/.style={font=\footnotesize},
  branch/.style={draw=gray!75!black,line width=1pt},
  mutation/.style={circle,fill=red!75!black,minimum size=4.8pt,inner sep=0pt},
  selected/.style={circle,fill=red!75!black,minimum size=5.2pt,inner sep=0pt},
  leaf/.style={circle,draw,fill=white,minimum size=5.8mm,inner sep=0pt}]
  \node[font=\small\bfseries] at (5.45,4.43) {weighted tree at time $s$};

  \draw[rounded corners=3pt,dashed,draw=orange!85!black,fill=orange!3]
    (2.85,1.72) rectangle (10.45,4.00);
  \node[anchor=north,fill=white,inner sep=1.5pt,font=\scriptsize,
    text=orange!85!black]
    at (6.65,4.00)
    {$Z_{\xi_1}(s)=R^2+R$: sum over the descendants of $\xi_1$};

  \coordinate (rootw) at (.55,1.75);
  \coordinate (splitw) at (2.15,1.75);
  \coordinate (xi1) at (3.25,2.55);
  \coordinate (splitup) at (4.45,2.55);
  \coordinate (xi2) at (5.55,2.95);
  \coordinate (splitlow) at (4.45,.88);
  \coordinate (xi3) at (5.55,.35);

  \draw[branch] (rootw)--(splitw)--(xi1)--(splitup);
  \draw[branch] (splitup)--(xi2)--(7.05,2.95);
  \draw[branch] (splitup)--(7.05,2.05);
  \draw[branch] (splitw)--(splitlow)--(7.05,1.18);
  \draw[branch] (splitlow)--(xi3)--(7.05,.20);

  \fill (rootw) circle (1.7pt);
  \fill (splitw) circle (1.7pt);
  \fill (splitup) circle (1.7pt);
  \fill (splitlow) circle (1.7pt);
  \node[anchor=east] at (.42,1.75) {$\emptyset$};

  \node[mutation] at (xi1) {};
  \node[above=3pt,font=\scriptsize,text=red!75!black] at (xi1)
    {$\xi_1:i_1$};
  \node[selected] at (xi2) {};
  \node[above=3pt,font=\scriptsize,text=red!75!black] at (xi2)
    {$\xi_2:i_2$};
  \node[selected] at (xi3) {};
  \node[below=3pt,font=\scriptsize,text=red!75!black] at (xi3)
    {$\xi_3:i_3$};

  \node[leaf] (wv1) at (7.05,2.95) {$v_1$};
  \node[leaf] (wv2) at (7.05,2.05) {$v_2$};
  \node[leaf] (wv4) at (7.05,1.18) {$v_4$};
  \node[leaf] (wv3) at (7.05,.20) {$v_3$};
  \node[anchor=west] at (7.48,2.95)
    {$R^{D_{v_1}(s)}=R^2$};
  \node[anchor=west] at (7.48,2.05)
    {$R^{D_{v_2}(s)}=R$};
  \node[anchor=west] at (7.48,1.18)
    {$R^{D_{v_4}(s)}=1$};
  \node[anchor=west] at (7.48,.20)
    {$R^{D_{v_3}(s)}=R$};

\end{tikzpicture}
\caption{Solid dots indicate events. The dashed region contains
the descendants from the event $\xi_1$, whose weights sum to $Z_{\xi_1}(s)$.}
\label{fig:weighted1}
\end{figure}
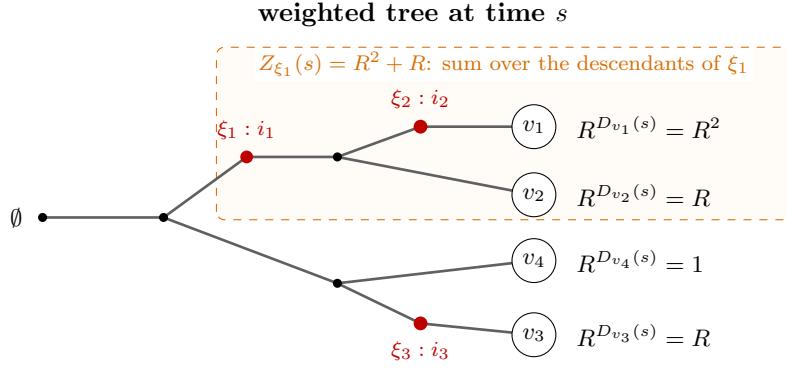


\begin{proposition}
\label{prop:weights1}
The martingale $M_R(s)$ converges in $L^2$ and almost surely to a random
variable $M_R(\infty)$ that is positive almost surely.  Moreover,
\begin{equation}\label{eq:balance1}
 \sup_{s\ge0}\E\!\left[e^{-2\chi s}
             \sum_{\xi:\,\xi\text{ occurs by time }s}Z_\xi(s)^2\right]<\infty.
\end{equation}
\end{proposition}

\begin{proof}
We first prove that $(M_R(s))$ converges in $L^2$. 
Using a similar argument to where we proved $(M_R(s))$ is a martingale, we have
\begin{equation}\label{eq:weighted1}
 \E\!\left[\sum_{v\in V_s}R^{2D_v(s)}\right]
 =\exp((\lambda+R^2-1)s).
\end{equation}
Indeed, a branching event of a particle of weight $z$ increases the sum on
the left by $z^2$, while a mutation increases it by $(R^2-1)z^2$.  Since the jumps of
$Z_R$ are $z$ and $(R-1)z$ with rates $\lambda$ and 1, we have
\begin{equation*}
\E[(M_R(s)-M_R(0))^2]= (\lambda+(R-1)^2)\int_0^s e^{-2\chi t}
 \E\!\left[\sum_{v\in V_t}R^{2D_v(t)}\right]\dd t.
\end{equation*}
By \eqref{eq:weighted1} and
Lemma~\ref{lem:binary1},
\[
\begin{split}
 \sup_{s\ge0}(\lambda+(R-1)^2)\int_0^s e^{-2\chi t}
 \E\!\left[\sum_{v\in V_t}R^{2D_v(t)}\right]\dd t
 &\ll \int_0^\infty e^{(\lambda+R^2-1-2\chi)t}\dd t\\
 &=\int_0^\infty e^{-(\lambda-(R-1)^2)t}\dd t\ll1.
\end{split}
\]
Hence, by the martingale convergence theorem, $(M_R(s))$ converges almost surely and in $L^2$ and
has mean one.

Conditionally on $\cF_n$, the branching property yields
$$M_R(\infty)
=\sum_{v\in V_n}e^{-\chi n}R^{D_v(n)}
M_R^{(v)}(\infty),$$
where $M_R^{(v)}(\infty)$ are independent copies of $M_R(\infty)$. Therefore, 
\[
 \Pp(M_R(\infty)=0)
 =\E\!\left[\Pp(M_R(\infty)=0)^{|V_n|}\right].
\]
 As $|V_n|\to\infty$ almost surely, this shows  $\Pp(M_R(\infty)=0)=0$.


It remains to prove \eqref{eq:balance1}. 
Using in order \eqref{eq:event1}, \eqref{eq:weights1}, and
\eqref{eq:weighted1}, we have
\[
\begin{split}
 \E\!\left[\sum_{\xi:\,\xi\text{ occurs by time }s}Z_\xi(s)^2\right]
 &=\int_0^s R^2\E[Z_R(s-t)^2]\,
 \E\!\left[\sum_{v\in V_t}R^{2D_v(t)}\right]\dd t\\
 &\ll R^2\int_0^s e^{2\chi(s-t)}
 \E\!\left[\sum_{v\in V_t}R^{2D_v(t)}\right]\dd t\\
 &\ll R^2e^{2\chi s}\int_0^s
 e^{-(\lambda-(R-1)^2)t}\dd t\\
 &\ll e^{2\chi s},
\end{split}
\]
where in the last step we have used
\[
 2\chi-(\lambda+R^2-1)=\lambda-(R-1)^2>0,
\]
which follows from  Lemma~\ref{lem:binary1}. Dividing by $e^{2\chi s}$
proves \eqref{eq:balance1}.
\end{proof}

\begin{lemma}\label{lem:martingale2}
For every $\epsilon>0$, there exist $w>0$ and $K<\infty$ such that, for all
sufficiently large $s$,
\begin{equation*}
 \Pp\!\left(M_R(s)\ge w,\ H_s\le K\right)\ge1-\epsilon.
\end{equation*}
\end{lemma}

\begin{proof}
By Proposition~\ref{prop:weights1}, clearly one can choose $w>0$ so that for $s$ large enough, $\Pp(M_R(s)<w)<\epsilon/2$. 
On $\{M_R(s)\ge w\}$, one has $Z_R(s)\ge we^{\chi s}$, and hence by definition
\[
 H_s^2\le w^{-2}e^{-2\chi s}
 \sum_{\xi:\,\xi\text{ occurs by time }s}Z_\xi(s)^2.
\]
Markov's inequality and \eqref{eq:balance1} of Proposition~\ref{prop:weights1} therefore give
\begin{equation*}
 \Pp(H_s>K,\ M_R(s)\ge w)
 \le\frac1{w^2K^2}\E\!\left[e^{-2\chi s}
       \sum_{\xi:\,\xi\text{ occurs by time }s}Z_\xi(s)^2\right]\ll_w K^{-2}.
\end{equation*}
 Choosing $K$ large enough then suffices.
\end{proof}

We have argued above \eqref{eq:weights1} that initial mutations tend to favor a target that is near the antipode. However, there are two sources of adversarial mutations:
\begin{itemize}
    \item a mutation that picks a coordinate where origin and target already agree;
    \item a mutation that picks a coordinate that has already been mutated before.
\end{itemize}
The next two lemmas quantify these two cases, respectively.


\begin{lemma}\label{lem:binary3}
On the event that no coordinate is selected by two distinct
mutation events until time $s$, we have uniformly in $y$ with zero-coordinates 
$S\subseteq\{1,\ldots,d\}$ satisfying $|S|=\ell$,
\begin{equation}\label{eq:overlap1}
 \sum_{v\in V_s}R^{d-d_{\rm H}(X_v(s),y)}
 \ge Z_R(s)R^{\ell-2H_s\sqrt\ell}.
\end{equation}
\end{lemma}

\begin{proof}
  Let $m_S(v)$ be
the number of mutation events on the ancestral line of $v$ whose chosen coordinate lies in
$S$.  Then
\begin{align}
     d-d_{\rm H}(X_v(s),y)=\ell+D_v(s)-2m_S(v).\label{dd}
\end{align}
Give each leaf $v\in V_s$ probability $R^{D_v(s)}/Z_R(s)$.  Averaging $m_S(v)$
with these probabilities gives
\[
\begin{aligned}
 \sum_v\frac{R^{D_v(s)}}{Z_R(s)}m_S(v)
 &=\frac{1}{Z_R(s)}
   \sum_v R^{D_v(s)}
   \sum_{\substack{\xi:\,\xi\text{ occurs by time }s\\
                    \xi\text{ chooses a coordinate in }S\\
                    v\text{ descends from }\xi}}1\\
 &=\frac{1}{Z_R(s)}
   \sum_{\substack{\xi:\,\xi\text{ occurs by time }s\\
                    \xi\text{ chooses a coordinate in }S}}
   \sum_{\substack{v\in V_s:\\v\text{ descends from }\xi}}
   R^{D_v(s)}\\
 &=\sum_{\substack{\xi:\,\xi\text{ occurs by time }s\\
                    \xi\text{ chooses a coordinate in }S}}
   \frac{Z_\xi(s)}{Z_R(s)}.
\end{aligned}
\]
Applying \eqref{dd}, Jensen's inequality, and then Cauchy--Schwarz gives
\[
\begin{aligned}
 \frac{1}{Z_R(s)}\sum_vR^{d-d_{\rm H}(X_v(s),y)}
 &=R^\ell\sum_v\frac{R^{D_v(s)}}{Z_R(s)}
 R^{-2m_S(v)}\\
 &\ge R^{\ell-2\sum_v R^{D_v(s)}m_S(v)/Z_R(s)}\\
 &=R^{\ell-2\sum_{\substack{\xi:\,\xi\text{ occurs by time }s\\
                   \xi\text{ chooses a coordinate in }S}}
                   Z_\xi(s)/Z_R(s)}\\
 &\ge R^{\ell-2H_s\sqrt\ell},
\end{aligned}
\]
where in the last step we used that the sum over mutation events whose coordinates lie in $S$ has at most
$\ell$ terms as coordinates do not repeat.  This proves
\eqref{eq:overlap1}.
\end{proof}



\begin{lemma}\label{lem:binary4}
Fix $0<\kappa<\infty$, and let $a$ be a constant supplied by
Lemma~\ref{lem:binary1}.  Then, with
probability tending to one:
\begin{enumerate}[label=(\roman*),leftmargin=*]
\item no coordinate is chosen by two mutation events before
$\kappa\log\log d$;
\item $\max_{v\in V_{\kappa\log\log d}}D_v(\kappa\log\log d)
\le a\kappa\log\log d$.
\end{enumerate}
\end{lemma}

\begin{proof}
Conditioning on the underlying Yule process, the total mutation count is distributed as 
$\mathrm{Pois}(\int_0^{\kappa\log\log d}|V_t|\dd t)$.  Using $\E[|V_u||V_v|]
 \ll e^{\lambda(u+v)}$, we have
\[
 \E\!\left[\left(\int_0^{\kappa\log\log d}|V_t|\dd t\right)^2\right]
 \ll e^{2\lambda\kappa\log\log d}.
\]
A union bound over pairs of coordinates gives
\[
 \Pp(\text{a coordinate is repeated})
 \le\frac1{2d}
 \E\!\left[\left(\int_0^{\kappa\log\log d}|V_t|\dd t\right)^2\right]
 \ll\frac{e^{2\lambda\kappa\log\log d}}d=o(1).
\]This proves (i). 
The many-to-one identity \eqref{eq:spine1} and the bound \eqref{eq:chernoff1} give
\[
\begin{aligned}
 \E\!\left[\#\{v\in V_{\kappa\log\log d}:
 D_v(\kappa\log\log d)>a\kappa\log\log d\}\right]
 &=e^{\lambda\kappa\log\log d}
 \Pp\!\left(\operatorname{Pois}(\kappa\log\log d)
 >a\kappa\log\log d\right)\\
 &\le e^{-(a\log a-a+1-\lambda)\kappa\log\log d}=o(1).
\end{aligned}
\]
Markov's inequality then proves (ii).
\end{proof}

\subsubsection{Completion of the upper bound}\label{sec:binary5}

\begin{proof}[Proof of the upper bound in Theorem~\ref{thm:cover2}]
We split the proof into three steps. 
\par\smallskip\noindent\emph{Step 1: setting up stages.} 
Fix $\epsilon>0$. For $A\ge0$, put
\begin{equation*}
 T_{A,d}^+=x_\star d+\chi^{-1}\log\log d+A.
\end{equation*}
Our goal is to prove that
\[
 \lim_{A\to\infty}\limsup_{d\to\infty}
 \Pp\!\left(\tcov(d)>T_{A,d}^+\right)=0.
\]
We split the targets on the hypercube into shells by the distance $d-\ell$ from the origin $\z$, where $0\leq\ell\leq d$. To do this, we need some preparations.

Take
$\delta$ from Lemma~\ref{lem:binary2}(ii) satisfying $\delta<\min\{x_\star/2,1/2\}$.  Decreasing
$\delta$ only lengthens the time
allowed in Lemma~\ref{lem:binary2}(ii), so its uniform lower bound
remains valid.  Then choose a fixed $C_{\rm shell}$ so large that
\begin{equation}\label{eq:shell1}
 \lambda\delta C_{\rm shell}>2.
\end{equation}
Take $a$ from Lemma~\ref{lem:binary1}.  Since
$\chi-a\log R>0$, we may choose $\kappa$ so large that
\begin{equation}\label{eq:binary15}
 \kappa(\chi-a\log R)>1+C_{\rm shell}\log R.
\end{equation}

In the rest of the proof, we split the targets according to the number of zeros $\ell$ (so distance from origin is $d-\ell$) into two cases: $0\leq\ell\leq C_{\rm shell}\log\log d$ and $C_{\rm shell}\log\log d<\ell\leq d$. In the first case, we condition on the process at time $\kappa\log\log d$; in the second case, we condition at time $\delta\ell$.


\begin{figure}[ht]
\centering
\begin{tikzpicture}[x=.82\linewidth,y=1cm,>=stealth,
  every node/.style={font=\footnotesize,inner sep=1pt}]
  \node[anchor=east,align=right,font=\scriptsize] at (.15,.72)
    {near-antipode shells\\[-1pt]
     $(0\le\ell\le C_{\rm shell}\log\log d)$};
  \draw[->] (.18,.72)--(.99,.72);
  \foreach \x in {.18,.33,.78,.97}
    \draw (\x,.65)--(\x,.79);
  \node[below=2pt] at (.18,.72) {$0$};
  \node[above=2pt] at (.33,.72) {$\kappa\log\log d$};
  \node[above=2pt,text=gray] at (.78,.72) {$x_\star d$};
  \node[above=2pt] at (.97,.72) {$T_{A,d}^+$};
  \draw[decorate,decoration={brace,mirror,amplitude=3pt}]
    (.18,.43)--(.33,.43)
    node[midway,below=4pt] {$\kappa\log\log d$};
  \draw[decorate,decoration={brace,mirror,amplitude=3pt}]
    (.33,.43)--(.97,.43)
    node[midway,below=4pt] {$T_{A,d}^+-\kappa\log\log d$};

  \node[anchor=east,align=right,font=\scriptsize] at (.15,-.72)
    {other shells\\[-1pt]
     $(C_{\rm shell}\log\log d<\ell\le d)$};
  \draw[->] (.18,-.72)--(.99,-.72);
  \foreach \x in {.18,.36,.61,.78,.97}
    \draw (\x,-.79)--(\x,-.65);
  \node[below=2pt] at (.18,-.72) {$0$};
  \node[above=2pt] at (.36,-.72) {$\delta\ell$};
  \node[above=2pt] at (.61,-.72) {$x_\star d-\delta\ell$};
  \node[below=2pt,text=gray] at (.78,-.72) {$x_\star d$};
  \node[above=2pt] at (.97,-.72) {$T_{A,d}^+$};
  \draw[decorate,decoration={brace,mirror,amplitude=3pt}]
    (.18,-1.01)--(.36,-1.01)
    node[midway,below=4pt] {$\delta\ell$};
  \draw[decorate,decoration={brace,mirror,amplitude=3pt}]
    (.36,-1.01)--(.61,-1.01)
    node[midway,below=4pt] {$x_\star d-2\delta\ell$};
\end{tikzpicture}
\caption{Split of the targets and the corresponding timelines.}
\label{fig:binary3}
\end{figure}
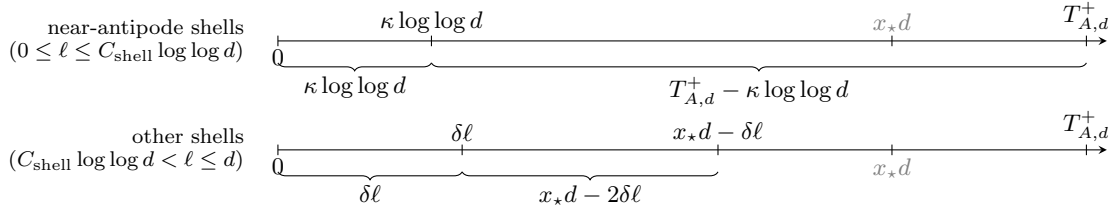

\par\smallskip\noindent\emph{Step 2: targets near the antipode.}
Let $w>0$ and $K<\infty$ be given by Lemma \ref{lem:martingale2} applied with
$\epsilon/3$.  Decreasing $w$ if necessary, we may also assume that
$\Pp(W<2w)<\epsilon/3$.  Intersecting the event therein 
with the event in Lemma~\ref{lem:binary4}, whose probability is
$1-o(1)$, shows that for all large $d$, with probability at least
$1-2\epsilon/3$, coordinates do not repeat and
\begin{equation}\label{eq:binary16}
 M_R(\kappa\log\log d)\ge w,\qquad
 H_{\kappa\log\log d}\le K,\qquad
 \max_v D_v(\kappa\log\log d)\le a\kappa\log\log d.
\end{equation}
Let $\mathcal E_{{\rm near},d}\in\cF_{\kappa\log\log d}$ denote the
intersection of the event that  \eqref{eq:binary16} holds and the event
that no coordinate is selected twice.
 Thus,
$\mathcal E_{{\rm near},d}\cap\{W\ge2w\}$ has probability at least
$1-\epsilon$ for all large $d$. We condition on time $\kappa\log\log d$ and restrict to the event 
$\mathcal E_{{\rm near},d}$.  On this event, by \eqref{eq:binary16}, for a target with
$0\le\ell\le C_{\rm shell}\log\log d$ zeroes, every particle
$v\in V_{\kappa\log\log d}$ satisfies
\[
 d-d_{\rm H}(X_v(\kappa\log\log d),y)
 \le(C_{\rm shell}+a\kappa)\log\log d.
\]
By \eqref{eq:binary15}, uniformly in these particles and targets,
the logarithm of the expression $e^{-\chi D}R^r$ appearing in
Lemma~\ref{lem:binary2}(i) is at most
\[
\begin{split}
 &-\chi\bigl((\kappa-\chi^{-1})\log\log d-A\bigr)+\bigl(d-d_{\rm H}(X_v(\kappa\log\log d),y)\bigr)\log R\\
 &\quad\le\bigl(1+C_{\rm shell}\log R
 -\kappa(\chi-a\log R)\bigr)\log\log d+\chi A
 \longrightarrow-\infty.
\end{split}
\]
Lemma~\ref{lem:binary2}(i), independence, and
Lemma~\ref{lem:binary3} together show that, conditionally on
the early history, the probability that a target $y$ remains unhit
is at most
\begin{equation}\label{eq:near1}
\begin{split}
 \exp\left(-c e^{-\chi((\kappa-\chi^{-1})\log\log d-A)}
       \sum_{v\in V_{\kappa\log\log d}}
       R^{d-d_{\rm H}(X_v(\kappa\log\log d),y)}\right)
 &\le
 \exp(-cw e^{\chi A}\log d\,
       R^{\ell-2K\sqrt\ell}).
\end{split}
\end{equation}
Here $c>0$ does not depend on $A$, $d$. 
Choose $A$ so large that
\[
 cwe^{\chi A}\inf_{j\ge0}\frac{R^{j-2K\sqrt j}}{j+1}>4.
\]
A conditional union bound using
$\binom d\ell\le d^\ell$ and \eqref{eq:near1} shows that the non-coverage probability for vertices with
$0\le\ell\le C_{\rm shell}\log\log d$ zeroes is at most
\[
 \sum_{0\le\ell\le C_{\rm shell}\log\log d}
 d^\ell d^{-4(\ell+1)}=o(1),
\]
 outside of an event with probability at most $\epsilon$.

\par\smallskip\noindent\emph{Step 3: all remaining targets.}
It remains to cover targets satisfying $\ell>C_{\rm shell}\log\log d$.  For each such $\ell$, we condition on time $\delta\ell$. 
Let $B_\ell$ be the number of particles at time $\delta\ell$ whose ancestral line has made
more than $\ell/2$ mutations.  Define the
$\cF_{\delta\ell}$-measurable event
\begin{equation}\label{eq:population1}
 \mathcal E_{\ell,d}:=
 \left\{|V_{\delta\ell}|\ge we^{\lambda\delta\ell}\right\}
 \cap\left\{B_\ell\le\frac12we^{\lambda\delta\ell}\right\},\qquad C_{\rm shell}\log\log d<\ell\le d.
\end{equation}
Next, we claim that
\begin{equation}\label{eq:far1}
 \Pp\!\left(
 \left(\bigcap_{\substack{m\in\mathbb Z\\
 C_{\rm shell}\log\log d<m\le d}}\mathcal E_{m,d}\right)^{\!\mathrm c}
 \cap\{W\ge2w\}\right)=o(1).
\end{equation}
Indeed, on $\{W\ge2w\}$, almost-sure convergence \eqref{eq:yule1}  gives
$|V_t|\ge we^{\lambda t}$ for all sufficiently large real $t$.  Moreover, the
many-to-one identity \eqref{eq:spine1} gives
\[
 \E[B_\ell]=e^{\lambda\delta\ell}
 \Pp\!\left(\operatorname{Pois}(\delta\ell)>\frac{\ell}{2}\right).
\]
Applying Markov's inequality and the bound \eqref{eq:chernoff1}, we have
\begin{equation}\label{eq:far2}
 \Pp\!\left(B_\ell>\frac12we^{\lambda\delta\ell}\right)
 \le\frac2w\Pp\!\left(\operatorname{Pois}(\delta\ell)>\frac{\ell}{2}\right)
 \le\frac2w\exp\!\left(-\left(
 \frac12\log\frac1{2\delta}-\frac12+\delta\right)\ell\right).
\end{equation}
Because $\delta<1/2$, one has
$\frac12\log\frac1{2\delta}-\frac12+\delta>0$.  The right-hand side of
\eqref{eq:far2} is thus summable in $\ell$.  Combining these bounds
 proves
\eqref{eq:far1}.

On $\mathcal E_{\ell,d}$, each particle $v$ not counted by $B_\ell$
satisfies $d-d_{\rm H}(X_v(\delta\ell),y)\ge\ell/2$ for every target $y$ with
$\ell$ zeroes. Also, there are at
least $|V_{\delta\ell}|-B_\ell\ge we^{\lambda\delta\ell}/2$ such particles by
\eqref{eq:population1}.
Lemma~\ref{lem:binary2}(ii) then gives that on the event $\mathcal E_{\ell,d}$,
\begin{equation}\label{eq:far3}
 \Pp\!\left(y\text{ is unhit by }x_\star d-\delta\ell
       \mid\cF_{\delta\ell}\right)
 \le\exp(-cwe^{\lambda\delta\ell}).
\end{equation}
Note that the time on the left-hand side of \eqref{eq:far3} is earlier than
$T_{A,d}^+$ for every $A\ge0$, as
$x_\star d-\delta\ell\le x_\star d<T_{A,d}^+$.  By
\eqref{eq:shell1} and Lemma~\ref{lem:shell1} below,
\begin{equation*}
 \sum_{C_{\rm shell}\log\log d<\ell\leq d}\binom d\ell
       e^{-cwe^{\lambda\delta\ell}}
 \le\sum_{C_{\rm shell}\log\log d<\ell\leq d}
       \exp\!\left(\ell\log\!\left(\frac{ed}{\ell}\right)
       -cwe^{\lambda\delta\ell}\right)=o(1).
\end{equation*}
In other words, the non-coverage probability for vertices with
$C_{\rm shell}\log\log d<\ell\leq d$ zeroes is at most $o(1)$, outside of an event with probability at most $\epsilon$.

Combining the above two steps and sending $\epsilon\to 0$ completes the proof.
\end{proof}

\section{Proofs of the corollaries}\label{sec:corollaries1}

 For the readers' convenience, we recall the statements of Corollaries~\ref{cor:mutation1} and~\ref{cor:occupation1}.

\begin{recalledmutation}
Fix an integer $b\ge2$ and a branching rate $\lambda>0$. For each fixed total mutation rate $\mu\ge\lambda$, let $\tcov^{(\mu)}(d)$ denote the cover time. Then $\tcov^{(\mu)}(d)/d\xrightarrow{\Pp}a_\mu$ as $d\to\infty$, where $a_\mu>0$ is the unique solution of
\[
 \lambda a_\mu-\log b
 +\log\bigl(1-e^{-\beta\mu a_\mu}\bigr)=0,
 \qquad \beta=\frac{b}{b-1}.
\]
The function $\mu\mapsto a_\mu$ is strictly decreasing on $[\lambda,\infty)$, and $a_\mu\downarrow(\log b)/\lambda$ as $\mu\to\infty$.
\end{recalledmutation}

\begin{proof}[Proof of Corollary~\ref{cor:mutation1}]
With the time change $s=\mu t$, the mutation rate becomes one and the branching rate becomes $\lambda/\mu\in(0,1]$. The corresponding root in \eqref{eq:rate1} is $\mu a_\mu$, so Theorems~\ref{thm:cover1} and~\ref{thm:cover2} give $\tcov^{(\mu)}(d)/d\xrightarrow{\Pp}a_\mu$ as $d\to\infty$. Writing $a_\mu'=\dd a_\mu/\dd\mu$, implicit differentiation of \eqref{eq:mutation1} gives
\[
 0=\lambda a_\mu'
   +\frac{\beta e^{-\beta\mu a_\mu}}
          {1-e^{-\beta\mu a_\mu}}
     (a_\mu+\mu a_\mu').
\]
Multiplying by $e^{\beta\mu a_\mu}-1$ and rearranging yields
\[
 \bigl(\lambda(e^{\beta\mu a_\mu}-1)+\beta\mu\bigr)a_\mu'
 =-\beta a_\mu.
\]
Since $a_\mu>0$ and $\mu\ge\lambda>0$, the coefficient of $a_\mu'$ is positive. Hence
\[
 \frac{\dd a_\mu}{\dd\mu}=
a_\mu' =-\frac{\beta a_\mu}
         {\lambda(e^{\beta\mu a_\mu}-1)+\beta\mu}<0.
\]
Finally, \eqref{eq:mutation1} implies $a_\mu>(\log b)/\lambda$, so $e^{-\beta\mu a_\mu}\to0$ as $\mu\to\infty$. Substituting into \eqref{eq:mutation1} gives $a_\mu\to\log b/\lambda$.
\end{proof}

\begin{recalledoccupation}
Let $b\ge2$ and $\lambda\in(0,1]$. Then
\[
 \sigma_d=
 \begin{cases}
  x_\star d+\lambda^{-1}\log d+O_{\Pp}(1),&b>2,\\
  x_\star d+\chi^{-1}\log\log d+O_{\Pp}(1),&b=2.
 \end{cases}
\]
\end{recalledoccupation}

\begin{proof}[Proof of Corollary~\ref{cor:occupation1}]
Let $m_d$ denote the deterministic centering in \eqref{eq:cover1} or \eqref{eq:cover2}. By definition, $\sigma_d\ge\tcov(d)$, so the lower bound on $\sigma_d$ follows from Theorems~\ref{thm:cover1} and~\ref{thm:cover2}. For the upper bound, it suffices to make the probability that some vertex is empty at time $m_d+A$ arbitrarily small by choosing $A$ large and then $d$ large. 

Recall that $N_y(t)$ denotes the number of particles at a target vertex $y$ at time $t$. In the following, occupation of $y$ at time $t$ means that $N_y(t)>0$. Intuitively, asserting a lower bound for occupation is stronger than asserting a lower bound on the probability of hitting by time $t$, but the proofs of the key inputs (Lemma~\ref{lem:cluster1} for $b>2$ and Lemma~\ref{lem:binary2} for $b=2$) already provide the lower bound for occupation, which allows us to conclude an upper bound for the simultaneous occupation time using essentially the same arguments as bounding the cover time from above. We detail the necessary changes below separately for the two cases.

\textbf{Case (i): $b>2$}. The proof of Lemma~\ref{lem:cluster1} in Appendix~\ref{app:cluster1} applies the second-moment inequality to $N_y(x_\star d)$ and establishes
\begin{equation}\label{eq:occupation3}
 \Pp_x\!\left(N_y(x_\star d)>0\right)\ge p_0
\end{equation}
for all large $d$, uniformly over all starting and target vertices $x,y$. For $A\ge0$, put
\[
 s=\lambda^{-1}\log d+A,\qquad
 T=s+x_\star d=m_d+A,\qquad Z_s=|V_s|.
\]
Conditional on the history $\cF_s$ up to time $s$, the $Z_s$ descendant families are independent. Since $T-s=x_\star d$, \eqref{eq:occupation3} implies that each family has a probability at least $p_0$ of having a particle at any prescribed target $y$ at the common time $T$, regardless of its starting position. Therefore, $\Pp(N_y(T)=0\mid\cF_s)\le(1-p_0)^{Z_s}$. If every vertex is occupied at time $T$, then $\sigma_d\le T$, so a union bound over the $b^d$ vertices gives
\begin{equation}\label{eq:occupation2}
\begin{aligned}
 \Pp(\sigma_d>T\mid\cF_s)
 &\le \sum_y\Pp(N_y(T)=0\mid\cF_s)
 \le b^d(1-p_0)^{Z_s}
 \le \exp(d\log b-p_0Z_s).
\end{aligned}
\end{equation}
Choose $K>(\log b)/p_0$. The geometric law~\eqref{eq:yule2} gives
\[
 \Pp(Z_s<Kd)\le Kd\,e^{-\lambda s}=Ke^{-\lambda A}.
\]
Splitting according to whether $Z_s<Kd$ in \eqref{eq:occupation2}, we obtain
\[
 \Pp(\sigma_d>m_d+A)
 \le Ke^{-\lambda A}+\exp\bigl(-(p_0K-\log b)d\bigr).
\]
The second term tends to zero as $d\to\infty$, and the first tends to zero as $A\to\infty$. This proves the upper bound for $b>2$.

\textbf{Case (ii): $b=2$}.  In this case, recall that the unique antipode of $\z$ is $(1,\ldots,1)$, and a target with $\ell$ zeros lies at Hamming distance $\ell$ from this antipode. The proof of Lemma~\ref{lem:binary2} in Appendix~\ref{app:binary1} applies the second-moment inequality to $N_y(t)$ and establishes
\begin{equation}\label{eq:occupation4}
 \Pp_x\!\left(N_y(x_\star d-D)>0\right)
 \gg_C e^{-\chi D}R^r
\end{equation}
for each fixed $C<\infty$ and all large $d$, uniformly when $r=d-d_{\rm H}(x,y)$, $0\le D,r\le C\log\log d$, and $e^{-\chi D}R^r\le1$. Here $R=(1+e^{-2x_\star})/(1-e^{-2x_\star})$, as in \eqref{eq:R}. The proof also establishes
\begin{equation}\label{eq:occupation5}
 \Pp_x\!\left(N_y(x_\star d-2\delta\ell)>0\right)\gg1
\end{equation}
for all large $d$, uniformly over $0\le\ell\le d$ and $x,y$ satisfying $d_{\rm H}(x,y)\le d-\ell/2$, where $\delta>0$ is sufficiently small and fixed. 

For targets at Hamming distance $0\le\ell\le C_{\rm shell}\log\log d$ from the antipode, with $C_{\rm shell}$ chosen below, the argument in Step~2 of the proof in Section~\ref{sec:binary5}, leading to \eqref{eq:near1}, applies also to the occupation event at time $m_d+A$ using \eqref{eq:occupation4}. For the remaining targets at Hamming distance $C_{\rm shell}\log\log d<\ell\le d$ from the antipode, we modify Step~3 of that proof, using the occupation bound~\eqref{eq:occupation5}. For a target at distance $\ell$ from the antipode, that step conditions at time $\delta\ell$ and obtains a hitting bound by time $x_\star d-\delta\ell$; see \eqref{eq:far3}. Choose $\delta>0$ sufficiently small as in the proof of Lemma~\ref{lem:binary2}(ii), and then choose $C_{\rm shell}$ large enough such that
\begin{equation}\label{eq:occupation1}
 \lambda\delta C_{\rm shell}>2,
 \qquad
 2\delta+\frac{1}{\chi C_{\rm shell}}<\frac14.
\end{equation}
Put $T=m_d+A$. For a target at distance $\ell>C_{\rm shell}\log\log d$ from the antipode, we condition on time
\[
 u_\ell=\chi^{-1}\log\log d+A+2\delta\ell,
 \qquad
 T-u_\ell=x_\star d-2\delta\ell.
\]
For each fixed $A\ge0$, \eqref{eq:occupation1} and $\ell>C_{\rm shell}\log\log d$ imply that uniformly over these shells,
\[
 \frac{u_\ell}{\ell}
 =2\delta+\frac{\chi^{-1}\log\log d+A}{\ell}
 \le 2\delta+\frac{1}{\chi C_{\rm shell}}
       +\frac{A}{C_{\rm shell}\log\log d}
 <\frac14
\]
for all sufficiently large $d$, since $A/(C_{\rm shell}\log\log d)\to0$. Thus $u_\ell\le\ell/4$.

Fix $w>0$. Let $\mathcal E_{\ell,d}$ be the event that at time $u_\ell$ there are at least $\tfrac12w e^{\lambda u_\ell}$ particles whose ancestral lines have made at most $\ell/2$ mutations. Since $\min_\ell u_\ell\to\infty$ and
\(
 \Pp\!\left(\operatorname{Pois}(u_\ell)>\ell/2\right)
 \le e^{-c\ell},
\)
the Yule and many-to-one arguments that lead to \eqref{eq:far1} and \eqref{eq:far2} give
\[
 \Pp\!\left(
   \bigcup_{\ell>C_{\rm shell}\log\log d}^{d}
   \mathcal E_{\ell,d}^{\,\mathrm c}\right)
 \le \Pp(W<2w)+o(1).
\]
On $\mathcal E_{\ell,d}$, each of these particles is within distance $d-\ell/2$ of every target at Hamming distance $\ell$ from the antipode. The occupation bound~\eqref{eq:occupation5}, applied to their independent descendant clusters, therefore gives
\[
 \Pp\!\left(N_y(T)=0\mid\cF_{u_\ell}\right)
 \le \exp\!\left(-cw e^{\lambda u_\ell}\right)
 \le \exp\!\left(-cw e^{2\lambda\delta\ell}\right).
\]
By Lemma~\ref{lem:shell1}, summing over these targets gives an upper bound of $\Pp(W<2w)+o(1)$. Sending $w\downarrow0$ and combining this with the near-antipode argument proves the required upper bound for $b=2$.
\end{proof}

\section*{Acknowledgments}
The author has supplied the key mathematical arguments to generative AI. During the revision process, GPT 5.5 and 5.6 were used to prepare figures and simulations, to expand omitted details, to review literature, and to produce neater proofs, particularly the elementary lemmas in the appendices. Every AI-assisted argument has been reviewed and verified independently by the author. The author also gratefully acknowledges support from the Jump Trading Fellowship.

\appendix
\section{First- and second-moment computations}
\label{app:moments1}

This appendix proves Lemmas~\ref{lem:endpoint1}, \ref{lem:cluster1},
\ref{lem:cluster2}, and~\ref{lem:binary2} involving first-
and second-moment calculations.  Recall the notation in \eqref{eq:rate1},
\eqref{eq:constants1}, \eqref{eq:kernel1}, \eqref{eq:arrival1}, \eqref{eq:predecessor1}, and \eqref{eq:R}.
The proofs of Lemmas~\ref{lem:endpoint1} and~\ref{lem:cluster1} apply to
fixed $b\ge2$, the proof of Lemma~\ref{lem:cluster2} does not depend on
$b$, and the proof of Lemma~\ref{lem:binary2} takes $b=2$.

\subsection{Proof of Lemma~\ref{lem:endpoint1}}\label{app:endpoint1}

\begin{lemma}\label{lem:endpoint2}
Fix $C_{\rm ov}<\infty$ and $\vartheta>R$.  There exists
$\epsilon=\epsilon(\vartheta)\in(0,x_\star/2)$ such that, for all
 $d$ large enough, uniformly for
$|s-x_\star d|\le\epsilon d$ and $\ell\le C_{\rm ov}\log d$,
\[
 \left(\frac{A_s}{B_s}\right)^\ell\le\vartheta^\ell,
 \qquad H_{d,\ell}(s)\asymp1,
 \qquad |H_{d,\ell}'(s)|\ll d^{-1},
\]
and the integrand in \eqref{eq:arrival1} satisfies
\begin{align}
     \frac{\mathrm d}{\mathrm d s}\log\!\left(
 e^{\lambda s}b^{-d}B_s^d
 \left(\frac{A_s}{B_s}\right)^\ell H_{d,\ell}(s)
 \right)\ge\frac\lambda2.\label{eq:growth1}
\end{align}
\end{lemma}

\begin{proof}
First, observe that
\begin{equation}\label{eq:derivative1}
 \frac{\mathrm d}{\mathrm d s}\log\!\left(e^{\lambda s}b^{-d}B_s^d\right)
 =\lambda+\frac{\beta e^{-\beta s/d}}{1-e^{-\beta s/d}}>0.
\end{equation}
Next,  for $|s-x_\star d|=O(\log d)$,\footnote{The same observation will be used in both parts of
Lemma~\ref{lem:endpoint1}.}
\begin{equation}\label{eq:endpoint3}
 e^{\lambda s}b^{-d}B_s^d
 =\exp\!\left(d\Phi\!\left(\frac{s}{d}\right)\right)
 =\exp\!\left(\chi(s-x_\star d)
 +O\!\left(\frac{(s-x_\star d)^2}{d}\right)\right)\asymp \exp\left(\chi(s-x_\star d)
 \right).
\end{equation}
Also, $A_s/B_s=R+O(|s-x_\star d|/d)$.  Fix $\vartheta>R$ as in the
statement.  Since $A_{x_\star d}/B_{x_\star d}=R$, continuity allows us
to choose $0<\epsilon=\epsilon(\vartheta)<x_\star/2$ such that
$$
 \sup_{|x-x_\star|\le\epsilon}
 \frac{1+(b-1)e^{-\beta x}}{1-e^{-\beta x}}<\vartheta.
$$
Consequently, for all large $d$,
\begin{equation}\label{eq:endpoint4}
 \left(\frac{A_s}{B_s}\right)^\ell\le\vartheta^\ell
 \qquad
 \text{ for }~~|s-x_\star d|\le \epsilon d~~\text{ and }~~ \ell\le C_{\rm ov}\log d.
\end{equation}
For $\epsilon$ small, we have that, on
$|s-x_\star d|\le\epsilon d$, the quantities $e^{-\beta s/d}$, $A_s$, and $B_s$ are
bounded above and below by positive constants.  Direct differentiation
gives
\begin{equation}\label{eq:endpoint5}
 \left|\frac{\mathrm d}{\mathrm d s}\log\!\left(\frac{A_s}{B_s}\right)\right|
 =\left|-\frac{\beta e^{-\beta s/d}}{d}
 \left(\frac{b-1}{A_s}+\frac{1}{B_s}\right)\right|\ll \frac{1}{d}.
\end{equation}
 On the same interval,
$A_s/B_s$ and $B_s/A_s$ are bounded above and below by positive
constants.  Moreover, $d-\ell\ge d/2$ for all sufficiently large $d$, and
direct differentiation of \eqref{eq:predecessor1} gives
\[
 \left|H_{d,\ell}'(s)\right|
 = \left|\frac\ell d\left(\frac{B_s}{A_s}\right)'
 +\frac{d-\ell}{d(b-1)}\left(\frac{A_s}{B_s}\right)'\right|\ll \frac{1}{d}.
\]
Therefore,
\begin{equation}\label{eq:endpoint6}
 H_{d,\ell}(s)\asymp1,
 \qquad |H_{d,\ell}'(s)|\ll d^{-1}.
\end{equation}
Hence, $|\frac{\mathrm d}{\mathrm d s}\log H_{d,\ell}(s)|\ll d^{-1}$.  Writing the
integrand in \eqref{eq:arrival1} as
$e^{\lambda s}b^{-d}B_s^d(A_s/B_s)^\ell H_{d,\ell}(s)$, equations
\eqref{eq:derivative1}, \eqref{eq:endpoint5},
and~\eqref{eq:endpoint6} give
$$
 \frac{\mathrm d}{\mathrm d s}\log\!\left(
 e^{\lambda s}b^{-d}B_s^d
 \left(\frac{A_s}{B_s}\right)^\ell H_{d,\ell}(s)
 \right)
 \ge \lambda-\frac{C\ell}{d}-\frac Cd
 \ge\frac\lambda2
$$
for all large $d$, uniformly when
$|s-x_\star d|\le\epsilon d$ and
$\ell\le C_{\rm ov}\log d$.  This proves the lemma.
\end{proof}

\begin{proof}[Proof of Lemma~\ref{lem:endpoint1}]
\smallskip\noindent\emph{Part~\textup{(i)}.}
Fix $C_{\rm ov},C_D<\infty$ and $\vartheta>R$, and take $\epsilon$ from
Lemma~\ref{lem:endpoint2}.  Decrease $\epsilon$, if necessary,
so that $e^{\lambda\epsilon}\beta\epsilon/b<1$.  We split the domain of integration 
in \eqref{eq:arrival1} into the three intervals
\[
 [0,\epsilon d],\qquad
 [\epsilon d,(x_\star-\epsilon)d],\qquad
 [(x_\star-\epsilon)d,x_\star d-D].
\]
We estimate them in this order.

\smallskip\noindent\emph{Case 1: $0\le s\le\epsilon d$.}
For $0<s\le\epsilon d$, \eqref{eq:predecessor1} and
$\ell\le C_{\rm ov}\log d$ give
$H_{d,\ell}(s)\ll A_s/B_s$ for all large $d$.  Using $A_s\le b$ and
$B_s\le\beta s/d\le\beta\epsilon$ gives
\[
\begin{split}
 e^{\lambda s}b^{-d}A_s^\ell B_s^{d-\ell}H_{d,\ell}(s)
 &\ll e^{\lambda\epsilon d}b^{-d}
 A_s^{\ell+1}B_s^{d-\ell-1}\\
 &\ll\left(\frac{e^{\lambda\epsilon}\beta\epsilon}{b}\right)^d
       \left(\frac{b}{\beta\epsilon}\right)^{\ell+1}.
\end{split}
\]
The second factor $(b/(\beta\epsilon))^{\ell+1}$ is at most a
fixed power of $d$ because $\ell\le C_{\rm ov}\log d$.  Thus,
multiplication by the interval length $\epsilon d$ leaves an exponentially
small bound:
\[
 \int_0^{\epsilon d}e^{\lambda s}b^{-d}A_s^\ell
 B_s^{d-\ell}H_{d,\ell}(s)\dd s
 \ll \epsilon d
 \left(\frac{e^{\lambda\epsilon}\beta\epsilon}{b}\right)^d
 \left(\frac{b}{\beta\epsilon}\right)^{C_{\rm ov}\log d+1}
 \le e^{-cd}
\]
for some $c>0$ and all large $d$.

\smallskip\noindent\emph{Case 2: $\epsilon d\le s\le
(x_\star-\epsilon)d$.}
Since $\Phi$ is strictly increasing and $\Phi(x_\star)=0$, for this fixed
$\epsilon>0$ there is $c>0$ such that
\[
 \sup_{\epsilon d\le s\le(x_\star-\epsilon)d}
 e^{\lambda s}b^{-d}B_s^d\le e^{-cd}.
\]
For $s\in[\epsilon d,(x_\star-\epsilon)d]$, $A_s/B_s\asymp1$ and $H_{d,\ell}(s)\ll1$.  Since
$\ell\le C_{\rm ov}\log d$, the contribution of this interval is
\[
 \int_{\epsilon d}^{(x_\star-\epsilon)d}
 e^{\lambda s}b^{-d}A_s^\ell B_s^{d-\ell}
 H_{d,\ell}(s)\dd s
 \ll d^{O(1)}e^{-cd}\le e^{-cd/2}
\]
for some $c>0$ and all large $d$.

\smallskip\noindent\emph{Case 3: $(x_\star-\epsilon)d\le s\le
x_\star d-D$.}
Integrating \eqref{eq:growth1} of Lemma~\ref{lem:endpoint2}
from $s$ to $x_\star d-D$ shows that
\begin{align*}
 &e^{\lambda s}b^{-d}B_s^d
 \left(\frac{A_s}{B_s}\right)^\ell H_{d,\ell}(s)\\
 &\quad\ll e^{\lambda(x_\star d-D)}b^{-d}B_{x_\star d-D}^d
 \left(\frac{A_{x_\star d-D}}{B_{x_\star d-D}}\right)^\ell
 H_{d,\ell}(x_\star d-D)
 e^{-\lambda(x_\star d-D-s)/2}.
\end{align*}
Consequently,
\begin{equation*}
\begin{split}
 &\int_{(x_\star-\epsilon)d}^{x_\star d-D}
 e^{\lambda s}b^{-d}A_s^\ell B_s^{d-\ell}H_{d,\ell}(s)\dd s\\
 &\quad\ll
 e^{\lambda(x_\star d-D)}b^{-d}
 A_{x_\star d-D}^\ell B_{x_\star d-D}^{d-\ell}
 H_{d,\ell}(x_\star d-D).
\end{split}
\end{equation*}
At $s=x_\star d-D$, equations \eqref{eq:endpoint3},
\eqref{eq:endpoint4}, and~\eqref{eq:endpoint6} together give the bound
\[
\begin{split}
 &e^{\lambda(x_\star d-D)}b^{-d}
 A_{x_\star d-D}^\ell B_{x_\star d-D}^{d-\ell}
 H_{d,\ell}(x_\star d-D)
 \ll \vartheta^\ell e^{-\chi D}.
\end{split}
\]
For each $c>0$, $ e^{-cd/2}=o(\vartheta^\ell e^{-\chi D})$. Therefore, combining the above three cases yields the proof of (i). 

\smallskip\noindent\emph{Part~\textup{(ii)}.}
Let $\epsilon$ be as in Lemma~\ref{lem:endpoint2}, applied with $C_{\rm ov}=1$ and
$\vartheta=R+1$, and decrease the resulting $\epsilon$, if necessary, as
in Cases 1 and 2 above.   Since $\ell$ is fixed, $\ell\le\log d$ for all large
$d$, so the lemma applies.
For \eqref{eq:endpoint2}, write
$s=x_\star d+S_d-v$. Similarly to \eqref{eq:endpoint3}, for each fixed $M<\infty$, uniformly for
$0\le v\le M$,
\begin{equation}\label{eq:endpoint8}
 d\Phi\!\left(\frac{s}{d}\right)=\chi(S_d-v)
 +O\!\left(\frac{(|S_d|+M)^2}{d}\right)
           =\chi(S_d-v)+o(1),
\qquad
 \frac{A_s}{B_s}=R+o(1).
\end{equation}
Since $\ell$ is fixed, \eqref{eq:predecessor1} and
\eqref{eq:endpoint8} give, uniformly for
$0\le v\le M$ as $d\to\infty$,
\begin{equation}\label{eq:predecessor2}
\begin{split}
 \left(\frac{A_s}{B_s}\right)^\ell H_{d,\ell}(s)
 &\to R^\ell\frac{R+b-2}{b-1}\\
 &=
 \frac{1+(b-1)q_\star+(b-2)(1-q_\star)}
 {(b-1)(1-q_\star)}\,R^\ell .
\end{split}
\end{equation}
Indeed, the first term in \eqref{eq:predecessor1} vanishes because
$\ell/d\to0$, while $(d-\ell)/d\to1$.  For the $\epsilon$ fixed above,
\eqref{eq:endpoint8} and
\eqref{eq:predecessor2} give, for each fixed $v\ge0$, as $d\to\infty$,
\[
\begin{split}
 &e^{-\chi S_d}e^{\lambda s}b^{-d}
 A_s^\ell B_s^{d-\ell}H_{d,\ell}(s)\xrightarrow[]{}
 e^{-\chi v}R^\ell\frac{R+b-2}{b-1},
 \qquad s=x_\star d+S_d-v.
\end{split}
\]
At $v=0$, \eqref{eq:endpoint8}
and~\eqref{eq:predecessor2} show that the normalized
integrand is $\ll1$.  By \eqref{eq:growth1} of 
Lemma~\ref{lem:endpoint2},
\[
\begin{split}
 &e^{-\chi S_d}e^{\lambda s}b^{-d}
 A_s^\ell B_s^{d-\ell}H_{d,\ell}(s)\\
 &\quad\le
 e^{-\chi S_d}e^{\lambda(x_\star d+S_d)}b^{-d}
 A_{x_\star d+S_d}^\ell B_{x_\star d+S_d}^{d-\ell}
 H_{d,\ell}(x_\star d+S_d)e^{-\lambda v/2}\\
 &\quad\ll e^{-\lambda v/2},
 \qquad s=x_\star d+S_d-v,
\end{split}
\]
whenever $0\le v\le \epsilon d+S_d$.  Integrating this bound and using the
estimates proved in Cases 1 and 2 on the earlier intervals
$[0,\epsilon d]$ and $[\epsilon d,(x_\star-\epsilon)d]$ shows that
\[
\begin{split}
 e^{-\chi S_d}\int_0^{x_\star d+S_d-M}
 e^{\lambda s}b^{-d}A_s^\ell B_s^{d-\ell}
 H_{d,\ell}(s)\dd s
 &\ll \int_M^\infty e^{-\lambda v/2}\dd v
      +e^{-cd+\chi|S_d|}\\
 &\ll e^{-\lambda M/2}+e^{-cd/2}
\end{split}
\]
for all large $d$, uniformly for $S_d=o(d^{1/2})$.  
On each fixed interval
$0\le v\le M$, \eqref{eq:endpoint8} and
\eqref{eq:predecessor2} allow us to apply dominated
convergence.
Letting first $d\to\infty$ and then $M\to\infty$ gives the factor
$\int_0^\infty e^{-\chi v}\dd v=\chi^{-1}$ and proves
\eqref{eq:endpoint2}.
\end{proof}

\subsection{Proof of Lemma~\ref{lem:cluster1}}\label{app:cluster1}

For $0<q\le p\le1$, define the functions
\begin{align}
 K_+(p,q)&=1+
 \frac{(b-1)q^2(1-p)(1+(b-1)p)}
 {p^2(1+(b-1)q)^2},\label{eq:kernel2}\\
 K_-(p,q)&=1+
 \frac{q^2(1-p)(b-1+p)}{p^2(1-q)^2}.
 \label{eq:kernel3}
\end{align}
Recall $\Phi(x_\star)=0$, which is equivalent to
\begin{equation}\label{eq:root1}
 q_\star^{\lambda/\beta}=\frac{1-q_\star}{b}.
\end{equation}
 For $0\le u\le1$, define $p=q_\star^u$ and 
\begin{equation}\label{eq:convexity1}
 f_\pm(u)=\frac{\lambda}{\beta}\log(q_\star^u)
          +\log K_\pm(q_\star^u,q_\star).
\end{equation}

The following elementary estimate is useful in the
second-moment calculation in the proof of Lemma~\ref{lem:cluster1}. 

\begin{lemma}\label{lem:cluster3}
The functions $f_+$ and $f_-$ are strictly convex, satisfying
\begin{equation}\label{eq:convexity2}
 f_+(0)=f_-(0)=0,\qquad f_-(1)=0,\qquad
 f_+(1)=-\log R.
\end{equation}
Moreover, there is a constant $c>0$ such that
\begin{align*}
 f_+(u)&\le-(\log R)u, \qquad0\le u\le1;\\
 f_-(u)&\le-cu, \qquad~\,\qquad0\le u\le\tfrac12;\\
 f_-(u)&\le-c(1-u), \qquad\tfrac12\le u\le1.
\end{align*}
\end{lemma}

\begin{proof}
The identities in  \eqref{eq:convexity2} follow from \eqref{eq:root1} and direct computation. 
For $p=q_\star^u$, let
\[
 \delta_+=\frac{(b-1)q_\star^2}
 {(1+(b-1)q_\star)^2},
\]
so \eqref{eq:kernel2} becomes
\[
 K_+(p,q_\star)=\delta_+p^{-2}+(b-2)\delta_+p^{-1}
 +1-(b-1)\delta_+.
\]
A direct differentiation then gives
\[
\begin{aligned}
 \frac{\mathrm d^2}{\mathrm d(\log(p^{-1}))^2}\log K_+(p,q_\star)
 &=\frac{p^{-1}}{K_+(p,q_\star)^2}\biggl(
 (b-2)\delta_+^2p^{-2}
 +4\delta_+\bigl(1-(b-1)\delta_+\bigr)p^{-1}\\
 &\hspace{35mm}
 +(b-2)\delta_+\bigl(1-(b-1)\delta_+\bigr)
 \biggr)>0,
\end{aligned}
\]
where we have used $(b-1)\delta_+=((b-1)q_\star/(1+(b-1)q_\star))^2<1$. 
Since $(\lambda/\beta)\log p$ is affine in $\log(p^{-1})$,
\eqref{eq:convexity1} shows that $f_+$ is strictly convex.

For $K_-$, put $\delta_-=q_\star^2/(1-q_\star)^2$.
Since $0<q_\star<1$ and $0<\lambda/\beta\le1/\beta<1$, 
\eqref{eq:root1} gives
\[
 q_\star<q_\star^{\lambda/\beta}
 =\frac{1-q_\star}{b},
 \qquad q_\star<\frac1{b+1}<\frac1b.
\]
Consequently, $0<\delta_-<1$.
Then \eqref{eq:kernel3} becomes
$$K_-(p,q_\star)=(b-1)\delta_-p^{-2}-(b-2)\delta_-p^{-1}+1-\delta_-.$$  
A direct differentiation yields
\[
 \frac{\mathrm d^2}{\mathrm d(\log(p^{-1}))^2}\log K_-(p,q_\star)
 =\frac{p^{-1}H(p^{-1})}{K_-(p,q_\star)^2},
\]
where
\[
 H(v)=-(b-1)(b-2)\delta_-^2v^2
      +4(b-1)\delta_-(1-\delta_-)v
      -(b-2)\delta_-(1-\delta_-).
\]
If $b=2$, then $H(v)=4\delta_-(1-\delta_-)v>0$.  If $b>2$, then $H$ is a
concave quadratic function on $[1,q_\star^{-1}]$, with endpoint values
\begin{align*}
 H(1)&=\delta_-(3b-2-b^2\delta_-),\\
 H(q_\star^{-1})
 &=\frac{q_\star^2}{(1-q_\star)^4}
 \left(\frac{4(b-1)(1-2q_\star)}{q_\star}
 -(b-2)(b-2q_\star)\right).
\end{align*}
To see $H(1)>0$, note that $q_\star<1/b$ gives
$\delta_-<1/(b-1)^2$, while $b^2/(b-1)^2\le4\le3b-2$.  To see $H(q_\star^{-1})>0$,  a direct differentiation and $q_\star<1/b$ give
$$\frac{\mathrm d}{\mathrm d q}
 \left(\frac{4(b-1)(1-2q)}q-(b-2)(b-2q)\right)
 =-\frac{4(b-1)}{q^2}+2(b-2)<0,$$
 so
\[
\begin{split}
&\frac{4(b-1)(1-2q_\star)}{q_\star}
 -(b-2)(b-2q_\star)\\
 &\qquad>
 4b(b-1)\left(1-\frac2b\right)
 -(b-2)\left(b-\frac2b\right)
 =(b-2)\left(3b-4+\frac2b\right)>0.
\end{split}
\]
By concavity, $H>0$, so $f_-$ is strictly
convex.

Since $f_-$ is strictly convex and $f_-(0)=f_-(1)=0$, one has
$f_-(1/2)<0$.  Fix a constant $c>0$ with
$c\le-2f_-(1/2)$. Applying the definition of convexity on $[0,1]$ and then on each half
of that interval gives the three asserted bounds.
\end{proof}

\begin{proof}[Proof of Lemma~\ref{lem:cluster1}]
We follow the second-moment method argument in \cite[Proposition~24]{BZ}. 
Recall that $P_t$ is the transition kernel of the underlying walk and that
$N_y(t)$ counts the particles at $y$ at time $t$, where
\begin{equation*}
 \E_x[N_y(t)]=e^{\lambda t}P_t(x,y).
\end{equation*}

\smallskip\noindent\emph{Step 1:  second moment computation.}
Every ordered pair of distinct particles at $y$ at time $t$ has a unique
latest common ancestor $z$ at time $r$. The conditional expectation of the number of ordered pairs at $y$ is
$2(e^{\lambda(t-r)}P_{t-r}(z,y))^2$.  The expected number of particles at
$z$ at time $r$ is $e^{\lambda r}P_r(x,z)$.  Integrating over branching
events, whose rate is $\lambda$ per particle, gives
\begin{equation*}
\begin{split}
 \E_x[N_y(t)(N_y(t)-1)]
 &=\int_0^t\sum_{z\in \{0,1,\ldots,b-1\}^d} \lambda e^{\lambda r}P_r(x,z)
     \,2\bigl(e^{\lambda(t-r)}P_{t-r}(z,y)\bigr)^2\dd r\\
 &=2\lambda e^{2\lambda t}\int_0^t e^{-\lambda r}
   \sum_{z\in \{0,1,\ldots,b-1\}^d} P_r(x,z)P_{t-r}(z,y)^2\dd r.
\end{split}
\end{equation*}
This is also the continuous-time $k=2$ many-to-two identity
\cite{HR}.  Thus,
\begin{equation}\label{eq:cluster5}
 \frac{\E_x[N_y(t)^2]}{\E_x[N_y(t)]^2}
 =\frac1{\E_x[N_y(t)]}+2\lambda\int_0^t e^{-\lambda r}
 \frac{\sum_{z\in \{0,1,\ldots,b-1\}^d}P_r(x,z)P_{t-r}(z,y)^2}{P_t(x,y)^2}\dd r.
\end{equation}

\smallskip\noindent\emph{Step 2: factorization over coordinates.}
Put $p=e^{-\beta r/d}$ and $q=e^{-\beta t/d}$, and suppose that
$\alpha=d_{\rm H}(x,y)/d$.  The fraction in
\eqref{eq:cluster5} factors over the coordinates
$i\in\{1,\dots,d\}$.  Write
$a(v)=(1+(b-1)v)/b$ and $c(v)=(1-v)/b$. There are two cases.
\begin{itemize}
    \item If $x_i=y_i$, then $z_i=x_i$ contributes $a(p)a(q/p)^2$, while each of
the other $b-1$ choices for $z_i$ contributes $c(p)c(q/p)^2$.  After
division by the denominator $a(q)^2$, the factor becomes
$$
 \frac{a(p)a(\frac{q}{p})^2+(b-1)c(p)c(\frac{q}{p})^2}{a(q)^2}=K_+(p,q).
$$
\item If $x_i\ne y_i$, then $z_i=x_i$, $z_i=y_i$, or
$z_i\notin\{x_i,y_i\}$.  These cases contribute respectively
$a(p)c(q/p)^2$, $c(p)a(q/p)^2$, and
$(b-2)c(p)c(q/p)^2$.  After division by $c(q)^2$, the normalized factor is
$$
 \frac{a(p)c(\frac{q}{p})^2+c(p)a(\frac{q}{p})^2
 +(b-2)c(p)c(\frac{q}{p})^2}{c(q)^2}=K_-(p,q).
$$
\end{itemize}
Since $e^{-\lambda r}=p^{(\lambda/\beta)d}$,
\eqref{eq:cluster5} becomes
\begin{equation}\label{eq:cluster6}
 \frac{\E_x[N_y(t)^2]}{\E_x[N_y(t)]^2}
 =\frac1{\E_x[N_y(t)]}+2\lambda\int_0^t
   \exp(dG_\alpha(p,q))\dd r,
\end{equation}
where
$$
 G_\alpha(p,q)=\frac{\lambda}{\beta}\log p
 +(1-\alpha)\log K_+(p,q)+\alpha\log K_-(p,q).
$$

\smallskip\noindent\emph{Step 3: uniform integral bound.}
We now set $t=x_\star d$ and $q=q_\star$.  For $0\le u\le1$, let
$p=q_\star^u$ and let $f_\pm$ be as in
Lemma~\ref{lem:cluster3}.  Thus,
\begin{equation*}
 G_\alpha(q_\star^u,q_\star)
 =(1-\alpha)f_+(u)+\alpha f_-(u).
\end{equation*}
Lemma~\ref{lem:cluster3} gives a constant $c>0$ such that uniformly in
$\alpha\in[0,1]$,
$$
 G_\alpha(q_\star^u,q_\star)\le
 \begin{cases}
  -\min\{\log R,c\}u,&0\le u\le\tfrac12;\\
  -\frac{c}{2}(1-u),&\tfrac12\le u\le1,\ \alpha\ge\tfrac12;\\
  -\frac{\log R}{4},&\tfrac12\le u\le1,\ \alpha<\tfrac12.
 \end{cases}
$$
Using $u=r/(x_\star d)$, the corresponding contributions to the integral in \eqref{eq:cluster6} 
are at most
\begin{align*}
 x_\star d\int_0^{1/2}e^{-d\min\{\log R,c\}u}\dd u
 &\le\frac{x_\star}{\min\{\log R,c\}},\\
 x_\star d\int_{1/2}^1e^{-dc(1-u)/2}\dd u
 &\le\frac{2x_\star}{c},\\
 x_\star d\int_{1/2}^1e^{-d\log R/4}\dd u
 &\le\frac{x_\star d}{2}e^{-d\log R/4}.
\end{align*}
Since $d e^{-d\log R/4}\ll1$, these three bounds give
\begin{equation}\label{eq:cluster7}
 \sup_{0\le\alpha\le1}
 \int_0^{x_\star d}\exp(dG_\alpha(p,q_\star))\dd r\ll1.
\end{equation}

\smallskip\noindent\emph{Step 4: Paley--Zygmund.}
At $t=x_\star d$, \eqref{eq:kernel1} and
\eqref{eq:root1} give exactly
$$
 \E_x[N_y(x_\star d)]=e^{\lambda x_\star d}b^{-d}
       (1+(b-1)q_\star)^{(1-\alpha)d}(1-q_\star)^{\alpha d}
     =R^{(1-\alpha)d}\ge1.
$$
Together, equations \eqref{eq:cluster6} and
\eqref{eq:cluster7} imply
\[
\begin{split}
 \frac{\E_x[N_y(x_\star d)^2]}
      {(\E_x[N_y(x_\star d)])^2}
 &=\frac1{\E_x[N_y(x_\star d)]}
   +2\lambda\int_0^{x_\star d}
     \exp\bigl(dG_\alpha(p,q_\star)\bigr)\dd r\\
 &\le1+2\lambda\sup_{0\le\alpha\le1}
   \int_0^{x_\star d}\exp\bigl(dG_\alpha(p,q_\star)\bigr)\dd r
 \ll1
\end{split}
\]
uniformly in $d,x,y$.  Paley--Zygmund now yields
$$
 \Pp_x\!\left(N_y(x_\star d)>0\right)\ge p_0.
$$
The event $\{N_y(x_\star d)>0\}$ is contained in
$\{\tau_y\le x_\star d\}$, which proves \eqref{eq:cluster1}.
\end{proof}

\subsection{Proof of Lemma~\ref{lem:cluster2}}
\label{app:cluster2}

\begin{proof}[Proof of Lemma~\ref{lem:cluster2}]
Let $X$ be the random walk of one particle, and let $q(x,z)$ be its
jump rate from $x$ to $z$, so $\sum_zq(x,z)=1$.  For $0\le t\le K$, let $J_A(t)$ be the
number of mutation jumps into $A$ by time $t$.  Equations
\eqref{eq:event1} and~\eqref{eq:spine1} give
\begin{equation}\label{eq:cluster8}
 \E_x[J_A(t)]=\int_0^t e^{\lambda s}
 \E_x\!\left[\sum_{z\in A}q(X(s),z)\right]\dd s.
\end{equation}
The sum inside the expectation is at most one.  Hence,
$\E_x[J_A(t)]\le(e^{\lambda t}-1)/\lambda
\le\lambda^{-1}e^{\lambda K}$,
uniformly in $x$ and $A$.

To compute the second moment, we first classify an unordered pair of distinct jumps into $A$ as follows.
\begin{itemize}
\item \emph{Ancestral pair.}  One jump lies on the ancestral line of the
other.  We select the earlier jump, which counts every such unordered pair
exactly once.  If the selected jump occurs at time $s$ and lands at $z$,
its descendants make on average $\E_z[J_A(t-s)]$ further jumps into $A$.

\item \emph{Split pair.}  The two ancestral lines separate at a unique
branching event.  We select that event, which counts every such unordered
pair exactly once.  The two descendant clusters rooted at its children are
independent, and each has conditional mean
$\E_{X(s)}[J_A(t-s)]$; these events occur at rate $\lambda$.
\end{itemize}
Equations \eqref{eq:event1} and~\eqref{eq:spine1}
give the following identity for the second factorial moment:
\begin{equation}\label{eq:cluster9}
\begin{split}
 \E_x[J_A(t)(J_A(t)-1)]
 &=2\int_0^t e^{\lambda s}\E_x\!\left[
   \sum_{z\in A}q(X(s),z)\E_z[J_A(t-s)]\right]\dd s\\
 &\quad+2\lambda\int_0^t e^{\lambda s}
   \E_x\!\left[\bigl(\E_{X(s)}[J_A(t-s)]\bigr)^2\right]\dd s.
\end{split}
\end{equation}
In both terms, the factor two converts the unordered pairs counted above
into the ordered pairs in $J_A(t)(J_A(t)-1)$.

In \eqref{eq:cluster9}, the bound
$\E_y[J_A(r)]\le\lambda^{-1}e^{\lambda K}$, valid for $0\le r\le K$, gives
\[
\begin{split}
 &2\int_0^t e^{\lambda s}\E_x\!\left[
   \sum_{z\in A}q(X(s),z)\E_z[J_A(t-s)]\right]\dd s\\
 &\qquad\le
 \frac{2e^{\lambda K}}{\lambda}\int_0^t e^{\lambda s}
 \E_x\!\left[\sum_{z\in A}q(X(s),z)\right]\dd s
 =\frac{2e^{\lambda K}}{\lambda}\E_x[J_A(t)]
\end{split}
\]
and, for every state $y$,
\[
 \bigl(\E_y[J_A(t-s)]\bigr)^2
 \le\frac{e^{\lambda K}}{\lambda}\E_y[J_A(t-s)].
\]
By the Markov property of $X$, \eqref{eq:cluster8}, and
Fubini's theorem,
$$
\begin{aligned}
 \int_0^t e^{\lambda s}\E_x\!\left[\E_{X(s)}[J_A(t-s)]\right]\dd s
 &=\int_0^t\int_0^{t-s}e^{\lambda(s+u)}
   \E_x\!\left[\sum_{z\in A}q(X(s+u),z)\right]\dd u\dd s\\
 &=\int_0^t v e^{\lambda v}
   \E_x\!\left[\sum_{z\in A}q(X(v),z)\right]\dd v
 \le t\E_x[J_A(t)].
\end{aligned}
$$
Consequently, the second term in
\eqref{eq:cluster9} is at most
\[
 2e^{\lambda K}\int_0^t e^{\lambda s}
 \E_x\!\left[\E_{X(s)}[J_A(t-s)]\right]\dd s
 \le2t e^{\lambda K}\E_x[J_A(t)],
\]
and hence
\[
 \E_x[J_A(t)(J_A(t)-1)]
 \le \frac{2e^{\lambda K}}{\lambda}(1+\lambda t)\E_x[J_A(t)].
\]
At $t=K$, $2\lambda^{-1}e^{\lambda K}(1+\lambda K)
\le e^{C(1+K)}$ for a constant depending only on $\lambda$.  This proves
\eqref{eq:cluster2}.
\end{proof}

\subsection{Proof of Lemma~\ref{lem:binary2}}\label{app:binary1}

\begin{lemma}\label{lem:binary5}
Let $N_y(t)$ be the number of particles at $y$ at time $t$ and denote by 
$\alpha=d_{\rm H}(x,y)/d$.  With $q=e^{-2t/d}$,
\begin{equation}\label{eq:binary17}
 \mu:=\E_x[N_y(t)]
 =e^{\lambda t}2^{-d}(1+q)^{(1-\alpha)d}(1-q)^{\alpha d}.
\end{equation}
For $0\le h\le-\log q$, we have
\begin{equation}\label{eq:binary18}
 \frac{\E_x[N_y(t)^2]}{\mu^2}
 =\frac1\mu+2\lambda\int_0^t
 \exp\!\left(dG_\alpha\!\left(e^{-2(t-s)/d},q\right)\right)\dd s,
\end{equation}
where
\begin{align}
    G_\alpha(e^{-h},q)={}&-\frac{\lambda h}{2}
 +(1-\alpha)\log\!\left(
 1+\frac{q^2}{(1+q)^2}(e^{2h}-1)\right)+\alpha\log\!\left(
 1+\frac{q^2}{(1-q)^2}(e^{2h}-1)\right).\label{eq:G1}
\end{align}
\end{lemma}

\begin{proof}
Equation~\eqref{eq:binary17} follows immediately from the many-to-one formula \eqref{eq:spine1}.
Since $b=2$, equations \eqref{eq:kernel2} and~\eqref{eq:kernel3} become
\[
 K_+(p,q)=1+\frac{q^2}{(1+q)^2}(p^{-2}-1),
 \qquad
 K_-(p,q)=1+\frac{q^2}{(1-q)^2}(p^{-2}-1).
\]
The second-moment formula \eqref{eq:cluster6} therefore becomes
\begin{equation*}
 \frac{\E_x[N_y(t)^2]}{\mu^2}
 =\frac1\mu+2\lambda\int_0^t e^{-\lambda r}
 K_+(e^{-2r/d},q)^{(1-\alpha)d}
 K_-(e^{-2r/d},q)^{\alpha d}\dd r.
\end{equation*}
Using the change of variable $s=t-r$ gives \eqref{eq:binary18}.
\end{proof}

\begin{lemma}\label{lem:binary6}
Fix $\bar q\in(q_\star,1/2)$, and suppose that
$q=e^{-2t/d}\in[q_\star,\bar q]$.  Let $\alpha$, $\mu$, and $G_\alpha$ be as in
Lemma~\ref{lem:binary5}. Then, uniformly in $d$ large enough and
$\alpha\in[0,1]$,
\begin{equation}\label{eq:binary20}
 \int_0^t
 \exp\!\left(dG_\alpha\!\left(e^{-2(t-s)/d},q\right)\right)\dd s
 \ll_{\bar q}\max\{1,\mu^{-1}\}.
\end{equation}
\end{lemma}

\begin{proof}
For $q\in[q_\star,\bar q]$ and
$0\le h\le-\log q$, the nonlinear part of 
$h\mapsto G_\alpha(e^{-h},q)$ has the form
$\log(1-\theta+\theta e^{2h})$, with
$\theta=q^2/(1\pm q)^2$, and
\begin{equation*}
 \frac{\mathrm d^2}{\mathrm d h^2}\log(1-\theta+\theta e^{2h})
 =\frac{4\theta(1-\theta)e^{2h}}
 {(1-\theta+\theta e^{2h})^2}\gg_{\bar{q}} 1.
\end{equation*}
 Hence, inserting into \eqref{eq:G1} yields 
$\frac{\mathrm d^2}{\mathrm d h^2}G_\alpha(e^{-h},q)\ge c$.

A direct computation yields $G_\alpha(1,q)=0$ and
\[
\begin{split}
 G_\alpha(q,q)
 &=\frac\lambda2\log q
 +(1-\alpha)\log\!\left(\frac2{1+q}\right)
 +\alpha\log\!\left(\frac2{1-q}\right)\\
 &=\frac\lambda2\log q+\log2
 -(1-\alpha)\log(1+q)-\alpha\log(1-q)\\
 &=-\frac1d\log\mu,
\end{split}
\]
where the last equality follows from $t/d=-\frac12\log q$ in
\eqref{eq:binary17}.  The function
$h\mapsto G_\alpha(e^{-h},q)-ch^2/2$ is convex and therefore lies below
the line segment joining its two endpoint values, i.e., 
\[
\begin{split}
 G_\alpha(e^{-h},q)-\frac c2h^2
 &\le \frac{h}{-\log q}
 \left(G_\alpha(q,q)-\frac c2(-\log q)^2\right)
 +\left(1-\frac{h}{-\log q}\right)G_\alpha(1,q)\\
 &=-\frac{h}{-\log q}\frac{\log\mu}{d}
 -\frac c2h(-\log q).
\end{split}
\]
Adding $ch^2/2$ to both sides gives
\begin{equation}\label{eq:binary22}
 G_\alpha(e^{-h},q)
 \le-\frac{h}{-\log q}\frac{\log\mu}{d}
 -\frac{c}{2}h(-\log q-h).
\end{equation}
On one hand, the exponential of the first term on the right-hand side of
\eqref{eq:binary22}, after multiplication by $d$, is
\[
 \exp\!\left(-\frac{h}{-\log q}\log\mu\right)
 =\mu^{-h/(-\log q)}\le\max\{1,\mu^{-1}\},
 \qquad 0\le h\le-\log q.
\]
On the other hand, since $-\log q\ge-\log\bar q$,
$$
 h(-\log q-h)
 \ge\frac{-\log\bar q}{2}\min\{h,-\log q-h\},
$$
because on the first half of the interval $-\log q-h\ge(-\log q)/2$,
and on the second half $h\ge(-\log q)/2$.
Then, inserting the above estimates into  \eqref{eq:binary22} gives
$$
 e^{dG_\alpha(e^{-h},q)}
 \le\max\{1,\mu^{-1}\}
 \exp\!\left(-\frac{dc(-\log\bar q)}{4}
 \min\{h,-\log q-h\}\right).
$$
Using the change of variable
$h=2(t-s)/d$ proves \eqref{eq:binary20}.
\end{proof}

\begin{proof}[Proof of Lemma~\ref{lem:binary2}]
Fix any $\bar q\in(q_\star,1/2)$.

\emph{Part~(i).}
Apply Lemma~\ref{lem:binary5} with $t=x_\star d-D$.  Then
$\alpha=1-r/d$ and $q=q_\star e^{2D/d}$.  Formula
\eqref{eq:binary17} can be rewritten exactly as
$$
 \log\mu
 =d\Phi\!\left(x_\star-\frac Dd\right)
   +r\log\!\left(\frac{1+q}{1-q}\right).
$$
Taylor's theorem, applied on a fixed neighborhood of $x_\star$, gives
$$
 d\Phi\!\left(x_\star-\frac Dd\right)
 =-\chi D+O\!\left(\frac{D^2}{d}\right),
 \qquad
 \log\!\left(\frac{1+q}{1-q}\right)
 =\log R+O\!\left(\frac Dd\right).
$$
The derivatives in these two applications are bounded uniformly for
$0\le D\le C\log\log d$.  Consequently,\footnote{Here $O_C(\cdot)$ allows
the implicit constant to depend on the fixed $C$ in part~(i), uniformly
over the ranges in Lemma~\ref{lem:binary2}(i).}
$$
 \log\mu=-\chi D+r\log R
 +O\!\left(\frac{D^2+rD}{d}\right)
 =-\chi D+r\log R
 +O_C\!\left(\frac{(\log\log d)^2}{d}\right).
$$
In particular,
$\mu/(e^{-\chi D}R^r)\to1$ uniformly, and the hypothesis in part~(i)
gives $e^{-\chi D}R^r\le1$.  Also $q=q_\star e^{2D/d}\in
[q_\star,\bar q]$ for all large $d$.  Lemmas
\ref{lem:binary5} and~\ref{lem:binary6} give
$\E_x[N_y(x_\star d-D)^2]\ll_C\mu+\mu^2$.
Thus, $\mu\le2$ for all large $d$, and Paley--Zygmund yields
\[
\begin{aligned}
 \Pp_x\!\left(\tau_y\le x_\star d-D\right)
 &\ge\Pp_x\!\left(N_y(x_\star d-D)>0\right)\\
 &\ge\frac{\mu^2}{\E_x[N_y(x_\star d-D)^2]}\\
 &\gg_C\mu\gg e^{-\chi D}R^r
\end{aligned}
\]
for all large $d$.  This proves part~(i).

\smallskip\noindent\emph{Part~(ii).}
Choose $\delta>0$ small enough that
$$
 \delta\le\frac{x_\star}{4},\qquad
 4\delta\le\log\!\left(\frac{\bar q}{q_\star}\right);
$$
our hypothesis gives $\alpha\le1-\ell/(2d)$.  At
$t=x_\star d-2\delta\ell$, one has
$q=e^{-2t/d}=q_\star e^{4\delta\ell/d}\in[q_\star,\bar q]$.
Lemma~\ref{lem:binary5} then leads to
\begin{align}
     \frac1d\log\E_x[N_y(t)]
 =\lambda\left(x_\star-\frac{2\delta\ell}{d}\right)-\log2
 +(1-\alpha)\log(1+q)+\alpha\log(1-q).\label{eq:mean1}
\end{align}
This right-hand side is decreasing in $\alpha$, so its minimum over
$\alpha\le1-\ell/(2d)$ is at $\alpha=1-\ell/(2d)$.  For
$0\le a\le\delta$, define
\[
 F(a):=\lambda\left(x_\star-\frac{2a\ell}{d}\right)-\log2
 +\frac{\ell}{2d}\log(1+q_\star e^{4a\ell/d})
 +\left(1-\frac{\ell}{2d}\right)\log(1-q_\star e^{4a\ell/d}).
\]
It follows from \eqref{eq:mean1} that
\[
 \frac1d\log\E_x[N_y(t)]\ge F(\delta).
\]
By \eqref{eq:root1},
$F(0)=\ell\log R/(2d)$, and
\[
 F'(a)=-\frac{2\lambda\ell}{d}
 +\frac{4\ell}{d}q_\star e^{4a\ell/d}\left(
 \frac{\frac{\ell}{2d}}{1+q_\star e^{4a\ell/d}}
 -\frac{1-\frac{\ell}{2d}}{1-q_\star e^{4a\ell/d}}\right).
\]
Next, we argue that $|F'(a)|\ll \ell/d$.  Indeed, since
$0\le\ell/(2d)\le1/2$ and $q_\star e^{4a\ell/d}\le\bar q$,
\[
 \frac{2\lambda\ell}{d}\ll\frac\ell d,
 \qquad
 \frac{4\ell}{d}q_\star e^{4a\ell/d}
 \left(
 \frac{\frac{\ell}{2d}}{1+q_\star e^{4a\ell/d}}
 +\frac{1-\frac{\ell}{2d}}{1-q_\star e^{4a\ell/d}}
 \right)
 \ll_{\bar q}\frac\ell d,
\]
where the second denominator is bounded below by $1-\bar q>1/2$.
For $\delta$ small enough depending on $R$, we have that  integrating the derivative bound over
$a\in[0,\delta]$ gives
\[
 F(\delta)\ge F(0)-\frac{C\delta\ell}{d}
 =\frac{\ell}{2d}\log R-\frac{C\delta\ell}{d}
 \ge\frac{\ell}{4d}\log R.
\]
In summary, we have proved
\begin{equation}\label{eq:binary23}
 q=e^{-2t/d}=q_\star e^{4\delta\ell/d}\in[q_\star,\bar q],
 \qquad \mu\ge R^{\ell/4}\ge1.
\end{equation}
Lemmas \ref{lem:binary5} and~\ref{lem:binary6},
together with \eqref{eq:binary23}, now give
$\E_x[N_y(t)^2]/\mu^2\ll1$.
Hence, Paley--Zygmund gives
$$
 \Pp_x(\tau_y\le t)
 \ge\Pp_x(N_y(t)>0)
 \ge\frac{\E_x[N_y(t)]^2}{\E_x[N_y(t)^2]}\gg1.
$$
This proves part~(ii).
\end{proof}
\section{Deterministic arguments}\label{sec:elementary1}

This appendix collects general estimates and deferred deterministic arguments
used in the main proofs.

\subsection{Proof of Lemma~\ref{lem:ratio1}}
\label{app:ratio1}

\begin{proof}[Proof of Lemma~\ref{lem:ratio1}]
By \eqref{eq:rate1} and $\Phi(x_\star)=0$,
$q_\star^{-\lambda/\beta}(1-q_\star)=b$.
Hence, $R-1=q_\star^{1-\lambda/\beta}=bq_\star/(1-q_\star)$.
In particular,
\[
 R-1<\lambda
 \quad\Longleftrightarrow\quad
 \frac{bq_\star}{1-q_\star}<\lambda
 \quad\Longleftrightarrow\quad
 q_\star<\frac{\lambda}{b+\lambda}.
\]
Since
$q\mapsto q^{-\lambda/\beta}(1-q)$ is strictly decreasing and equals $b$ at
$q_\star$, it suffices to show that its value at $\lambda/(b+\lambda)$ is less than $b$.  Indeed,
\[
 \frac1b\left(\frac{\lambda}{b+\lambda}\right)^{-\lambda/\beta}
 \left(1-\frac{\lambda}{b+\lambda}\right)
 =(b+\lambda)^{\lambda/\beta-1}\lambda^{-\lambda/\beta}<1
\]
is equivalent to
$(b+\lambda)^{1-\lambda/\beta}\lambda^{\lambda/\beta}>1$.
Since $\lambda/\beta=\frac{b-1}{b}\lambda<1$ and
$\log(b+\lambda)>\log b$, it is enough to prove
$(1-\frac{b-1}{b}\lambda)\log b+\frac{b-1}{b}\lambda\log\lambda>0$.
The derivative of the left-hand side is
$\frac{b-1}{b}(1+\log\lambda-\log b)<0$ on $(0,1]$ when $b\ge3$, while its value at
$\lambda=1$ is $b^{-1}\log b>0$.  The desired inequality then follows.
\end{proof}

\subsection{Proof of Lemma~\ref{lem:boxes1}}\label{app:boxes1}

\begin{proof}[Proof of Lemma~\ref{lem:boxes1}]
Independence gives
$$
 \Pp\!\left(y\notin\bigcup_vS_v\right)
 =\prod_v\bigl(1-\Pp(y\in S_v)\bigr).
$$
For $0\le x\le p$,
$$
 -\log(1-x)=\int_0^x\frac{1}{1-s}\dd s
 \le\frac{x}{1-p}.
$$
Consequently,
$$
 \Pp\!\left(y\notin\bigcup_vS_v\right)
 \ge \exp\!\left(-\frac{\sum_v\Pp(y\in S_v)}{1-p}\right),
$$
which proves \eqref{eq:boxes1}.  Let
$I_y=\one_{\{y\notin\bigcup_v S_v\}}$.  For $y\ne z$, the assumption
$|S_v|\le1$ gives
$\Pp(y,z\notin S_v)=1-\Pp(y\in S_v)-\Pp(z\in S_v)$.  Independence over
$v$ gives
$$
 \E[I_yI_z]=\prod_v\bigl(1-\Pp(y\in S_v)-\Pp(z\in S_v)\bigr)
 \le\prod_v\bigl(1-\Pp(y\in S_v)\bigr)
             \bigl(1-\Pp(z\in S_v)\bigr)
 =\E[I_y]\E[I_z].
$$
Thus, the indicators $(I_y)$ are pairwise negatively correlated, and
$$
 \operatorname{Var}\!\left(\sum_{y\in\mathcal U}I_y\right)
 \le\sum_{y\in\mathcal U}\operatorname{Var}(I_y)
 \le\sum_{y\in\mathcal U}\E[I_y]
 =\E\!\left[\sum_{y\in\mathcal U}I_y\right].
$$
Chebyshev's inequality then leads to
$$
\begin{aligned}
 \Pp\!\left(\mathcal U\subseteq\bigcup_vS_v\right)
 &=\Pp\!\left(\sum_{y\in\mathcal U}I_y=0\right)\\
 &\le\frac{\operatorname{Var}(\sum_{y\in\mathcal U}I_y)}
 {\E[\sum_{y\in\mathcal U}I_y]^2}
 \le\frac{1}{\E[\sum_{y\in\mathcal U}I_y]}
 =\left(\sum_{y\in\mathcal U}
 \Pp\!\left(y\notin\bigcup_vS_v\right)\right)^{-1}
 \\&\le\left(
 \sum_{y\in\mathcal U}
 \exp\!\left(-\frac{\sum_v\Pp(y\in S_v)}{1-p}\right)
 \right)^{-1}.
\end{aligned}
$$
This proves \eqref{eq:boxes2}.
\end{proof}

\subsection{A deterministic compound-Poisson estimate}\label{app:compound1}

\begin{lemma}\label{lem:compound1}
Let $\nu$ be a finite measure on
$\mathbb N$ and $(N_j)_{j\ge1}$ be independent Poisson random variables with means
$\nu(\{j\})$.  For every $P\ge1$ and
$\gamma\ge0$,
\begin{equation}\label{eq:compound1}
 \sum_{k\ge1}(1+k)^\gamma
 \Pp\!\left(\sum_{j\ge1}jN_j\ge k\right)^P
 \ll_\gamma(1+\nu(\mathbb N))^{\gamma+1}
 \sum_{k\ge1}(1+k)^\gamma\nu([k,\infty))^P.
\end{equation}
\end{lemma}

\begin{proof}
If $\nu(\mathbb N)=0$, both sides vanish.  Suppose $\nu(\mathbb N)>0$.
Then $\sum_{j\ge1}N_j\sim\operatorname{Pois}(\nu(\mathbb N))$.
Conditional on $\sum_{j\ge1}N_j=n\ge1$, the $n$ Poisson points have
independent weights with common law $\nu/\nu(\mathbb N)$ \cite[Proposition~3.8]{LastPenrose}.  If the sum of
their weights is at least $k$, at least one weight is at least
$\lceil k/n\rceil$.  The union bound therefore gives
\begin{equation}\label{eq:compound2}
 \Pp\!\left(\sum_{j\ge1}jN_j\ge k\,\middle|\,
 \sum_{j\ge1}N_j=n\right)
 \le\frac n{\nu(\mathbb N)}\nu\!\left(
 [\lceil k/n\rceil,\infty)\right).
\end{equation}
Since
\[
 \sum_{n\ge1}
 \Pp\!\left(\sum_{j\ge1}N_j=n\right)
 \frac{n}{\nu(\mathbb N)}
 =
 \frac{\E[\sum_{j\ge1}N_j]}{\nu(\mathbb N)}
 =1,
\]
averaging \eqref{eq:compound2} and applying Jensen's inequality with
these weights yields
\[
 \Pp\!\left(\sum_{j\ge1}jN_j\ge k\right)^P
 \le\sum_{n\ge1}\Pp\!\left(\sum_{j\ge1}N_j=n\right)
 \frac n{\nu(\mathbb N)}
 \nu\!\left([\lceil k/n\rceil,\infty)\right)^P.
\]
For every $n\ge1$, grouping the integers $k$ according to
$j=\lceil k/n\rceil$ gives
\[
 \sum_{k\ge1}(1+k)^\gamma
 \nu\!\left([\lceil k/n\rceil,\infty)\right)^P
 \ll_\gamma n^{\gamma+1}
 \sum_{j\ge1}(1+j)^\gamma\nu([j,\infty))^P.
\]
Finally, Lemma~\ref{lem:poisson1} applied with exponent
$\gamma+2$ gives
\begin{equation*}
 \sum_{n\ge1}\Pp\!\left(\sum_{j\ge1}N_j=n\right)
 \frac{n^{\gamma+2}}{\nu(\mathbb N)}
 =\frac{\E[(\sum_{j\ge1}N_j)^{\gamma+2}]}{\nu(\mathbb N)}
 \ll_\gamma(1+\nu(\mathbb N))^{\gamma+1}.
\end{equation*}
Combining the last three bounds proves \eqref{eq:compound1}.
\end{proof}

\subsection{Two elementary integral and moment bounds}

Proposition~\ref{prop:load1} relies on the following two lemmas.

\begin{lemma}\label{lem:beta1}
For every fixed $a>0$, every $s\ge0$, and every integer $k\ge1$,
\begin{equation}\label{eq:beta3}
 \int_0^\infty e^{-a\lambda u}
 (1-e^{-\lambda u})^{a(k-1)}\dd u\ll_a(1+k)^{-a}.
\end{equation}
Moreover, for $k>e^{\lambda s}$,
\begin{equation}\label{eq:beta4}
 \int_0^s e^{-a\lambda u}
 (1-e^{-\lambda u})^{a(k-1)}\dd u
 \ll_a(1+k)^{-a}e^{-a k e^{-\lambda s}/4}.
\end{equation}
\end{lemma}

\begin{proof}
With $z=e^{-\lambda u}$, the left-hand side of 
\eqref{eq:beta3} becomes
$$
 \frac1\lambda\int_0^1z^{a-1}(1-z)^{a(k-1)}\dd z.
$$
For $k\ge2$, use $(1-z)^{a(k-1)}\le e^{-a(k-1)z}$ and extend the integral
to $[0,\infty)$ to obtain
$$
 \frac1\lambda\int_0^\infty z^{a-1}e^{-a(k-1)z}\dd z
 =\frac{\Gamma(a)}{\lambda\bigl(a(k-1)\bigr)^a}
 \ll_a(1+k)^{-a}.
$$
For $k=1$, the integral equals $(a\lambda)^{-1}$, and hence 
\eqref{eq:beta3} follows.

For \eqref{eq:beta4}, the same substitution and
$1-z\le e^{-z}$ give
$$
 \int_0^s e^{-a\lambda u}(1-e^{-\lambda u})^{a(k-1)}\dd u
 \le\frac1\lambda\int_{e^{-\lambda s}}^\infty
 z^{a-1}e^{-a(k-1)z}\dd z.
$$
Since $k>e^{\lambda s}\ge1$, it holds $k\ge2$ and $k-1\ge k/2$.
Using the change of variable $y=a(k-1)z$,
\[
\begin{split}
 \int_0^s e^{-a\lambda u}
 (1-e^{-\lambda u})^{a(k-1)}\dd u
 &\le\frac{1}{\lambda\bigl(a(k-1)\bigr)^a}
 \int_{a(k-1)e^{-\lambda s}}^\infty y^{a-1}e^{-y}\dd y\\
 &\le\frac{e^{-a(k-1)e^{-\lambda s}/2}}
 {\lambda\bigl(a(k-1)\bigr)^a}
 \int_0^\infty y^{a-1}e^{-y/2}\dd y\\
 &\ll_a(1+k)^{-a}
 e^{-a k e^{-\lambda s}/4}.
\end{split}
\]
This proves \eqref{eq:beta4}.
\end{proof}

\begin{lemma}\label{lem:poisson1}
If $N$ is Poisson with mean $\mu>0$, then, for every fixed real $r\ge1$,
\begin{equation*}
 \frac{\E[N^r]}{\mu}\ll_r(1+\mu)^{r-1}.
\end{equation*}
\end{lemma}

\begin{proof}
If $0<\mu\le1$, then
$$
 \E[N^r]=e^{-\mu}\sum_{n\ge1}\frac{n^r\mu^n}{n!}
 \le\mu\sum_{n\ge1}\frac{n^r}{n!}\ll_r\mu.
$$
If $\mu\ge1$, the Poisson factorial-moment identities give
\[
\begin{split}
 \E[N^{\lceil r\rceil}]
 &\ll_r\sum_{j=1}^{\lceil r\rceil}
 \E\left[\prod_{k=0}^{j-1}(N-k)\right]=\sum_{j=1}^{\lceil r\rceil}\mu^j
 \ll_r\mu^{\lceil r\rceil}.
\end{split}
\]
Lyapunov's inequality then gives
$\E[N^r]\le
\E[N^{\lceil r\rceil}]^{r/\lceil r\rceil}\ll_r\mu^r$. This completes the proof.
\end{proof}

\subsection{A deterministic packing estimate}

\begin{lemma}\label{lem:packing2}
Fix $b>2$ and $0<\epsilon<1/2$.  There exist
$\delta\in(0,1/2)$ and $h>0$ such that, for all sufficiently large $d$,
every set $\mathcal A_d\subseteq\{0,\ldots,b-1\}^d$ satisfying
$|\mathcal A_d|\ge(b-1)^{(1-\epsilon)d}$ contains a subset
$\mathcal C_d$ for which
$$
 |\mathcal C_d|\ge e^{hd},\qquad
 d_{\rm H}(y,z)\ge\delta d\quad(y\ne z).
$$
\end{lemma}

\begin{proof}
For $0<x\le1/2$,
$$
 \frac{\mathrm d}{\mathrm d x}\bigl(-x\log x-(1-x)\log(1-x)\bigr)
 =\log\!\left(\frac{1-x}{x}\right)\ge0.
$$
Thus, for $0<\delta<1/2$ and $0\le j\le\delta d$, the entropy inequality \(\binom dj\le \exp(dH(\frac jd)),
~
H(x)=-x\log x-(1-x)\log(1-x)\)
gives
$$
 \binom dj(b-1)^j
 \le\exp\!\left(d\bigl(-\delta\log\delta
 -(1-\delta)\log(1-\delta)+\delta\log(b-1)\bigr)\right).
$$ Thus, a Hamming ball in the Hamming graph of
radius $\delta d$ has size at most
\begin{equation}\label{eq:packing1}
 \sum_{j=0}^{\lfloor\delta d\rfloor}
 \binom dj(b-1)^j
 \le(d+1)\exp\!\left(d\bigl(-\delta\log\delta
 -(1-\delta)\log(1-\delta)+\delta\log(b-1)\bigr)\right).
\end{equation}
Choose $\delta>0$ small such that
$$
 h:=\frac12\left((1-\epsilon)\log(b-1)
       +\delta\log\delta+(1-\delta)\log(1-\delta)
       -\delta\log(b-1)\right)>0.
$$
Select any element and delete its
radius-$\delta d$ ball.  By \eqref{eq:packing1}, each such operation removes a set of size at most the right-hand side of \eqref{eq:packing1}.
Therefore, for all sufficiently large $d$,
\[
 |\mathcal C_d|
 \ge\frac{(b-1)^{(1-\epsilon)d}}
 {(d+1)\exp(d(-\delta\log\delta-(1-\delta)\log(1-\delta)
                  +\delta\log(b-1)))}
 =\frac{e^{2hd}}{d+1}
 \ge e^{hd}.
\]
This gives the desired separation.
\end{proof}

\subsection{Proof of Lemma~\ref{lem:binary1}}\label{app:binary2}

\begin{proof}[Proof of Lemma~\ref{lem:binary1}]
By \eqref{eq:binary13},
$R-1=q_\star^{1-\lambda/2}=2q_\star/(1-q_\star)$.  Hence, the claim
$R-1<\lambda$ is equivalent to
$q_\star<\lambda/(2+\lambda)$.  Since
$q\mapsto q^{-\lambda/2}(1-q)$ is strictly decreasing on $(0,1)$ and
equals $2$ at $q_\star$, it is enough to show that its value at
$q=\lambda/(2+\lambda)$ is less than $2$, i.e., 
$(2+\lambda)^{1-\lambda/2}\lambda^{\lambda/2}>1$.
Indeed, $(2+\lambda)^{1-\lambda/2}\ge\sqrt2$ and
$\lambda^{\lambda/2}\ge e^{-1/(2e)}$, while $\sqrt2\,e^{-1/(2e)}>1$.

For the existence of $a$, it is enough to prove
\begin{equation}\label{eq:binary24}
 \frac{\chi}{\log R}\log\!\left(\frac{\chi}{\log R}\right)
 -\frac{\chi}{\log R}+1>\lambda
\end{equation}
whenever $0<R-1<\lambda\le1$.  The function
$2x-(1+x)\log(1+x)$ vanishes at $x=0$ and has derivative
$1-\log(1+x)>0$ on $(0,1)$.  Taking $x=R-1$ shows that
$2(R-1)>R\log R$ and
$(\lambda+R-1)/\log R>R$ for $\lambda\ge R-1$.  For fixed $R$, the
difference between the two sides of \eqref{eq:binary24} is increasing in
$\lambda$, since its derivative is
\begin{align}
     \frac{\mathrm d}{\mathrm d\lambda}
 \left(
 \frac{\lambda+R-1}{\log R}
 \log\!\left(\frac{\lambda+R-1}{\log R}\right)
 -\frac{\lambda+R-1}{\log R}+1-\lambda
 \right)
 =\frac{\log\!\left(\frac{\lambda+R-1}{\log R}\right)}{\log R}-1>0.\label{eq:derivative2}
\end{align}
Thus, it suffices to consider $\lambda=R-1$.

Observe that the continuous extension of $x\mapsto x/\log(1+x)$, with value $1$ at
$x=0$, is concave on $[0,1]$.  Considering the segment between the endpoint
values therefore gives
$$
 \frac{2x}{\log(1+x)}
 \ge2+\left(\frac{2}{\log2}-2\right)x.
$$

By \eqref{eq:derivative2} and since the function $y\mapsto y\log y-y+1$ is increasing for $y>1$, we have for every
$0<x<\lambda\le1$ 
\[
\begin{split}
 &\frac{\lambda+x}{\log(1+x)}
 \log\!\left(\frac{\lambda+x}{\log(1+x)}\right)
 -\frac{\lambda+x}{\log(1+x)}+1-\lambda\\
 &\quad\ge
 \frac{2x}{\log(1+x)}
 \log\!\left(\frac{2x}{\log(1+x)}\right)
 -\frac{2x}{\log(1+x)}+1-x\\
 &\quad\ge
 \left(2+\left(\frac2{\log2}-2\right)x\right)
 \log\!\left(2+\left(\frac2{\log2}-2\right)x\right)-\left(2+\left(\frac2{\log2}-2\right)x\right)+1-x.
\end{split}
\]
It is elementary to check that the right-hand side is decreasing and positive on $x\in[0,1]$.
Taking $x=R-1$ proves \eqref{eq:binary24}.  Since
$\chi/\log R>R>1$, continuity of $x\mapsto x\log x-x+1$ allows us to
choose $1<a<\chi/\log R$ with $a\log a-a+1>\lambda$.  The upper bound on
$a$ also gives $a\log R<\chi$.
\end{proof}


\subsection{A binomial sum upper bound}
\begin{lemma}\label{lem:shell1}
Fix $c,w,\lambda,\delta>0$ and choose $C_{\rm shell}$ so that
$\lambda\delta C_{\rm shell}>2$.  Then
\begin{equation}\label{eq:far5}
 \sum_{\ell>C_{\rm shell}\log\log d}^{d}
 \binom d\ell\exp\!\left(-cw e^{\lambda\delta\ell}\right)=o(1).
\end{equation}
\end{lemma}

\begin{proof}
The standard estimate $\binom d\ell\le(ed/\ell)^\ell$ gives the upper bound
\begin{align}
     \sum_{\ell>C_{\rm shell}\log\log d}^{d}
 \exp\!\left(\ell\log\!\left(\frac{ed}{\ell}\right)
 -cw e^{\lambda\delta\ell}\right).\label{eq:L1}
\end{align}
For $C_{\rm shell}\log\log d\le x\le d$ and $d$ large enough,
$$
 \frac{\mathrm d}{\mathrm d x}\log\!\left(
 \frac{e^{\lambda\delta x}}{x\log(\frac{ed}{x})}\right)
 =\lambda\delta-\frac1x+\frac1{x\log(\frac{ed}{x})}
 \ge\frac{\lambda\delta}{2}>0.
$$
Thus, this ratio is increasing on
$[C_{\rm shell}\log\log d,d]$, and at $x=C_{\rm shell}\log\log d$ we have
$$
 \frac{(\log d)^{\lambda\delta C_{\rm shell}}}
 {C_{\rm shell}\log\log d\,
  \log\!\left(\frac{ed}{C_{\rm shell}\log\log d}\right)}
 \asymp
 \frac{1}{C_{\rm shell}}
 \frac{(\log d)^{\lambda\delta C_{\rm shell}-1}}{\log\log d}
 \longrightarrow\infty.
$$
Consequently, for every integer $\ell>C_{\rm shell}\log\log d$ and all
sufficiently large $d$,
\begin{align}
    cw e^{\lambda\delta\ell}
 \ge2\ell\log\!\left(\frac{ed}{\ell}\right).\label{eq:L2}
\end{align}
Combining \eqref{eq:L1} and \eqref{eq:L2} proves \eqref{eq:far5}.
\end{proof}

\section{Simulations}\label{app:simulation1}

This appendix reports exact event-driven simulations of the cover time for $b=2$ and $b=3$. We take $\lambda=\log 2$. Solving \eqref{eq:rate1} numerically gives $x_\star=1.7019033$ for $b=3$. For $b=2$, it gives $x_\star=1.1518307$ and $\chi=0.9151038$. Figure~\ref{fig:simulation1} shows the simulation results.

\begin{figure}[H]
    \centering
    \includegraphics[width=\linewidth]{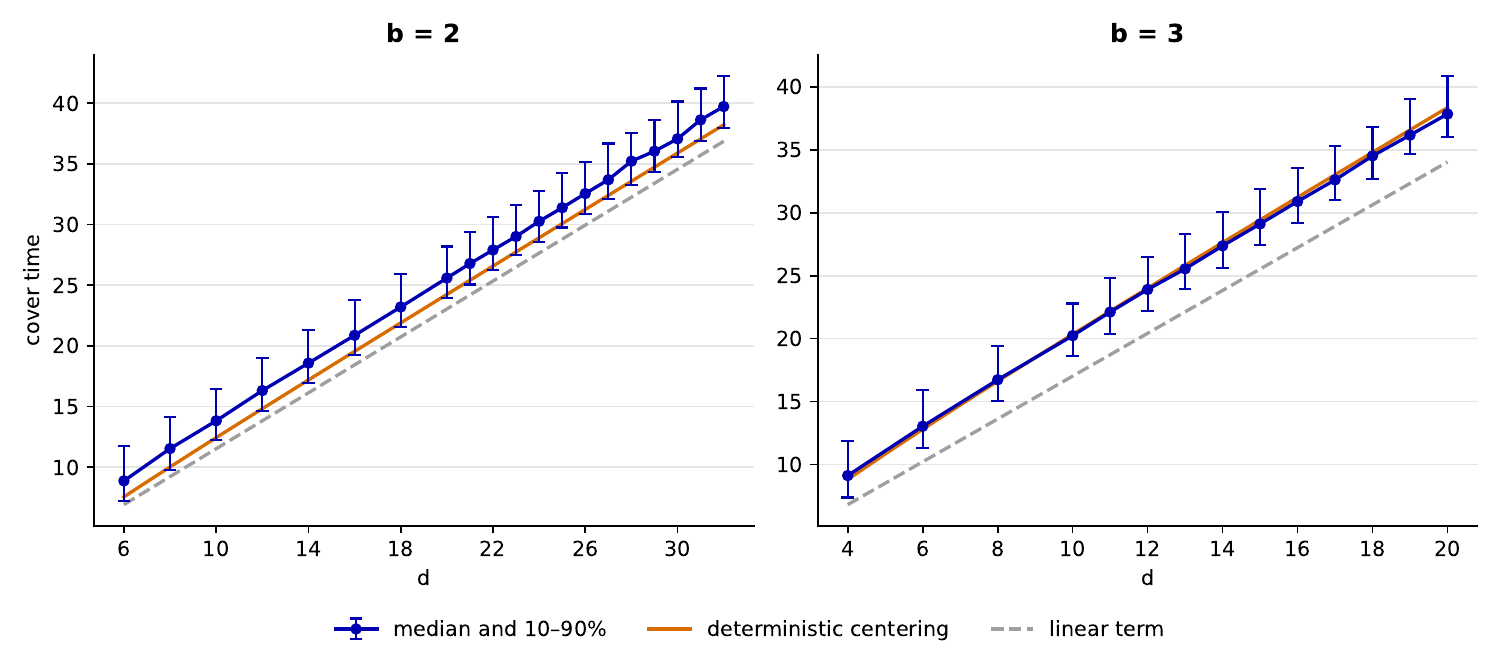}
    \caption{Cover time as a function of the dimension $d$ for $\lambda=\log 2$. Dots are empirical medians and vertical bars span the empirical $10^{\mathrm{th}}$ to $90^{\mathrm{th}}$ percentiles. Each simulated dimension uses between $100$ and $2000$ independent replicates, with fewer replicates at larger dimensions. The solid orange curves are the deterministic centerings in Theorems~\ref{thm:cover1} and~\ref{thm:cover2}; the dashed gray curves are the leading terms $x_\star d$.}
    \label{fig:simulation1}
\end{figure}


\phantomsection

\end{document}